\documentclass[a4paper,11pt]{article}
\usepackage{hyperref}

\usepackage[dvipsnames]{xcolor}
\usepackage[left=2.5cm,right=2.5cm,top=2.5cm,bottom=2.5cm]{geometry}
\usepackage[utf8]{inputenc}
\usepackage{amsmath,indentfirst,stmaryrd,qsymbols,graphicx,amsfonts,enumerate,amsthm,ulem,dsfont,cleveref,footmisc,scalerel,stackengine,xspace,amsmath}
\usepackage{fourier}
\usepackage{tikz,pgfplots}
\usetikzlibrary{shapes,snakes}
\usetikzlibrary{patterns}
\usetikzlibrary{calc}
\usepackage{comment}
\usepackage{aurical}

\DeclareMathOperator{\Leb}{Leb}

\DeclareMathOperator{\argmax}{argmax}

\def\Var{{\sf Var}}
\def \1{\mathds{1}}
\def \A#1{{\sf A}{\small (#1)}}
\def \A{{\bf A}}

\def \B{{\sf B}}

\definecolor{amber}{rgb}{1.0, 0.75, 0.0}

\def \K{{\textbf{\textit{K}}}}

\def \L{{\sf L}}

\def \Op{{\sf Op}}

\def \P{\mathbb{P}}
\def \R{\mathbb{R}}

\def \Z{\mathbb{Z}}

\def \app#1#2#3#4#5{\begin{array}{rccl} #1:&#2&\longrightarrow&#3\\ &#4&\longmapsto&#5\end{array}}

\def \bM{\begin{bmatrix}}

\def \bQ{{\bf Q}}
\def \bS{{\bf S}}

\def \bar{\overline}
\def \ba{\begin{align}}
\def \ba{{\bf a}}

\def \bc{{\bf c}}
\def \ben{\begin{eqnarray}}
\def \beq{\begin{equation}}
\def \be{\begin{eqnarray*}}

\def \bia{\begin{itemize}\compact \setcounter{d}{0}}

\def \bir{\begin{itemize}\compact \setcounter{c}{0}}
\def \bis{\begin{itemize}\compact }
\def \bi{\begin{itemize}\compact \setcounter{b}{0}}

\def \bpar#1{\left\{\begin{array}{#1} }

\def \bs{{\bf s}}
\def \bs{{\sf bs}}

\def \bt{{\bf t}}
\def \build#1#2#3{\mathrel{\mathop{\kern 0pt#1}\limits_{#2}^{#3}}}

\def \bz{{\bf z}}

\def \captionn#1{\begin{center}\begin{minipage}{15cm}\sf\caption{\small #1}\end{minipage}\end{center}}
\def \ceil#1{\lceil#1\rceil}

\def \dd#1{\frac{\partial}{\partial #1}}
\def \dd{\xrightarrow[n]{(d)}}
\def \det{{\sf det}}
\def \dis{\displaystyle}

\def \eM{\end{bmatrix}}
\def \ea{\end{align}}
\def \een{\end{eqnarray}}
\def \ee{\end{eqnarray*}}
\def \eia{\end{itemize}\vspace{-2em}~}
\def \eir{\end{itemize}\vspace{-2em}~}
\def \eis{\end{itemize}\vspace{-2em}~}
\def \ei{\end{itemize}\vspace{-2em}~}

\def \epar { \end{array}\right.}
\def \eqd{\sur{=}{(d)}}

\def \eq{\end{equation}}
\def \eref#1{(\ref{#1})}
\def \floor#1{\lfloor#1\rfloor}

\def \imp{\Rightarrow}

\def \ita{\addtocounter{d}{1}\item[(\alph{d})]} 

\def \itr{\addtocounter{c}{1}\item[($\roman{c}$)]}

\def \l{\left}

\def \ov#1{\overline{#1}}

\def \proba{\xrightarrow[n]{(proba.)}}

\def \r{\right}

\def \sous#1#2{\mathrel{\mathop{\kern 0pt#1}\limits_{#2}}}
\def \sur#1#2{\mathrel{\mathop{\kern 0pt#1}\limits^{#2}}}

\def\bma{ \begin{bmatrix}}
\def\ema{\end{bmatrix}}
\def\cro#1{\llbracket#1\rrbracket}

\newcommand{\E}{\mathbb{E}}

\newcommand{\ie}{\textit{i.e. }}

\newcommand{\equ}{\underset{n\to\infty}{\sim}}
\newcommand{\cvg}{\underset{n\to\infty}{\longrightarrow}}
\newcommand{\compact}{ \topsep0pt   \itemsep=0pt   \partopsep=0pt   \parsep=0pt}

\newcounter{b}
\newcounter{c}
\newcounter{d}
\newqsymbol{`B}{\mathcal{B}}
\newqsymbol{`E}{\mathbb{E}}
\newqsymbol{`I}{\mathbb{I}}
\newqsymbol{`N}{\mathbb{N}}
\newqsymbol{`O}{\Omega}
\newqsymbol{`P}{\mathbb{P}}
\newqsymbol{`Q}{\mathbb{Q}}
\newqsymbol{`R}{\mathbb{R}}
\newqsymbol{`W}{\mathbb{W}}
\newqsymbol{`Z}{\mathbb{Z}}
\newqsymbol{`a}{\alpha}
\newqsymbol{`e}{\varepsilon}
\newqsymbol{`o}{\omega}
\newqsymbol{`t}{\tau}
\newqsymbol{`w}{{\cal W}}
\newtheorem{lem}{Lemma}[section]

\newtheorem{conj}[lem]{Conjecture}
\newtheorem{cor}[lem]{Corollary}
\newtheorem{defi}[lem]{Definition}

\newtheorem{pro}[lem]{Proposition}
\newtheorem{rem}[lem]{Remark}
\newtheorem{theo}[lem]{Theorem}

\renewcommand{\mod}{{~\sf mod~}}

\usepackage{chngcntr,xspace}
\counterwithin*{equation}{section}

\def \A{{\cal A}}
\def \d {\textrm{d}}
\def \Area{{\sf Area}} 
\def \Var{{\sf Var}}

\def \ov#1{{\overline{#1}}}

\def \obK{\ov{{\bf K}}}
\def \obS{\ov{{\bf S}}} 
\def \eps{{\varepsilon}}
\def \CH{{\sf CH}}

\def \Bara{B{\'a}r{\'a}ny\xspace}
\def \rl{r_\lambda}
\def \fnl{\floor{n\la}}
\def \bn{\textbf{n}}
\def \bm{\textbf{m}}

\def \la{{\lambda}}
\def \D{\Delta}
\def \bc{{\bf c}}
\def \bs{{\bf s}}
\def \bt{{\bf t}}
\def \bQ{{\bf Q}}

\def \bQt{{\bf Q}^{\triangle \bullet\bullet}}
\def \bQK{{\bf Q}^{\K}}
\def \Qt{{\sf Q}^{\triangle \bullet\bullet}}

\def \QK{{\sf Q}^{\K}}
\def \L{{\sf L}}
\def \A{{\sf A}}

\def \Int{{\sf Int}}
\def \bs{{\bf s}}
\def \bt{{\bf t}}

\def \PHI#1#2{\Phi_{#1}^{#2}}
\def \PLK{\PHI\lambda\K}
\def \CCST{{\sf CCS}_\triangle^{\bullet\bullet}}

\begin{document}

\begin{center}
	\textbf{\LARGE{Conditioning random points by the number of vertices of their convex hull: the bi-pointed case}}~\\~\\~\\
\textsf{\Large
Jean-Fran\c{c}ois Marckert \& Ludovic Morin}~\\~\\~\\
{\normalsize Univ. Bordeaux, CNRS, Bordeaux INP, LaBRI, UMR 5800, F-33400 Talence, France}
\end{center}

\begin{abstract} Pick $N$ random points $U_1,\cdots,U_{N}$ independently and uniformly in a triangle ABC with area 1, and take the convex hull of the set $\{A,B,U_1,\cdots,U_{N}\}$. The boundary of this convex hull is a convex chain    $V_0=B,V_1,\cdots,$ $V_{\bn(N)}$, $V_{\bn(N)+1}=A$ with random size $\bn(N)$. The first aim of this paper is to study the asymptotic behavior of this chain, conditional on $\bn(N)=n$, when both $n$ and $m=N-n$ go to $+\infty$. We prove a phase transition: if $m=\floor{n\lambda}$ where $\lambda>0$, this chain converges in probability for the Hausdorff topology to an (explicit) hyperbola ${\cal H}_\lambda$ as $n\to+\infty$, while, if $m=o(n)$, the limit shape is a parabola. We prove that this hyperbola is solution to an optimization problem: among all concave curves ${\cal C}$ in $ABC$ (incident with $A$ and $B$), ${\cal H}_\lambda$ is the unique curve maximizing the functional ${\cal C}\mapsto {\sf Area}({\cal C})^{\lambda} \L({\cal C})^3$ where $\L({\cal C})$ is the affine perimeter of ${\cal C}$. We also give the logarithm expansion of the probability $\bQt_{n,\floor{n\la}}$, that $\bn(N)=n$ when $N=n+\floor{n\lambda}$.\par
  Take a compact convex set $\K$ with area 1 in the plane, and denote by $\bQK_{n,m}$ the probability of the event that the convex hull of $n+m$ iid uniform points in $\K$ is a polygon with $n$ vertices. 
  We provide some results and conjectures regarding the asymptotic logarithm expansion of $\bQK_{n,m}$, as well as results and conjectures concerning limit shape theorems, conditional on this event. 
  
  These results and conjectures generalize \Bara{'s} results, who treated the case $\lambda=0$.
\end{abstract}

\def \Co#1{\color{orange} #1 \color{black}}
\section{Introduction}
\paragraph*{Notation.}
For a subset $S$ of $\R^2$, we denote by $\CH(S)$ its convex hull , $\Int(S)$ its interior, and $\partial S$ its boundary. We denote by $d_H$ the Hausdorff distance between compact sets of $\R^2$. 
For any $n\geq 1$, the notation $Z[n]$ refers to the $n$-tuple $(Z_1,\cdots,Z_n)$, and by a slight abuse of language, we will write $\CH(U[N])$ for $\CH(\{U_1,\cdots,U_N\})$.

If $(a,b,c)$ and $(a',b',c')$ are two non-flat triangles in the plane, there exists a unique affine map $\psi:\R^2\to\R^2$ such that $\psi(a)=a', \psi(b)=b', \psi(c)=\psi(c')$. We denote this affine map by
\beq \label{eq:psiabc}\psi_{abc\to a'b'c'}.\eq
Finally, ``i.i.d.'' means ``independent and identically distributed''.

  \centerline{--------------------------------------}
  
Pick $n$ uniform random points $U_1,\cdots,U_n$ in a compact convex set $\K$ of area 1 of the plane, and denote by $\bQ^\K_{n,0}$, the probability that these points are in convex position, that is, form the set of vertices of a convex polygon.
\Bara{'s} \cite{barany2} proved that 
\[n^2\left(\bQ^\K_{n,0}\right)^{1/n}\cvg \frac{e^2}{4}\L\left(C^\star)\right)^3\] where $\L$ is the functional ``affine perimeter'', and $C^\star$ the convex set included in $\K$  whose affine perimeter is maximal among all convex subsets of $\K$.
The affine length of a curve (see \Cref{sec:RW}) is an important invariant in planar geometry: it is invariant by affine maps with determinant 1 (\cite{MR2680490}). The affine perimeter of a convex set $C$ is the affine length of its boundary $\partial C$. 

B\'ar\'any also proved that if one conditions the $U_i$ to be in convex position, there is a limit shape theorem for this polygon:
  \[d_H\l(\CH(U[n]), C^\star\r)\proba 0.\]    
\begin{figure}[h!]
	\centerline{\includegraphics{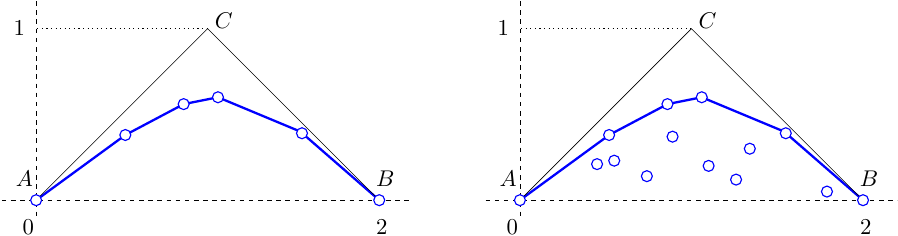}} 	
	\captionn{\label{fig:ABC} Representation of the unitary triangle $ABC$. On the first picture, 4 points $U_1,\cdots,U_4$ are represented. Together with $A$ and $B$, they form a convex chain. On the second picture, $12=4+8$ points $U_1,\cdots,U_{12}$ are taken uniformly and independently at random in $ABC$. The convex hull of $\{A,B, U_1,\cdots,U_{12}\}$ can be identified as a chain going from $A$ to $B$, passing here through four $U_i$, the other 8 being strictly below this chain.}
  \end{figure}
We provide an overview of B\'ar\'any's argument in \Cref{sec:AbC}.
These results are constructed on various works, notably a result concerning convex chains in the triangle, results that we will qualify thereafter as the \underbar{bi-pointed case}, since two vertices of the triangle are distinguished and play a special role.

\paragraph{The bi-pointed case.}
All along the paper, we will work in the unit triangle $ABC$  represented in Fig. \ref{fig:ABC}, whose vertices are 
\[A = (0,0),~~~ B= (2,0),~~~ C=(1,1),\]
and ``unit'' refers here to the area of $ABC$. As represented in Fig. \ref{fig:ABC}, when $n$ points are uniformly drawn in a triangle $abc$, we may of course wonder what the probability that these points are in convex position is, as well as question the asymptotic shape of the convex polygon under this condition. But \underbar{alternatively}, we may wonder about the probability that these points together with the two vertices $a$ and $b$ form a convex chain, as well as the limit shape of this chain under this condition. It turns out that this ``bi-pointed case'' has a tremendous importance because the general case, that is the computation of $\bQ^\K_{n,0}$ reduced somehow into ``several'' bi-pointed triangles (see more explanation in \Cref{sec:RW}).

Pick $N$ iid uniform random points in the unit triangle $ABC$, and denote by $\bn(N)$ the random variable giving the number of $U_i$ on the boundary of the polygon $\CH(\{A,B,U_1,\cdots,U_{N}\})$, and $\bm(N)$ the number of those in the interior of this convex hull, so that a.s., $\bn(N)+\bm(N)=N$.  The vertex set of the polygon $\CH(\{A,B,U_1,\cdots,U_{N}\})$ has $\bn(N)+2=n+2$ vertices, including $A$ and $B$.


For any $(n,m)$ summing to $N$, denote by
\beq\label{eq:bqt}\bQt_{n,m} = \P( (\bn(N),\bm(N))=(n,m)).\eq
 A result due to Buchta  \cite{buchta_2006} (see \Cref{theo:qdzada} to which we provide a new proof) provides an exact expansion of $\bQt_{n,m}$ as a sum indexed by $\binom{n+m-1}{m}$ terms. We will prove, and this is one of the main results of this paper, that for any $\lambda>0$,
 \[  \log\l(\bQt_{n,\fnl }\r) +  2n\log(n) \sim n\beta_\lambda, \textrm{ as } n \to +\infty\]
 for an explicit function $\beta$ (this is \Cref{theo:betala}).
 
In  \Cref{theo:lim} we provide a limit shape theorem for the convex hull of $\left\{A,B,U_1,\cdots,U_{n+\fnl}\right\}$ under $\Qt_{n,\fnl}$, where for a general $(n,m)$,  
 \ben\label{eq:Qtnm} \Qt_{n,m}:= {\cal L}\l( U[n+m]~|~ \bn(n+m)=n\r).\een
 The limit shape is an explicit hyperbola ${\cal H}_\lambda$, continuously depending on $\lambda$. \Cref{theo:mon}$(i)$ shows that under  $\Qt_{n,m(n)}$, for $m(n)=o(n)$, the limit shape is a parabola (the same one that B\'ar\'any \& al. \cite{baranybipointed} obtained for $m=0$). Hence, these results show that a phase transition occurs when we condition the set of points $U_{1},\cdots,U_{N}$ to have a linear fraction of elements below the chain. 
The internal points produce pressure that deforms the chain (we can give a formal sense to this statement, see e.g. \Cref{lem:sto}).

 The chain is pushed away from the horizontal axis, and an asymptotic stabilization occurs, when $n\to +\infty$ and the quotient $m/n$ converges.

 We also present a {\bf characterization of ${\cal H}_\la$, as the solution of an optimization problem:} among all concave curves ${\cal C}$ in $ABC$ going from $A$ to $B$,  ${\cal H}_\la$ is the curve that maximizes $\A({\cal C})^\lambda \L({\cal C})^{3}$, and $\A({\cal C})$ being the area below the curve ${\cal C}$ (this is \Cref{theo:optri}). 

  These results open the way to the characterization of the limit shape in the classical, general (not bi-pointed) case. For the moment, some of the elements in the proof remain inconclusive, but nevertheless, we present first results in this direction.

Let $\K$ be a compact convex set of $\R^2$ with non-empty interior. Fix some $\lambda>0$. Let ${\sf CCS}_\K$ be the set of compact convex sets included in $\K$, and for $C\in {\sf CCS}_\K$ set $\PLK(C)=\L(C)^{3} \A(C)^\lambda$.
We are interested in $\argmax \PLK$, the set of compact convex sets that maximize $\PLK$. A continuity argument will prove that this set is non-empty. If ${\cal C}\in\argmax \PLK$ is an optimizing curve, then each part of ${\cal C}$ between consecutive contact points ${\cal C}\cap \partial K$ is an hyperbola, and these hyperbolas (for a same optimizing curve ${\cal C}$) have same signature (see \Cref{def:signa}), meaning that up to some linear map respecting some constraints, they can be sent on ``portions'' of the same hyperbola.

We denote by $\bQK_{n,m}$ the probability that the number of vertices of $\CH(U[n+m])$ is $n$ (for the $U_i$ iid uniform in $\K)$, and denote by  $\QK_{n,m}$ the conditional law with respect to this event.
\Cref{theo:conv1} concerns the asymptotic expansion of $\bQK_{n,\fnl}$ .
We also conjecture that under $\QK_{n,\fnl}$, when $n\to+\infty$ and $\lambda$ is fixed, $\CH(U[n+\fnl])$ converges to  $\argmax \PLK$, when this set contains a unique element. Some elements supporting these conjectures will be presented in \Cref{sec:AbC}.

\subsection{Main convergence results in the bi-pointed case}
\label{sec:MCR}
The following section aims at introducing formal elements and at stating the main results in the triangular, bi-pointed case. 
Put $N$ iid uniform random points $U_1,\cdots,U_{N}$ in the unit triangle $ABC$. 
Now, take
\[{\cal C}(N):={\sf CH}\left\{A,B, U_1,\cdots,U_{N}\right\}.\] 
The boundary of ${\cal C}(N)$ can be identified with a convex chain\footnote{technically, it is rather a concave chain in such a triangle,}
\[B:=V_0,V_1,\cdots,V_{\bn},V_{{\bn}+1}=A,\]
including $A$ and $B$ and a random number ${\bf n}:={\bf n}(N)$ of  $U_i$.
For technical reasons, the vertices $V_i$, whose coordinates are $(X_i,Y_i)$ are ordered so that their abscissa are decreasing\footnote{This choice simplifies a bit a decomposition formula further in the paper, but of course, the other classical ordering is also possible. And the reader who prefer the standard order, can apply the symmetry $\bar{X}_i=2-X_i$ everywhere where $X_k$ appear.}, that is,  we have 
\[X_i > X_{i-1}.\]
As represented in Fig. \ref{fig:KV}, each triangle
\[\Delta_i:=(V_{i+1},V_{i},V_{0})\] contains \underbar{in its interior} $\Int\; \Delta_i$ a non-negative random number
\[K_i = \l|  \{U_1,\cdots,U_N\} \cap \Int\; \Delta_i \r|\]
of elements of the set $\{U_1,\cdots,U_{N}\}$. We will call the array $K[{\bf n}]=(K_1,\cdots,K_{{\bf n}})$ the ``contents sequence''. 
Almost surely,
\[\sum_{i=1}^{\bf n} K_i = N-\bn=:\bm(N),\] meaning that all the $U_i$ that are not vertices of ${\cal C}(N)$ are in the interior of one of the $\Int \Delta_k$. Almost surely too, the $(V_i)$ are vertices, i.e. extremal points, of ${\cal C}(N)$.
\begin{figure}[htbp]
	\centerline{\includegraphics{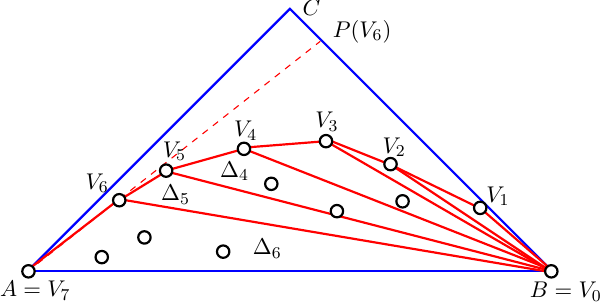}} 	
	\captionn{\label{fig:KV}On this example, $N=12$, ${\bf n}(N)=6$ and ${\bf m}(N)=6$,  $(K_6,K_5,K_4,K_3,K_2,K_1)=(3,0,2,1,0,0)$.}
  \end{figure}
 We want now to  understand the behavior of ${\cal C}(N)$  conditional on $({\bf n}(N),{\bf m}(N))$, and for this we will need to understand the law of ${\bf n}(N)$. 
 
Recall the definition of $\bQt_{n,m}$ and $\Qt_{n,m}$ given in \eref{eq:bqt} and \eref{eq:Qtnm}.
We have
\begin{theo}\label{theo:qdzada}[Buchta \cite{buchta_2006}]  For every $n,m$ such that $n\geq 1$ and $m\geq 0$,
\beq
\bQt_{n,m} = \frac{1}{(n+m)(n+m+1)}\sum_{k=0}^m 2\;\bQt_{n-1,m-k} (k+1)
\eq
(where $\bQt_{0,0}=1$, $\bQt_{0,j}=0$ if $j\geq 1$)
so that 
\beq\label{eq:qfe} \bQt_{n,m}=\sum_{k[n]\in{\sf Comp}(n,m)}  \prod_{j=1}^{n} \frac{ 2\,(1+k_j)}{ (1+j+S_{j}(k[n]))(j+{S}_j(k[n]))}\eq
where ${\sf Comp}(n,m)$ is the set of compositions of $m$ in $n$ non-negative parts (that is the set of $n$ tuples $k[n]$ summing to $m$) and \beq S_i(k[n]) = k_1+\cdots+k_i. \eq Moreover, the support of contents sequence random variable $K[n]$ is ${\sf Comp}(n,m)$, and for all $k[n]\in {\sf Comp}(n,m)$,
\beq\label{eq:QtK} \Qt_{n,m}(K[n]=k[n])=\frac{1}{\bQt_{n,m}} \prod_{i=1}^n \frac{2(1+k_i)}{\bigl(i+S_i(k[n])\bigl) \bigl(i+1+S_i(k[n])\bigl)}.\eq
\end{theo} 
We will give a new proof of this result in \Cref{sec:pqgqr}.\par
Note that if $m=0$ in \eref{eq:qfe}, we retrieve a formula due to \Bara{,} Röte, Steiger and Zhang \cite{baranybipointed}:
\begin{align}\label{eq:baranybipointed}
	\bQt_{n,0} = \frac{2^n}{n!(n+1)!}.
\end{align}

In order to provide a limit formula for $\bQt_{n,\fnl}$ and to express the limit shape theorem under $\Qt_{n,\fnl},$ we need a last \underbar{technical expedient}: consider the map
\ben\label{eq:Ps}
\app{\Psi}{[0,+\infty)}{[0,+\infty)}{r}{\Psi(r)=\dis \frac{\sinh(2r )}{2r }-1}.\een
 This map is continuous, increasing, and sends bijectively $[0,+\infty)$ onto $[0,+\infty)$.
 Hence, for any $\lambda\geq 0$, there exists a unique element $r_\lambda$ of $[0,\infty)$ such that
 \beq\label{eq:qdfs} \Psi(r_\lambda)= \lambda \iff r_\lambda=\Psi^{-1}(\lambda),
 \eq
 and $r $ is then a short notation for $\Psi^{-1}$.\par
As we will see, the implicit function $r_\lambda$ is indispensable to express our limit theorems.

One of the main results of the paper is the following asymptotic expansion of $\bQt_{n,m(n)}$ when $m(n)$ is linear:
\begin{theo}\label{theo:betala} Let $\lambda>0$, as $n\to +\infty$,
\beq n^{-1}\l(\log\l(\bQt_{n,\floor{n\lambda} }\r) +  2n\log(n)\r)   
\label{eq:jspquoi}\longrightarrow\beta_\la\eq with
\ben\label{eq:betalambda}
\beta_\la 
&=& \log\l(2e^2 \rl^2\tanh(\rl)^{-2(\lambda+1)}\r). 
\een
\end{theo} Because of the relation \eref{eq:qdfs} between $r_\la$ and $\la$, there exists numerous formulas for the same quantity, using ``more or less'' the $r_\lambda$ parametrization. For example, we also have 
\ben\label{eq:betalambdap}
\beta_\la
&=&   - \frac{\sinh(2\rl)}{\rl} \log(\tanh(\rl)) +\log  \left(2e^2 \rl^2 \right)\\
&=&\log(2e^2\rl^2)+2(\la+1)\log\l(\frac{\cosh(\rl)^2}{\rl(\la+1)}\r).
\een

\begin{figure}[!]
  \centerline{\includegraphics[width=6cm]{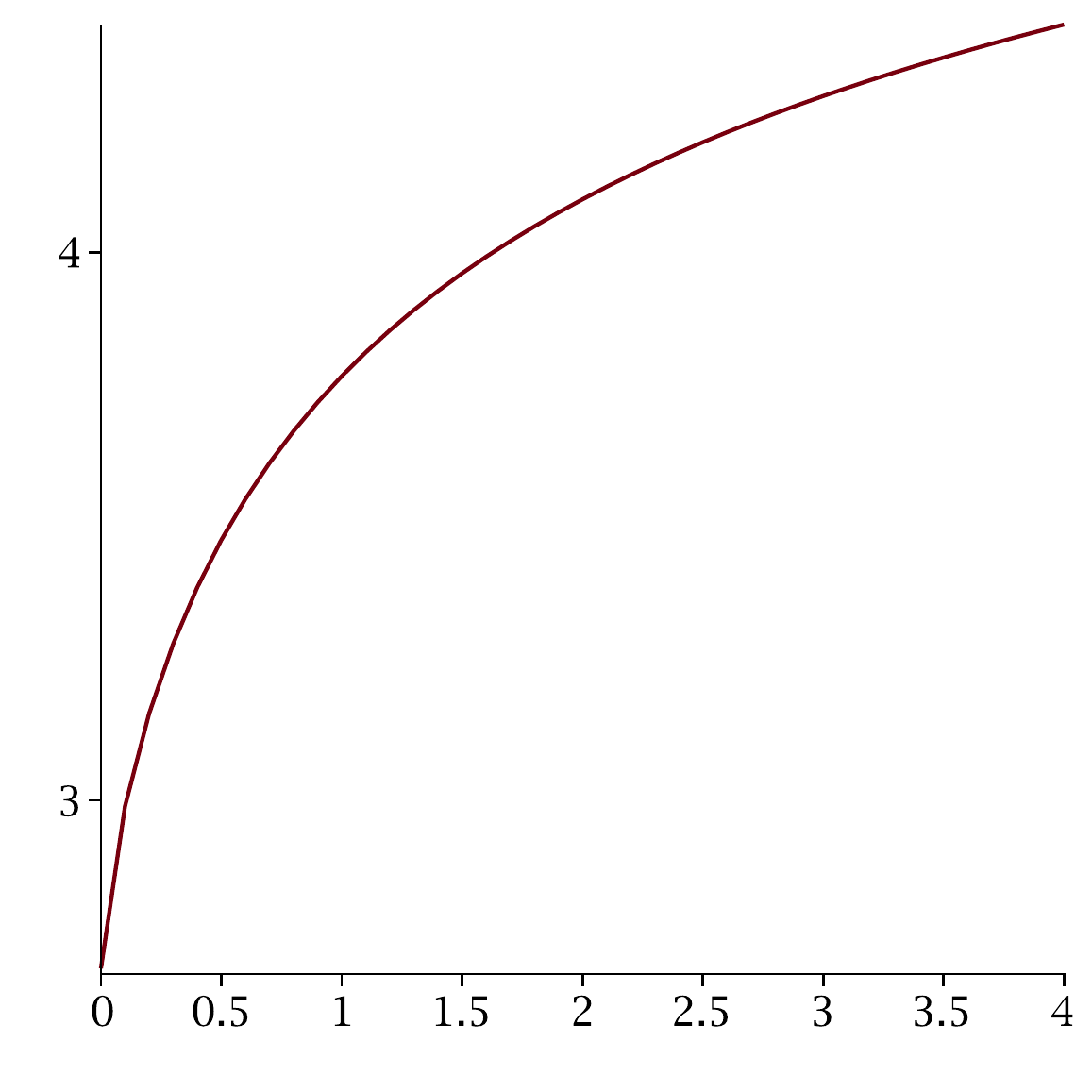}}
  \captionn{Function $\lambda\mapsto \beta_\lambda$. We have $\beta_0=2+\log(2)$.}
  \end{figure}
  To state our limit shape theorem under  $\Qt_{n,\fnl}$, we will work with the interpolated version of the sequence $(V_i,0\leq i \leq n+1)$, so that now, $(V_{x},0\leq x \leq n+1)$ is defined for all real number $x$, as well as its two coordinates $((X_x,Y_x), 0\leq x \leq n+1)$.
  We want to state a global convergence theorem, so we need a change of time to keep a compact parametrization at the limit, and we choose to encode our processes on $[0,1]$: we define
  \[X^{(n)}(x)= X_{x(n+1)}, ~~Y^{(n)}(x)= Y_{x(n+1)}, ~~\textrm{for }x\in[0,1],\]
  or in other words, the time $t$ at which the piece-wise linear map $X^{(n)}$ (resp. $Y^{(n)}$) takes the value $X_{k}$ (resp. $Y_k$) is $t= k/(n+1)$.
\begin{theo}\label{theo:lim}[Limit shape theorem] Let $\lambda>0$.  As $n\to+\infty$, under $\Qt_{n,\fnl}$,
  \[(X^{(n)},Y^{(n)}) \proba \l(X,Y\r),\]
 in $C([0,1],\R)^2$ (equipped with the uniform topology), toward the deterministic pair $\l(X,Y\r)$, where for $t\in[0,1]$,  
\ben
\label{eq:X}X(t)&=&1+\frac{\sinh(\rl(1-2t))}{\sinh(\rl)},\\
\label{eq:Y}Y(t)&=&\frac{2\cosh(\rl)}{ \sinh(\rl)^2} \sinh(\rl(1-t))\sinh(\rl t).
\een
\end{theo}
Of course, $(X,Y)$ depends on $\la$, and we could have written $(X^{(\lambda)},Y^{(\lambda)})$ instead. The next theorem allows to see that a phase transition occurs under $\Qt_{n,m(n)}$ when $m(n)$ has a linear order.
\begin{theo}[Phase transition]\label{theo:mon}
  \bir
  \itr If $m(n)=o(n)$, then  $(X^{(n)},Y^{(n)})\proba (X,Y)$ in $C([0,1],\R)^2$ where  $X_t=2-2t$ and $Y_t=2t(1-t)$, that is, the limit shape is \Bara \& al.  \cite{baranybipointed}   parabola.
  \itr If $m(n)/n \to +\infty$ then $d_H(\CH(U[n+m(n)]),ABC)\proba 0$.
  \eir
\end{theo}
In particular, $(i)$ holds when $m(n)=0$, that is $\lambda=0$, in which we recover \Bara\& al. result.
\color{black}

\begin{figure}[h!]
    \centering
    \includegraphics[height=4.2cm]{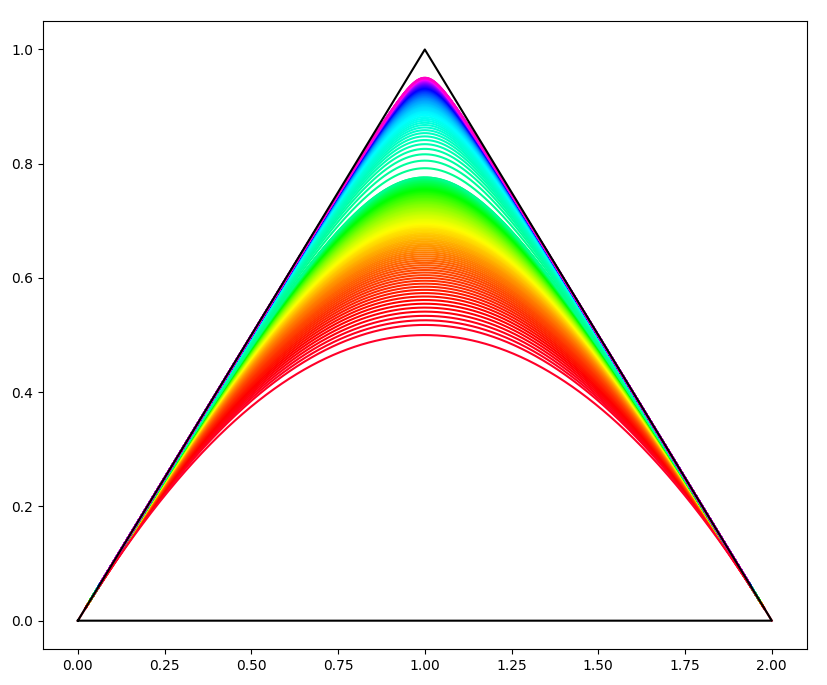}
    \includegraphics[height=4.2cm]{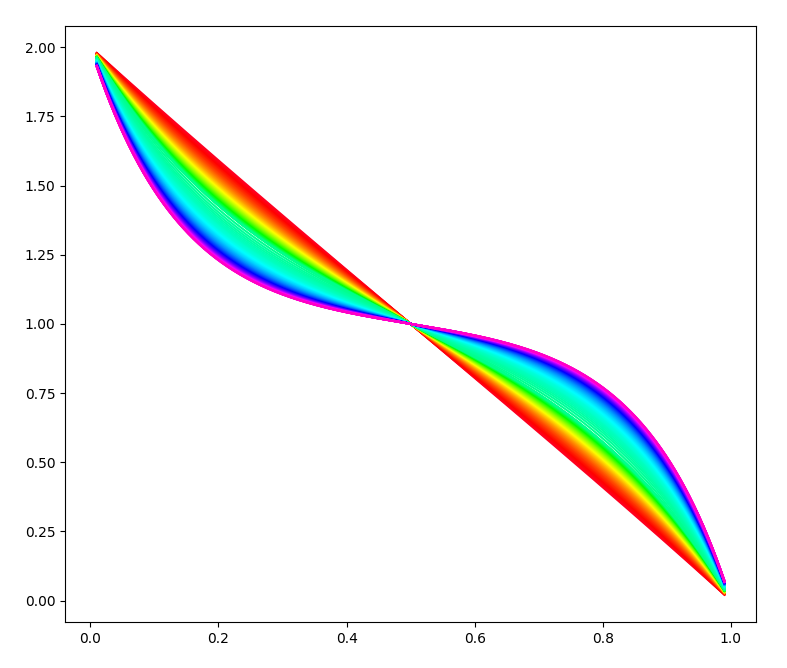}
    \includegraphics[height=4.2cm]{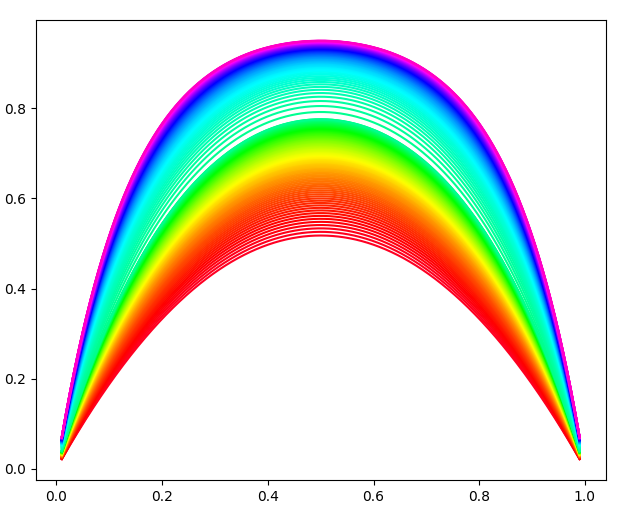}
    \captionn{Limiting curves for $\lambda$ going from $0.1$ to $5$ (by step $0.05$) and then from $5$ to $100$ (step of $1$). The curves become hot when $\lambda\to0$ and cold when $\lambda$ grows . 
 On the left the parametric curve $\{(X(t),Y(t)), t\in[0,1]\}$ (which is thus the hyperbola ${\cal H}_\la$), in the center $t\mapsto X(t)$ and at the right $t\mapsto Y(t)$. When $\lambda$ is close to zero, $X$ is essentially linear, meaning that the abscissa of the points $V_i$ are somehow uniform on $[0,1]$, while a deformation arises as $\lambda$ grows. Observe that $X(t)$ goes from 2 to 0 because of our choice of taking the $(X_i)$ decreasing.}
    \label{fig:my_label2}
  \end{figure}
Consider the curve  \[{\cal H}_\lambda:=\{ (X(t),Y(t), t \in [0,1]\}\]
and its area  $\A({\cal H}_\lambda)=\Leb\l( \CH( {\cal H}_\lambda)\r)$.

\begin{rem}
  Theorems \ref{theo:lim} and \ref{theo:mon}$(i)$ are stronger than a limit shape theorem for the Hausdorff topology, since they provide the limiting distribution of the points on the boundary of the convex hull and inside.
  Denote by $W_1,\cdots,W_{{\bf m}(N)}$ the $U_i$ that are not vertices of ${\cal C}(N)$, that is, those that lie in $\Int {\cal C}(N)$.
Under $\Qt_{n,n\lambda}$, if one wants to understand the limit of the pair of probability measures  
\[(\mu^{n}_B,\mu^n_I):=\l(\frac{1}{n} \sum_{i=1}^n \delta_{V_i}, \frac{1}{\floor{n\lambda}} \sum_{ j=1}^{\floor{n\lambda}} \delta_{W_i}\r)\]
that encodes the exact positions of the points on the boundary of the convex hull of the $U_i$, and those in the interior of the convex hull, then  \Cref{theo:lim} implies that \[(\mu^n_B,\mu_n^I)\proba (\mu_B,\mu_I)\] where $\mu_B$ is the deterministic measure on $\R^2$ which integrates (continuous and bounded) test functions $f:\R^2\to \R$ as $\int f \d\mu_B=\int_{0}^1 f(X_s,Y_s) \d s $ and 
$\int f \d\mu_I= \int_{ 0}^2 \int_0^1 f(x,y) \frac{\1_{y\leq Y(x)}}{{\Area({\cal H}_\lambda)}} \d y\d x$
(the convergence of the first marginal is a consequence of the theorem, and the convergence of the second marginal comes from the fact that the $W_i$ are, conditional on $V[n]$, independent and uniform in the interior of the convex hull of the $U_i$). The values of the deterministic processes $X$ and $Y$ show that the distribution of the abscissa of a random point taken among the $V_i$ depends (asymptotically) on $\lambda$ (see also Fig. \ref{fig:my_label2}).  
\end{rem}

\begin{theo}\label{eq:hyper} The curve ${\cal H}_\lambda$ is a hyperbolic curve that solves the equation
\[\tanh(\rl)^2(Y(t)-\tanh(\rl)^{-2})^2=\frac{1-\tanh^2(\rl)}{\tanh(\rl)^2}+(X(t)-1)^2,~~~\forall t\in[0,1].\]
The map  $Y$ can be expressed in terms of $X$. We have $X(t)=x\in[0,2]$ iff
\[  t=\frac {\rl-{\rm arcsinh} \left(\sinh \left( \rl \right) x- \sinh \left( \rl \right) \right)}{2\rl},  \]
and then
 \beq Y(x)=x-2\,{\frac {  \dis\sinh \left( \frac \rl2+\frac12\,{\rm arcsinh} \left(\sinh
 \left( \rl \right)  \left( x-1 \right) \right) \right)  ^{2}}{
\sinh \left( \rl \right)   ^{2}}},~~~ 0\leq x \leq 2.\eq
The area   and affine length of ${\cal H}_{\lambda}$ are given by
   \beq\label{eq:area}
   {\A}({\cal H}_\lambda)=  \rl {\frac {\cosh \left(\rl \right)  }{ 
\sinh \left( \rl \right)  ^{3}} \left( {\frac{\sinh \left( 2\,\rl \right)}{2\rl}}-1\right) }  =\rl \la  \frac {\cosh \left( \rl \right)  }{ 
\sinh \left( \rl \right)  ^{3}}
=\frac{\lambda}{(\lambda+1)^3}\frac{\cosh(\rl)^4}{\rl^2}    \eq
and
\beq\label{eq:Lr} 
\L_\lambda:= \L({\cal H}_\lambda)
=  2\prod_{k\geq 0} 2 \l(  \frac{\cosh(\rl/2^k)}{2(\cosh(\rl/2^k)+1)^2}\r)^{1/3} =\frac{2\rl\cosh(\rl)^{1/3}}{\sinh(\rl)} =  
\frac{2 \cosh(\rl)^{4/3}}{\lambda+1}.
\eq
For all $t\in[0,1]$,
\beq {Y''(t)X'(t)-X''(t)Y'(t)}  = \L_\lambda^3.
\eq
\end{theo}
We offered several formulas for $\L({\cal H}_\lambda)$ and $\A({\cal H}_\lambda)$. Passing from one formula to another is a simple exercise (in general). As we will see, playing with the different representations of these same quantities will be needed all along the paper.
\begin{rem} [Identification of the parabola ${\cal P}$ and the ``hyperbola'' ${\cal H}_0$]
  The parabola  ${\cal P}$ is the set of points $\{(x,f(x))~:~ x\in[0,2]\}$, where $f(x)=x(2-x)/2$. We have
  \[\A\bigl({\cal P}\bigl)=2/3,~~\L\bigl({\cal P}\bigl)=2.\]
  All the formulas we gave concerning ${\cal H}_\lambda$ are valid for the parabola ${\cal P}$, either by taking $\lambda=0$, or by taking the limit when $\lambda \to 0$. For example, observe \eref{eq:area}. 
Since $\lambda+1=\sinh(2\rl)/(2\rl)$, for $\lambda$ close to $0$, $r_\lambda$ is close to 0 too. And then, in the neighborhood of zero, since $\sinh(x)=x+x^3/6+o(x^4)$, we get 
\[\lambda+1-(1+2r_\lambda^2/3)=o(r_\lambda^3)\]
and this gives $- 2(1+\lambda)  \log(\tanh(\rl)) +2\,\log  \left( \rl \right)\underset{\lambda\to 0}{\to} 0$, so that  
\[n^{-1}\l(\log(\bQt_{n,\fnl }) +  2n\log(n)\r) \underset{\lambda\to 0}{\to} \frac{e^2}{4}\L({\cal  P})^3=2 e^2.\]  This is consistent with \eqref{eq:baranybipointed}. 
Moreover,
\[\A({\cal H}_\lambda)= \frac{\rl\cosh(\rl)}{\sinh(\rl)}\frac{\lambda}{\sinh(\rl)^2}\sim\frac{\rl\cosh(\rl)}{\sinh(\rl)}\frac{\lambda}{ \rl^2} \to 1\times 2/3 \]
and we recover the area $\A({\cal P})=\lim_{\lambda\to 0} \A({\cal H}_\lambda)$.\\
In the sequel, when we write ``for $\lambda\geq 0$'' in a formula involving ${\cal H}_\lambda$, we are implicitly identifying ${\cal H}_0$ with ${\cal P}$.
\end{rem}

\begin{rem}\label{rem:qd} There is a general simple relation between $\L({\cal H}_\lambda)$ and $\A({\cal H}_\lambda)$, valid for $\lambda>0$:
\[\lambda\; \L({\cal H}_\lambda)^3=8\;r_\lambda^2\; \A({\cal H}_\lambda) .\]
Moreover, each of the quantity $\L({\cal H}_\lambda)$, $\A({\cal H}_\lambda)$ and $\L({\cal H}_\lambda)^3/{\A}({\cal H}_\lambda)$ characterize $\lambda$ (it suffices to prove the monotonicity of each of these maps in terms on $\lambda$, which is a simple exercise). 
\end{rem}

\begin{rem} The map 
$x\mapsto 2x\frac{\cosh(x)^{1/3}}{\sinh(x)}$ is decreasing, so that for  $0<\lambda_1<\lambda_2,$ we have  $\L({\cal H}_{\lambda_1}) > \L({\cal H}_{\lambda_2})$ but  $\A({\cal H}_{\lambda_1}) <\A({\cal H}_{\lambda_2})$
\end{rem} 


\begin{figure}[hbtp]
    \centering
    \includegraphics[width=7cm]{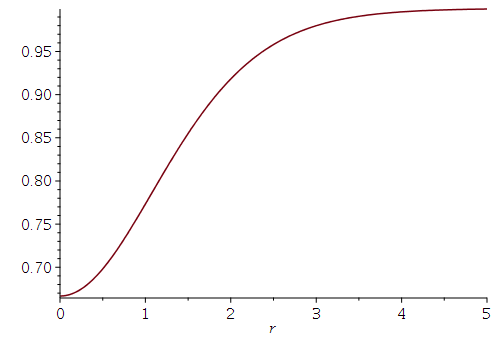}
    \captionn{$\A({\cal H}_\la)$ as a function of $\rl$}
    \label{fig:my_label}
\end{figure}

\begin{rem} The asymptotics given in \Cref{theo:betala} and  \Cref{theo:lim} concerning the exponential expansion of $\bQt_{n,m(n)}$ and the limit shape theorem when $m(n)=\fnl$, are valid if we had only assumed that $m(n)/n\to \lambda$. As we will see, the proof of  \Cref{theo:betala} and  \Cref{theo:lim} are quite long and involved, and taking $m(n)=\fnl$ simplifies a bit the exposition, but they are no obstruction to modifying each statement to take the condition $m(n)/n\to \lambda$ as long as $\lambda >0$. The case $\lambda=0$ is simpler, and is treated in \Cref{theo:mon} (but demands a specific argument).  
  \end{rem}

\subsection{Deterministic geometrical results in the triangle}
Denote by $\CCST$ the sets of Compact Convex Sets included in $ABC$ which contain $A$ and $B$. This is the bi-pointed case, and the two ``bullets'' of the notation are here for this reason.  \par
  Let ${\sf Conc}(ABC) $ be the set of  concave functions $F$ indexed by $[0,2]$ such that the curve ${\cal C}_F$ of $F$, 
  \[ {\cal C}_F=\{(x,F(x))~:~0\leq x\leq 2\}\]
  is included in the unit triangle $ABC$. This is the set of functions whose graphs are ``continuous convex chain'' from $A$ to $B$, in $ABC$.
It is easy to check that
\[ \CCST = \bigl\{ \CH({\cal C}_F) ~:~ F \in {\sf Conc}(ABC) \bigl\}.\]
Notice that in particular, the boundary of each convex set in $\CCST$ has no vertical parts.\par
We will work all along this paper with the functional
\[\app{\Phi_\lambda}{\CCST}{\R^+}{{\cal C}}{ \A({\cal C})^\lambda \L({\cal C})^{3}}.\]


When $\lambda=0$, it is known from \Bara \& al. \cite{baranybipointed} that $\argmax \Phi_0=\{{\cal P}\}=\{{\cal H}_0\}$.\par
For $\lambda>0$, the following result shows that the hyperbola too is solution to an optimization problem: 
\begin{theo}\label{theo:optri} For any $\lambda>0$,
  \[\argmax \Phi_\lambda = \{{\cal H}_\la\}.\]
\end{theo}

\begin{rem}[Value of $\Phi_\lambda({\cal H}_\lambda)$ and link with $\beta_\lambda$] 
  One has, still using the relation  $\la +1=\sinh(2\rl)/(2\rl)$, 
  \[\L({\cal H}_\lambda)^3 \A({\cal H}_\lambda)^\lambda=\frac{\cosh(r)^{4(\lambda+1)}}{r^{2\lambda}}  \frac{8}{(\lambda+1)^3} \l(\frac{\lambda}{\lambda+1}\r)^\lambda.\]
and then writing  $\tanh(\rl)=\frac{\sinh(\rl)}{\cosh(\rl)}=\frac{2\rl}{2\rl}\frac{\sinh(2\rl)}{2\cosh(\rl)^2}=\frac{({\lambda+1})r}{\cosh(\rl)^2},$
\be
\exp\l(\beta_\la\r)&=&2e^2 \rl^2 (\lambda+1)^{-2(1+\lambda)}\l(\frac{\cosh(\rl)^2}{\rl}\r)^{2(1+\lambda)}=\frac{\cosh(\rl)^{4(1+\lambda)}}{\rl^{2\lambda}} \frac{2e^2 }{(\lambda+1)^{2(1+\lambda)}}\\
&=& \L({\cal H}_\lambda)^3 \A({\cal H}_\lambda)^\lambda     \frac{2e^2 }{(\lambda+1)^{2(1+\lambda)}} \frac{(\lambda+1)^3}{8} \l(\frac{\lambda+1}{\lambda}\r)^\lambda=\Phi_\lambda({\cal H}_\la)     \frac{e^2 }{4(\lambda+1)^{\lambda-1}\lambda^\lambda}.
\ee
\end{rem}

\paragraph{Enveloping triangle (see Fig. \ref{fig:ET}).}
Let $C$ be a simple closed connected curve, meaning that there exists an injective continuous map $\gamma:[0,1]\to \R^2$, such that $C=\{\gamma(t)~:~t\in[0,1]\}$. The two points $\gamma(0)$ and $\gamma(1)$ are called extremities of $C$.
\begin{defi}  We say that a triangle $T=abc$ is the \underbar{enveloping triangle} of a  simple closed connected curve $C$, if it satisfies the three following conditions:\\
  $\bullet$ $\{a,b\}$ is the set of extremities of $C$, \\
  $\bullet$ $C$ is included in $T$, \\
  $\bullet$ $T$ is the intersection of all triangles having these first two properties.
  \end{defi}
\begin{figure}
    \centering
    \includegraphics[height=3cm ]{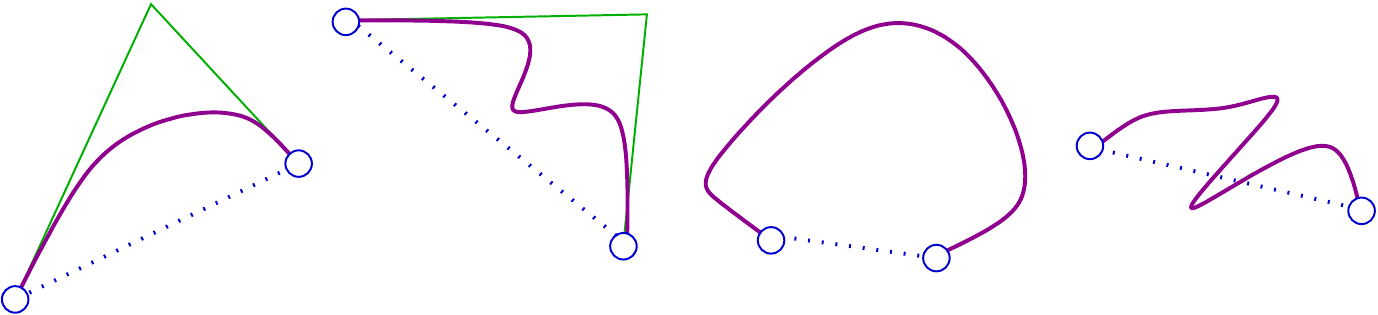}
    \captionn{Enveloping triangles of closed connected curves: the first two curves admit an enveloping triangle, while the last two do not. }
    \label{fig:ET}
\end{figure}
The following proposition, that will be at the core of the notion of ``signature'', that we  define below, relies on an important property of hyperbolas: in words, two sections of an hyperbola that have a common property, for example, enveloping triangle with same area, are in correspondence by an affine map having determinant 1:
\begin{pro}\label{pro:sdgr}  Consider an hyperbola ${\cal H}_\alpha$ (as defined in \Cref{eq:hyper}), and let $C_1$ and $C_2$ be two curves that are both closed connected subsets of ${\cal H}_\alpha$, for some $\alpha\geq 0$.  Denote by $T_1$ and $T_2$ their respective enveloping triangles.
  \bir
  \itr If $\A(T_1)=\A(T_2)$  then   $\L(C_1)=\L(C_2)$ and $\A(\CH(C_1)) = \A(\CH(C_2))$.
  \itr If $\L(C_1)=\L(C_2)$ then $\A(\CH(C_1)) = \A(\CH(C_2))$ and $\A(T_1)=\A(T_2)$
  \itr If $\A(\CH(C_1)) = \A(\CH(C_2))$ then $\L(C_1)=\L(C_2)$  and $\A(T_1)=\A(T_2)$.
  \itr Each of the following pairs:
  \bia \ita $( \Leb(T_1), \Leb(\CH( C_1)))$, \ita $(\Leb(T_1), \L(C_1))$, \ita $(\Leb(\CH(C_1)), \L(C_1))$\eia~\\
  characterizes $\alpha$.
  \eir
  \end{pro}
  The proof is postponed to Section \ref{sec:ProofSignCons}.
  \begin{defi}\label{def:signa}
   A connected subset $P$ of an hyperbola ${\cal C}$ (not reduced to a point) is said to have signature $v$ if the following condition hold:\\
    If $T$ denotes the enveloping triangle of $P$, then $\L(P )$ coincides with the affine length $\L(T'\cap {\cal H}_v)$ where $T'$ is any enveloping triangle of any connected subset of ${\cal H}_v$, such that $\A(T)=\A(T')$.
  \end{defi}

\subsection{Deterministic geometrical results in a compact convex set $\K$}

Let $\K$ be a compact convex set of $\R^2$ with non-empty interior, and let ${\sf CCS}_\K$ be the set of compact convex subsets of $\K$. For $\lambda>0$, set \[\app{\PLK}{{\sf CCS}_\K}{\R^+}{C}{\Phi_\lambda^\K(C)=\L(C)^{3} \A(C)^\lambda}.\]
We are interested in $\argmax \Phi_\lambda^\K$.

It is easy to see that any element $C\in \argmax \Phi_\lambda^\K$ must intersect $\partial \K$ at least twice, otherwise, we could design an affine map $\phi:z\mapsto w+M.z$ such that $\det(M)>1$, such that $\phi(C)=C'\subset \K$, and then, since
$\A(\phi(C))^\lambda=\det(M)^\lambda \A(C)^\lambda$, $\L(\phi(C))^3=\det(M) \L(C)^{3}$, 
it follows easily that $\Phi_\la^\K(\phi(C))= \det(M)^{\lambda+1} \Phi_\la^\K(C)\geq \Phi_\la^\K(C)$, and then $C$ could not be in $\argmax \Phi_\la^\K$.

The set of contact points ${\sf CP}_\K(C)$ for an element $C\in{\sf CCS}_\K$ with the boundary of $\K$ is
\[ {\sf CP}_\K(C)= C \cap \partial \K.\]
The set ${\sf CP}_\K(C)$ is  equal to $\partial \K$ if and only if $C=\K$, otherwise  the  connected subsets $(C_i,i\in I)$ of $\partial C$ ``between the contact points'' are the connected components of $(\partial C) \setminus \K$. The index set $I$ is either finite, or countable.

\begin{pro}\label{pro:hss} Let $\lambda>0$.
\bia
\ita The set $\argmax  \PLK$ is not empty. 
\ita If $C\neq C'$ are both in  $\argmax \PLK$, then $\A(C)\neq \A(C')$.
\ita If $\K$ has an axis of symmetry, then any element in $\argmax  \PLK$ has the same axis of symmetry.
\ita For each element $C$ in $\argmax  \PLK$, $C$ is signature-homogeneous in the sense that the connected components of $(\partial C)\setminus \K$ (that is in between contact points), are hyperbolas with same signature.
\eia
\end{pro}
 
\begin{proof}
  {\it (a).} It amounts to proving that the supremum $s:=\sup_{D\in  {\sf CCS}_\K} \PLK(C)$ is reached. 
    Take  a sequence $S_n\in{\sf CCS}_\K$ with $\PLK(S_n)\cvg \sup_{S\in{\sf CCS}_\K} \PLK(S)$. Choose a convergent subsequence $S_{n_k}$, for the Hausdorff topology, and let $C$ be its limit. Since $\A:{\sf CCS}_\K\to\R_+$ and $\L:{\sf CCS}_\K\to\R_+$ are upper-semi-continuous, so is $\PLK$. This means $\PLK(C)=\max_{S\in{\sf CCS}_\K} \PLK(S)$.
	
	{\it (b).} Assume that $C_1,C_2$ are in $\argmax \PLK$, and that $\A(C_1)=\A(C_2)$. In this case, we have also, $\L(C_1)= \L(C_2)$. Now take the Minkowski sum $ C_0=\frac1{2}{(C_1+C_2)}$, which is a convex compact subset of $\K$. Since 
	\[\L(C_0)\geq \frac1{2}(\L(C_1)+\L(C_2))=\L(C_1),\]
	and by the Brünn-Minkoswki inequality,
	\[\A(C_0)\geq \frac1{2}(\A(C_1)+\A(C_2))=\A(C_1).\]
	Note that this latter inequality becomes an equality if and only if $C_1$ and $C_2$ are equal up to a translation. So if $C_1$ and $C_2$ are two distinct convex bodies then $\PLK(C_0)>\PLK(C_1)$, which cannot be.
	
	{\it (c).} Let $C_1\in\argmax \PLK$, and assume that $C_1$ do not show the same axis of symmetry as $\K$. Define $C_2$ as $C_1$'s symmetrical convex body with respect to this axis, so that $C_2\in{\sf CCS}_\K.$  Let then $C_0=\frac1{2}{(C_1+C_2)}$. Once more, $C_0\in {\sf CCS}_\K$. Since $\L(C_1)=\L(C_2),$ and $\A(C_0)>\frac1{2}(\A(C_1)+\A(C_2))=\A(C_1)$ (clearly, $C_1$ and $C_2$ are not equal up to a translation), then $\PLK(C_0)>\PLK(C_1)$ which cannot be.
	
	{\it (d).} The proof is postponed to Section \ref{sec:qsdyjsd2}.
    
  \end{proof}

\begin{cor}\label{cor:dqfqe}
  If $\K$ is a disk with positive radius, for all $\lambda \geq 0$, $\argmax \PLK=\{\K\}$.
\end{cor}
Indeed, by \Cref{pro:hss}, when $\K$ is a disk, $\argmax  \PLK$ contains only disks.

\paragraph{Case of the regular $\kappa$-gone.}

For $\kappa\geq 3$, denote by ${\sf Reg}(\kappa)$ the regular $\kappa$-gone with vertices
\[w_j= \rho_\kappa \exp( i 2\pi j / n),~~ \textrm{ for }~~j\in\{0,\cdots,\kappa-1\}\]
where $\rho_\kappa= \l(\frac{2}{\kappa \sin(2\pi/\kappa)}\r)^{1/2}$, chosen so that  ${\sf Reg}(\kappa)$ has area 1. Denote by $m_i$ the middle of the segment $[w_i,w_{i+1\mod \kappa}]$. The area $a_\kappa$ of the polygon $\CH(\{m_0,\cdots,m_{\kappa-1}\})$ is
\be
a_\kappa&=& \frac{\kappa}{2} \sin(2\pi/\kappa) |m_0|^2
=\cos(\pi/\kappa)^2.
\ee
Each small triangle $m_im_{i+1}w_{i+1}$ has area
\beq\label{eq:bk} b_\kappa:=(1-a_\kappa)/\kappa=\sin(\pi/\kappa)^2/\kappa.\eq
Hence, for each $i$, $\psi_{ABC\to m_im_{i+1}w_{i+1}}$ (as defined in \eref{eq:psiabc})  has determinant $b_\kappa$.
Now, recall the definition of $\Psi:r\mapsto \sinh(2r)/(2r)-1$ (given in \eref{eq:Ps}) and define $f:(0,+\infty)\to \R^+$ by  
\beq\label{eq:fh} f(h)= {\cosh(h)}/{\sinh(h)^3}.\eq
\begin{pro}\label{pro:Reg}Let $v$ be the unique solution to \ben\label{eq:dqdu} \lambda= \frac{a_\kappa+\kappa r_v f(r_v)\Psi(r_v)b_\kappa}{\kappa r_v f(r_v) b_\kappa}\een
 If $\K={\sf Reg}(\kappa)$, then the unique element $C$ of $\argmax \PLK$ is constituted with $\kappa$ hyperbolas where the $i$ th hyperbola lies inside the triangle $m_im_{i+1}w_{i+1}$ and is the image by the linear map that sends ABC onto  $m_im_{i+1}w_{i+1}$, of ${\cal H}_v$.
  \end{pro}

\subsection{Convergence in a compact convex set under $\QK_{n,\fnl}$}
Recall that  $\bQK_{n,m}$ is the probability that the convex hull of $n+m$ iid uniform points in $\K$ has $n$ vertices, and  $\QK_{n,m}$ is the law of these $n+m$ points conditional on this event.

A first theorem to state here is the following one: we can deduce the limiting behavior under  $\QK_{n,\floor{n\lambda'}}$ from that under  $\QK_{n,\fnl}$ in a particular case: 
\begin{theo}\label{theo:compa}   Let $\K$ be a compact convex set of area 1, and assume that for some $\lambda\geq 0$,  $\CH(U[n+\fnl])\proba\K$ under $\QK_{n,\fnl}$, for the Hausdorff topology. In this case, for any $\lambda'\geq \lambda$, the analogue result holds for $\lambda'$ instead.
  \end{theo}

\begin{conj}\label{theo:conv1} Let $\K$ be a compact convex set of area 1, and let $\lambda>0$ such that $\argmax \PLK$ is a set with a single element, $C^\star$.
  In this case
  \beq n^{-1}\l(\log\l(\bQK_{n,\fnl})\r)+2n\log(n)\r)\to\log\l({\frac{e^2}{4}\cdot\frac{(\lambda+1)^{\lambda+1}}{\lambda^\lambda}\cdot} \A(C^\star)^{\lambda}\L(C^\star)^3\r).\eq
\end{conj}

\begin{conj}\label{theo:conv2} Under the hypothesis of \Cref{theo:conv1},
  under $\QK_{n,\fnl}$,
  \[d_H\l(\CH(U[n+\fnl]),C^\star\r) \proba 0.\]
\end{conj}
We present in \Cref{sec:AbC} some elements supporting these conjectures.

\subsection{Related works}
\label{sec:RW}
 

Let $\mathsf{U}^{(n)}_\K$ be the law of $n$ i.i.d. points drawn uniformly at random in a convex domain $\K$.
	
	\paragraph{Sylvester's problem.}
One cannot write a paper about random points in convex domains without mentioning Sylvester and his {\it four-points problem} \cite{sylvester}. This problem, which initially was ill-posed, was focusing on the convex domain $\K$ optimizing the probability that 4 i.i.d. uniform points drawn in $\K$ formed the vertices of a convex quadrilateral.
The answer was given by Blaschke in 1917 \cite{blaschke1917affine} who showed that for any $\K$,
\[\bQ^{\triangle}_{4,0}\leq \bQ^{\K}_{4,0}\leq \bQ^{\bigcirc}_{4,0},\]
where $\triangle$ and $\bigcirc$ symbolize a triangle and a disk.
Recently, Marckert and Rahmani \cite{marckert:hal-02913348} proved that we have also
\[\bQ^{\triangle}_{5,0}\leq \bQ^{\K}_{5,0}\leq \bQ^{\bigcirc}_{5,0}.\]

\paragraph{Expected number of vertices.}
In 1963, Rényi and Sulanke \cite{Renyi} gave the expected number $\bn({n})$ of vertices of the boundary of the convex hull of $n$ points $\mathsf{U}^{(n)}_\K$-distributed when $\K$ is either a polygon or domain whose frontier is $\mathcal{C}^2$ (a disk for example). In the first case we have
	\[\E\left(\bn({n})\right)= \frac{2\kappa}{3}\log(n)+{\sf Cste}+o(1),\]
where $\kappa$ is the number of sides of $\K$, and in the second one, 
\[\E\left(\bn({n})\right)\sim {\sf Cste}(\K)n^{1/3}.\]
For the first case, Groeneboom \cite{groeneboom1988} later gave a central limit theorem for $\bn({n})$ :
\[\frac{\bn({n})-\frac{2\kappa}{3}\log(n)}{\sqrt{\frac{10\kappa}{27}\log(n)}}\dd \mathcal{N}(0,1).\]

\paragraph{The interlaced quantities $\bQt_{n,m}$ and $\bQ^\K_{n,m}$.}
	
The particular cases $\bQt_{n,0}$ and $\Qt_{n,0}$ were extensively studied. In addition to the exact formula \eqref{eq:baranybipointed} of $\bQt_{n,0}$ of B\'ar\'any {\it et al.}, \cite{baranybipointed} also includes results about the asymptotic behavior of $U[n]$ under the conditional law $\Qt_{n,0}$. We already mentioned the fact that under this conditional law, $\partial \CH(U[n])$ converges for the Hausdorff distance towards the parabola $\mathcal{P}$,  but in \cite{baranybipointed} is also given 
a functional central limit theorem for the fluctuation around the limit, at the scale $1/\sqrt{n}$.

These theorems are almost the only tools used to evaluate $\bQ^\K_{n,0}$ for a general $\K$, as well as to prove a limit shape under $\QK_{n,0}$ (see \Cref{sec:AbC} for an overview). 
This is also the case for both results of Valtr \cite{valtr1996probability,Valtr1995}, who gave in 1995 when $\K$ is a parallelogram
\[\mathbf{Q}^\Box_{n,0}=\frac{1}{(n!)^2}{2n-2\choose n-1}^2,\] and in 1996 when $\K$ is a triangle  \[\mathbf{Q}^\triangle_{n,0}=\frac{2^n (3n-3)!}{(2n)!((n-1)!)^3},\] though Valtr's proof relies on some slightly different arguments.

Since the late 90s, a substantial body of work has followed. Essential work tackles limit theorems established around convex polygons formed on lattices. In this model, for a given integer $n$, a convex polygon on a lattice is a convex polygon contained in the square $[-n,n]^2$ and whose vertices have integer coordinates. 
Vershik then questioned the possibility of identifying the number and typical shape of such a convex polygon. Three coherent solutions were put forward by Bárány \cite{barany3}, Vershik \cite{vershik} and Sinai \cite{sinai} in 1994, and are described below:
In the square, the convex polygon can be naturally decomposed into 4 pieces (delimited by its extremities in the North/East/South/West directions), which delimit 4 “polygonal convex lines” between them. By studying these objects, which can initially be considered as convex chains joining $(0,0)$ to $(n, n)$ in the square $[-n,n]^2$ (apart from rotations and translations), it was shown that when $n\to+\infty,$
\begin{enumerate}
	\item the number of such polygonal lines is $\exp(3(\zeta(3)/\zeta(2))^{1/3} n^{2/3} + o(n^{2/3}))$ where $\zeta$ is Riemann's zeta function,
	\item the random number of vertices in such a chain is concentrated around $\left(\zeta(3)^2/\zeta(2)\right)^{-1/3} n^{2/3}$,
	\item the limit form of such a chain, normalized in both directions by $n$, is an arc of parabola.
\end{enumerate}

In 1997, \Bara \cite{barany1} generalized most of these results to characterize the asymptotics under $\QK_{n,0}$: He showed that the convex hull of a tuple $U[n]$ of points under $\QK_{n,0}$ converges for the Hausdorff distance to the unique convex domain $C^\star\in{\sf CCS}_\K$ maximizing the affine length of the convex domains contained in $\K$. The boundary  of $C^\star$ is composed of arcs of parabola and pieces of the boundary of $\K$. As we will be working quite a lot on the affine length, let us precise that the proofs in \cite{barany1} use properties of the affine length of curves, and relies on another fundamental work of Blaschke and Reidemeister about differential geometry \cite{blaschkedifferential}.

Based on this latter results, B\'ar\'any \cite{barany2} gave in 1999 a logarithmic equivalent of the probability $\bQ^\K_{n,0},$ valid for any convex domain $\K$ of non-empty interior of the plane : 
\begin{align}\label{eq:eqloga}
	n^2\left(\bQ^\K_{n,0}\right)^{1/n}\cvg \frac{e^2}{4}\L\left(C^\star\right)^3.
\end{align}
This formula is of course  quite reminiscent of  \eqref{eq:jspquoi} for the case $\lambda=0$.  

B\'ar\'any's logarithmic equivalent for $\bQ^\K_{n,0}$ was refined by Morin \cite{morin2024} into an actual equivalent for the case where $\K$ is a regular polygon, and he generalized it afterwards for any convex polygon \cite{morin2024bis}. Notice that a part of the proof relies on $\bQt_{n,0}$. Furthermore, a characterization of the convex set $C^\star$ as a solution of a polynomial system is also proved, as well a central limit theorem.

Things are a bit more complex for the disk. Marckert \cite{marckert2017probability} determined in 2016 a formula which enables the explicit computations of the first terms  of $\bQ^\bigcirc_{n,m}$.

Hilhorst, Calka and Schehr \cite{hilhorst:hal-00330444} gave in 2008 the first terms of the asymptotic development of $\log{\mathbf{Q}^\bigcirc_{n,0}}$, \ie
\[\log\l(\mathbf{Q}^\bigcirc_{n,0}\r)=-2n\log(n)+n\log(2\pi^2 e^2)-c_0 n^{1/5}+. ..,\]
where $c_0$ is an explicit constant, which corroborates the logarithmic equivalent of B\'ar\'any quoted above.

Let us introduce a similar object as $\bQt_{n,0}$. For any non-negative concave function $f$ defined on $[0,2]$, let $\K_f$ be the unique convex domain such that $\partial \K_f=\l([0,2]\times\{0\}\r)\cup \{(t,f(t)),t\in[0,2]\}$ : we let $\bQ^{K_f\bullet\bullet}_{n,0}$ be the probability that $n$ points $\mathsf{U}_{\K_f}^{(n)}$-distributed form a convex chain between $A=(0,0)$ and $B=(2,0)$. The same authors  
\cite{marckertmorin2024} give a recursive formula of $\bQ^{\K_f\bullet\bullet}_{n,0}$ for all concave functions $f$. In the same paper, the authors also investigate a generalization to other dimensions of this matter : in this model, we impose a concave "shape" over a convex section of an hyperplane (a floor) and draw points under this shape to ask them to be in convex position together with this floor.  

We already mentioned the fact that Buchta \cite{buchta_2006} was the first to give the probability $\bQt_{n,m}$. In \cite{buchta2}, Buchta goes further and computes  the probability $\mathbf{Q}^{\K}_{n,m}$ when $\K$ is either a square or a triangle. In \cite{GT}, Gusakova and Th\"ale, studied the probability generating function $G_N(z)=\sum_{k=1}^N \bQt_{k,N-k}z^k$ of the random variable ${\bf n}(N)$. They proved that $G_0(z)=1$, $G_1(z)=z$, and
\[(n+1) G_n(z)=  \l(\frac{2z}{n} +2(n-1)\r)G_{n-1}(z)-\frac{(n-1)(n-2)G_{n-2}(z)}{n}.\]
From this second order recursion, they deduce that these polynomials correspond (up to a multiplicative factor) to some orthogonal polynomials (for a certain inner product), and prove again a central limit theorem for $({\bf n}(N)- \E( {\bf n}(N)))/\sqrt{\Var({\bf n}(N))}$ with Berry-Esseen bounds. 
In a very recent prepublication, Besau and Th\"ale \cite{bt} studied the bi-variate generating function $F(u,z)=\sum_{N\geq 0} G_N(z)u^N$; they proved that it is a Gaussian hypergeometric function (Theo. 2.3). This analytic representation provides a new tool that allows them to get also a Berry-Esseen theorem (with same speed as in \cite{GT}), a moderate and large deviation principle for ${\bf n}(N)/\log(N)$, and, for $k$ fixed, the asymptotics of $\bQt_{k,n-k}$ as $n\to+\infty$.

\paragraph{Affine length}

The affine length of a curve is an important affine invariant (by the elements of $SL(2)$), and it is also defined  in higher dimension under the name of affine surface area; this notion goes back to Blaschke and Reidemeister \cite{blaschkedifferential}, and is still an important domain of research (see Sch\"utt and Werner \cite{MR4654484} for a recent survey) with the developments of new concepts as the $p$-affine surface area, (see Lutwak \cite{MR1378681}, Hug \cite{MR1416712} and see also Ludwig and Reitzner \cite{MR2680490,MR3614772}, for the  fact that, up to some details, the $p$-affine surface areas are somehow, the only upper semicontinuous valuations. A valuation is a function that satisfies $\Psi(P)+\Psi(Q)=\Psi(P\cup Q)+\Psi(P\cap Q)$, on the set of compact convex subsets of $\R^n$.  

\Bara \cite[Sec.2]{barany1} recalls the main properties of the affine length that we will need in this paper: If ${\cal C}$ is a compact convex set, then $\L({\cal C})=\int_{\partial {\cal C}} \kappa^{1/3}(s)ds$ where $\kappa$ is the curvature of ${\cal C}$ at $s$. If $\phi:=z\mapsto v +M.z$ is an affine map in $\R^2$, then $\L(\phi({\cal C}))=\det(M)^{1/3} \L({\cal C})$. Moreover, the affine length of a polygon is zero, and more generally, since the affine length is obtained by integration on the boundary, a section of the boundary which is a segment, does not contributes to the affine length. By \Bara \cite[(3.3)]{barany1}, $\L( (S_1+S_2)/2)\geq \frac{1}{2}(\L(S_1)+\L(S_2))$ where the first addition is the Minkowski sum.   
\color{black}

\section{Bi-pointed case}

\subsection{Finite $(n,m)$ considerations}
Recall the notation introduced in \Cref{sec:MCR}, and in particular the unit triangle $ABC$, the variables $\bn(N), \bm(M)$ and the contents sequence $K[n]$, the vertices $B,V_1,\cdots,V_n,A$ of the convex chain surrounding the convex hull of $U[n+m]$, and the coordinates of $V_i$ being $V_i=(X_i,Y_i)$.

The fact that $\bQt_{n,m}$ possesses an explicit formula, and better than that, that the distributions of the $(V_j)$ have an intelligible form is the corner stone of this work. We present the key arguments in this section. Though some of the ideas leading to the formula of $\bQt_{n,m}$ are similar to those of Buchta \cite{buchta_2006}, new points of view arise to describe the law of the $(V_i)$. 

One of the keys of these close formulas are decomposition formulas, which, in turn, rely on four interconnected facts $(a), (b), (c), (d)$ that we explicit immediately.

\paragraph{{\it (a).} Invariance by affine maps principle.}
Let us swap triangles for a short moment, and instead of the unit triangle $ABC$ pick another non-flat triangle $abc$. Take $u_1,\cdots,u_N$ iid uniform points in $abc$, and once again,  consider the convex hull of $\{a,b,u_1,\cdots,u_N\}$. Denote by $(\bn'(N),\bm'(N))$ the two variables equivalent to $(\bn(N),\bm(N))$ in this setting, and $K'_n,\cdots,K'_1$ the analogue to $K_n,\cdots,K_1$.

Since (invertible) affine maps preserve convexity and the uniform distribution, it is easy to see that the tuple of variables $(\bm(N),\bn(N),K[\bn(N)])\eqd(\bm'(N),\bn'(N),K'[\bn'(N)])$ and better than that, if one considers the affine map
$\psi_{ABC\to abc}$ that sends $(A,B,C)$ onto $(a,b,c)$ (in this order), the map $\psi_{ABC\to abc}$ sends $(U_1,\cdots,U_N)$ conditioned by $(\bn(N),\bm(N))=(n,m)$, onto some random variables $(u_1,\cdots,u_N)$, that are distributed as $N$ iid uniform random variables $(u'_1,\cdots,u'_N)$ in $abc$,  conditioned by $(\bn'(N),\bm'(N))=(n,m)$.

\paragraph{{\it (b).} The  ``rewinding'' affine map.} Fix $N\geq 0$ and a pair of integers $(n,m)$ with $n+m=N$, and $n\geq 1$. 

The \textbf{leftmost point of the convex chain} $V_n$, together with the leftmost ``content'' variable $K_n$, will play an important role: we will decompose the complete chain $V_n,\cdots,V_1$ by removing the leftmost point $V_n$ together with the triangle $\Delta_n=(V_{n+1},V_{n},V_{0})$ (whose interior thus contains $K_n$ points), and working with a smaller set of points reduced to a subtriangle $T(V_n)$ that we discuss now. \par

We will need to understand the joint law of $(V_n,K_n)$ for this, but assume for a moment that we know it, and let us discuss the decomposition.\\
Observe in Fig. \ref{fig:KV} a special case in which $n=6$ and $K_n=3$. The line $AV_n$ intersects $BC$ in $P(V_n)$ ($P(V_6)$ on the figure). The knowledge of $V_n$ is sufficient to see that apart from the $K_n$ points inside $\Delta_n$, all the $n-1+m-K_n$ other $U_i$ (their number is $6-1+6-3$ on the figure) must be in the triangle $T(V_n)=(V_n,V_0,P(V_n))$.
The set of these points has same distribution as $U'_1,\cdots,U'_{n-1-m-K_n}$ uniform independent points in $T(V_n)$ , conditioned as follows:  the convex hull of $\{V_n,V_0,U'_1,\cdots,U'_{n-1-K_n}\}$ has $(n-1)$ points on its boundary  among the $U'_i$ (besides $V_n$ and $V_0$), and $m-K_n$ points in its interior.

In words, if we take $U[n+m]$ according to $\Qt_{n,m}$,  conditional on $(V_n,K_n)$, the points $(U'_i, 1\leq i \leq n-1+m-K_n)$ that are obtained by keeping only the elements of $(U_1,\cdots,U_{n+m})$, in their initial order, that are in  $T(V_n)$, forms a $n-1+m-K_n$ tuple of iid uniform random variables in  $T(V_n)$ under the analogous distribution  $\Qt_{n-1,m-K_n}$ in  $T(V_n)$, instead of $ABC$.

The affine map
  \[\Theta^{V_n}:=\psi_{ABC\to V_nV_0P(V_n)}\] which maps our favorite unit triangle to $T(V_n)$ can be computed:
Conditional on $V_n=(x,y)$, this affine map is
\beq\label{eq:Theta}\Theta^{(x,y)}:\bma \alpha\\ \beta\ema \mapsto \bma x\\ y\ema+\Gamma^{(x,y)}.\bma \alpha\\ \beta\ema\eq with
\beq\label{eq:Ga}\Gamma^{(x,y)}=\frac{1}2\bma 2-{x} &    -2+ x+2x|T(x,y)|/y\\-y & y+2|T(x,y)|\ema\eq
where $T(x,y)$ is the triangle of vertices $(x,y)$, $B$ and $P(x,y)$ whose area is given by  \[|T(x,y)|=\frac{y(2-x-y)}{x+y}.\]

\begin{figure}[h!]
  \centerline{\includegraphics{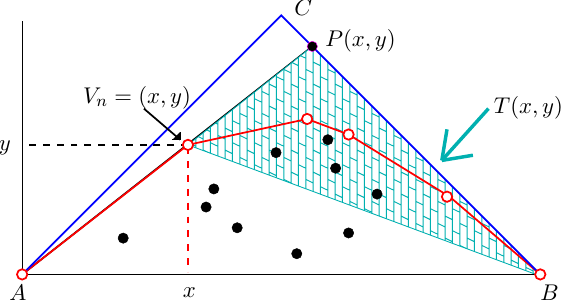} }
  \captionn{A list of 14 points, 4 of which being on the boundary of the convex hull of $\{A,B,U_1,\cdots,U_{14}\}$. In this case $(V_n,K_n)=(V_4,K_4)=((x,y),6)$. Conditional on $(V_4,K_4)$, the $U_i$ that are not in $ABV_4$ are in $V_4PB$. To get a total boundary containing 4 $U_i$, the convex hull of $V_n$, $B$, and the $U_i$ that are in the triangle $V_4PB$ must form, together with $V_n$ and $B$, a convex chain of size 5, using 3 $U_i$ in the interior of $V_4PB$, and the 4 remaining $U_i$ must be in $V_4PB$ below this partial convex chain. The 7 points in the interior of $V_4PB$, are distributed under the analogue distribution of $\Qt_{3,4}$ inside  $V_4PB$ instead of $ABC$.}
\end{figure}

	\paragraph{{\it (c).} Decomposition along the leftmost point.}

    Again, by the affine maps invariance principle, a representation of the $U'_1,\cdots,U'_{n-1+m-K_n}$ that will prove to be very rich in the sequel, consists in noticing that
    \begin{lem}\label{eq:qez}
      \beq {\cal L}_{n,m}((U'_1,\cdots,U'_{n-1+m-K_n})~|~(V_n,K_n))= {\cal L}_{n-1,m-K_n}\l(\Theta^{V_n}(U_1),\cdots,\Theta^{V_n}(U_{n-1+m-K_n})\r)\eq
      \end{lem}
      In words, in order to place $N$ points according to $\Qt_{n,m}$, it suffices, somehow, to know $(V_n,K_n)$. Once the point $V_n$ is set, we construct the triangle $(V_n,V_0,A)$ in which we draw $K_n$ points uniformly and independently. Then, we construct new points $(U_1'',\cdots,U_{N-1-K_n}'')$ \underbar{in our unit triangle $ABC$} under $\Qt_{n-1,m-K_n}$, and map these points with $\Theta^{V_n}$ into $(V_n,V_0,P(V_n))$.

With this approach, a good understanding of the generic left point random variables $(V_n,K_n)$  under $\Qt_{n,m}$ in $ABC$ is sufficient to understand all the chain $(V_1,\cdots,V_n)$ since this later will be encoded by successive affine maps. Hence, the construction of the segment $[(0,0),V_n]$ is the first step in the construction, the rest being done in $(V_n,V_0,P(V_n))$, or more exactly, is done under $\Qt_{n-1,m-K_n}$ in $ABC$ and sent in $(V_n,V_0,P(V_n))$ by $\Theta^{V_n}$.\\

Denote by $\mu^\star_{n,m,k}$ the distribution of $V_n$, the left most point, under $\Qt_{n,m}(~.~| K_n=k)$.

\paragraph{{\it (d).} Decomposition and rewinding formula.}

In order to  represent (or simulate) the distribution of $(V_n,\cdots,V_1)$ under $\Qt_{n,m}$, it suffices to first sample the random variables $(K_n,\cdots,K_1)$ (using the explicit formulae  given in \Cref{theo:qdzada}, or in \Cref{theo:rewind}, below).

\paragraph{The leftmost points in sub-problems.}
Then, take the successive \underbar{independent} leftmost points $(V_j^\star,1\leq j \leq n)$, where
\beq\label{eq:LL}{\cal L}(V_{n-i}^\star) = \mu^\star_{n-i, m-k_n-\cdots-k_{n-i+1}, k_{n-i}}\eq
which is the law of  the leftmost point $V_{n-i}$ under  $\Qt_{n-i,m-k_n-\cdots-k_{n-i+1}}(~.~|~ K_{n-i}=k_{n-i})$ (that is, in the model where there is a total of $n-i$ points on the chain and $m-k_n-\cdots-k_{n-i+1}$ below, and $k$ under the left most triangle).
Let us write $\Qt_{n,m,k}$ for the distribution of the set of points $\{U_1,\cdots,U_{n+m}\}$ conditioned by $\bn(M)=n$ and $K_n=k$.  

The affine map $\Theta^{V_a^\star}$ writes
\beq\label{eq:theta}\Theta^{V_a^\star}: \bma \alpha\\ \beta \ema \mapsto V_a^\star+ \Gamma^{V_a^\star}\bma \alpha\\ \beta\ema,\eq
for  $\Gamma^{V_a^\star}$ given in \Cref{eq:Theta}.
\begin{theo}\label{theo:rewind}[The rewinding/forwarding formula] Conditional on $(K_1,\cdots,K_n)=(k_1,\cdots,k_n)$, for the $(V^\star_j,1\leq j \leq n)$ given in \eref{eq:LL}
  \[ \l( V_j, 1\leq j \leq n\r) \eqd\l( \Phi^{V^{\star}_n}\big(\cdots \big(\Phi^{V^{\star}_j} \big(A\big)\big)\cdots\big),1\leq j \leq n\r)\]
  and then (by \eref{eq:theta})
  \[ \l(V_j, 1\leq j \leq n\r)\eqd \l( \sum_{ \ell = j}^n \Gamma^{V^{\star}_n}\big(\cdots \big(\Gamma^{V^\star_{\ell+1}}\big(V^{\star}_\ell\big)\big)\cdots\big),1\leq j \leq n\r).\]
\end{theo}
\begin{proof} This is a consequence of what has already been said, about the first point decomposition. First, to place $V_n$ one can indeed use the formula $V_n=\Phi^{V^{\star}_n}\big(  A\big)$, since $\Phi^{V^\star_n}$ sends $A$ onto $V_n$, which is indeed distributed as $V_n$. This same map sends $ABC$ inside $(V_n,V_0,P(V_n))$, and in particular, it sends $V_{n-1}^\star$ onto $V_{n-1}$ (that is more exactly, knowing $V_n$, $V_{n-1}$ is distributed as $\Phi^{V^\star_n}\big(V_{n-1}^\star$). Using that $V_{n-1}^\star=\Phi^{V_{n-1}^\star}\big( A\big)$,  knowing $V_n$, $V_{n-1}$ is distributed as $\Phi^{V^\star_n}\big(\Phi^{V_{n-1}^\star}\big( A\big)\big)$. A simple iteration allows to conclude.  
  \end{proof}
\color{black}

The nice part of this construction relies on the fact that the law of $(V_n,K_n)$ can be totally explicited:
the following theorem strengthens Buchta's result \Cref{theo:qdzada}:
\begin{theo} For $n\geq 1$, $m\geq 0$, and $(x,y)\in ABC$,
\label{theo:qdzada2}
\bia\ita For $k\in\{0,\cdots,m\}$,
\beq \label{eq:qe2}\Qt_{n,m}\l(V_n \in \d x\times \d y, K_n=k\r) = \frac{(n+m) \binom{n+m-1}{k} y^k |T(x,y)|^{n-1+m-k} \;\bQt_{n-1,m-k}   \; \d x\d y }{\bQt_{n,m}}.
\eq
and
\[\Qt_{n,m}(K_n=k)=    \frac{2(k+1)\;\bQt_{n-1,m-k} }{(n+m+1)(n+m)\;\bQt_{n,m}}. \]
\ita Conditional on $(\bn(N),\bm(N),K_n)=(n,m,k)$,
\beq (X_n,Y_n)\eqd   \l(2\;b_{k+2,n+m-k}\;(1-b_{n+m,1}/2) ,b_{n+m,1}\;b_{k+2,n+m-k}\r).\eq
where $b_{k+2,n+m-k}$ is a beta $\beta(k+2,n+m-k)$ r.v., and $b_{n+m,1}$ is an independent $\beta(n+m,1)$ r.v., and hence,
\beq\label{eq:qrh} 2-(X_n+Y_n) \eqd 2 \;(1-b_{k+2,n+m-k})= 2\; b_{n+m-k,k+2}\eq
where $b_{n+m-k,k+2}$ is $\beta(n+m-k,k+2)$-distributed.
\ita 
Take $k[n]\in {\sf Comp}(n,m)$. Conditional on $(\bn(N),\bm(N),K[n])=(n,m,k[n])$, 
\be
2-(X_j+Y_j)&\eqd& 2(1-b_{2+k_n,n+ s_n-k_n})\cdots(1-b_{2+k_j, j+ s_j-k_j})\\
&\eqd& 2(1-b_{2+k_n,n+ s_{n-1}})\cdots(1-b_{2+k_j,j+ s_{j-1}})\\
&\eqd & 2b_{n+ s_{n-1},2+k_n}\cdots b_{j+ s_{j-1},2+k_j}\ee
where for all $i$, $s_i=k_1+\cdots+k_i$, and the random variables $b_{s,t}$ are all independent, and are $\beta$-distributed with parameters $s$ and $t$.
\eia
\end{theo}
Let the reader be aware that Formula \eref{eq:qrh} is a key stone for the limit shape theorem.
 
\subsubsection{Proofs of \Cref{theo:qdzada} and \Cref{theo:qdzada2} }\label{sec:pqgqr}
\begin{figure}[h!]
  \centerline{\includegraphics {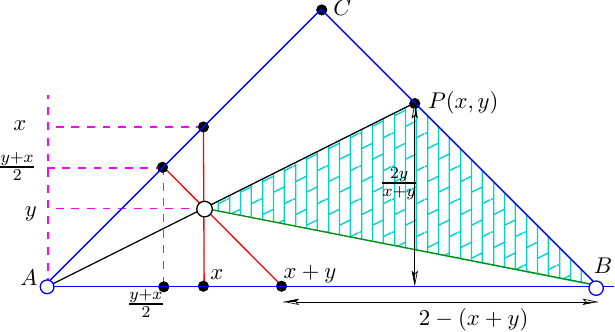}}
\caption{\label{fig:quo} Observe that for $V_n=(x,y)$, the parallel line to $CB$ passing through $V_n$ intersects the $x$-axis at distance $2-(x+y)$ of $B$. }
\end{figure}



 We first state a technical lemma that will be used several times in the sequel.
 \begin{lem}\label{lem:tech}
   For all integers $a,b,c$ such that   $a+1-c>-1$ and $b>-1$,we have
\[
\rho(a,b,c ):=\int_{0}^2 \int_{0}^{t/2} \frac{y^a (2-t)^b}{t^c} \d y \d t 
            = \frac{2^{b-c+1}\,b!\,(a-c+1)!}{( a+1)(a+b-c+2)!}.\]
            Moreover for all $a,b\geq 0$, 
            \[\int _T y^a |T(x,y)|^b  \d x \d y =\rho(a+b,b,b)=2\frac{b!(a+1)!}{(a+b+1)(a+b+2)!}=  \frac{2}{(a+b+1)(a+b+2)}\binom{a+b+1}{b}^{-1}.\]
            \end{lem}
        \begin{proof}    The proof is immediate by first integrating according to $y$. For the second statement, write $\int _T y^a |T(x,y)|^b  \d x \d y = \int_T y^a(y \frac{2-x-y}{x+y})^b \d x \d y $ and then proceed to the change of variables  $(y,x+y)=(y,t)$.\end{proof}

        For a triangle $t=abc$, and $a,b$ two chosen vertices, denote by ${\sf C}^{t,a,b}(n)$ the subset of $t^n$ formed by the $z[n]$ such that $\{a,b, z_1,\cdots,z_n\}$ are in convex position (this is the bi-pointed case in $t$).

Consider the set of points 
${\sf C}_{n,m}^{t,a,b}$, subset of $t^n\times t^m$, composed of the pairs $(z[n],w[m])$ such that $z[n]\in {\sf C}^{t,a,b}(n)$ and such that $w_1,\cdots,w_m \in {\sf CH}\{a,b,z_1,\cdots,z_n\}$. 

If $\triangle$ is the unit triangle $ABC$ we set 
\beq c_{n,m}^\triangle:=\Leb\l({\sf C}_{n,m}^{ABC,A,B}\r)= \int_{ABC^n} \Leb\l({\sf CH}\{A,B,z_1,\cdots,z_n\}\r)^m \1_{z[n]\in {\sf C}^{ABC,A,B}(n)}\prod_{i=1}^n \d z_i.
\eq
Notice that   $\sum_m c_{n-m,m}^\triangle\binom{n}{m}=1$ since
\beq\label{eq:bQt}\bQt_{k,n-k}=c_{k,n-k}^\triangle \binom{n}{k}. \eq 
{The following Lemma contains the fundamental argument on the fact that the distribution of $(V_n,K_n)$ is intelligible: we may disintegrate $\Leb\l({\sf C}_{n,m}^{ABC,A,B}\r)$ according to the position $\d x \d y$ of $V_n$ and the value $k$ of $K_n$, and this produces an integral formula in which the distribution of these quantities appear:
\begin{lem}\label{lem:dec} For every pair of integers $(n,m)$ such that $n\geq 1$ and $m\geq 0$, or $(n,m)=(0,0)$, we have
  \beq\Leb\l({\sf C}_{n,m}^{ABC,A,B}\r) = \int_{ABC} n \sum_{k=0}^m \binom{m}{k} y^k \Leb\l({\sf C}_{n-1,m-k}^{T(x,y),(x,y),B}\r)   \d x \d y. \eq
  \end{lem}
 Since
  \beq\label{eq:thdsff}\Leb\l({\sf C}_{n-1,m-k}^{T(x,y),(x,y),B}\r)=c_{n-1,m-k}^\triangle |T(x,y)|^{n+m-1-k},\eq Lemma \ref{lem:tech} leads to
\beq\label{eq:Cnm}
c_{n,m}^\triangle = 2n  \sum_{k=0}^m \binom{m}{k}  \frac{c_{n-1,m-k}^\triangle}{(n+m)(n+m+1)  \binom{n+m}{k+1}},\eq
with initial condition, for $n,m\geq 0$,
\[c_{n,0}^\triangle=t_n=\frac{2^n}{n!(n+1)!},~~ c_{1,m}^\triangle=\frac{2}{(m+1)(m+2)}.\]
\begin{proof}[Proof of \eref{lem:dec}]  The sum over $k$, is the sum on the number of points among the $w_i$ that will be in the triangle with vertices $A$, $B$ and $(x,y)$, $\binom{m}k$ accounts for the number of choices of the $k$ element in the triangle $AB(x,y)$ among the $w_1,\cdots,w_m$,  $y$ is the area of the triangle $AB(x,y)$ and then $y^k$ is the Lebesgue measure of a $k$-tuple of elements in this triangle, and then $\Leb\l({\sf C}_{n-1,m-k}^{T(x,y),(x,y),B}\r)$ simply comes from the fact that given $V_n$ and the $k$ points in $ABC$, the points $(z[n],w[n])$ is in ${\sf C}_{n,m}^{ABC,A,B}$ if and only if, all the other points are in $T(x,y)$, and are in ${\sf C}_{n-1,m-k}^{T(x,y),(x,y),B}$.\end{proof}

\subsubsection*{Proof of \Cref{theo:qdzada2}(a)}
Using \eref{eq:bQt}, the second statement in \Cref{theo:qdzada2}(a) is equivalent to \eref{eq:Cnm}.
Now, to prove the first statement, use the fact that   \Cref{lem:dec} together with \eref{eq:thdsff} and \eref{eq:bQt}, show that
\[\Qt_{n,m}(V_n \in \d x \d y,K_n=k)=\binom{n+m}{n} \frac{n  \binom{m}{k} y^k c_{n-1,m-k}^\triangle |T(x,y)|^{n+m-1-k}      \d x \d y}{\bQt_{n,m}}.\]
Replacing $c_{n-1,m-k}^\triangle$ by $\bQt_{n-1,m-k}/\binom{n-1+m-k}{n-1}$ (see \eref{eq:bQt}) ends the proof.
\subsubsection*{Proof of \Cref{theo:qdzada2}(b)}

The law of  $\Qt_{n,m}(V_n \in \d x \d y~|~K_n=k)$ given in \Cref{theo:qdzada2}(a) seems a bit complex, but it is an illusion: since this density is proportional to $y^k |T(x,y)|^{n-1+m-k}$ with full support on $ABC$, and since   by \eqref{eq:qe2}, we already know  that
\ben\label{eq:psops} \Psi_{n,m,k}(x,y) := \frac{y^k |T(x,y)|^{n-1+m-k}}{{\rho(n+m-1,n+m-1-k,n+m-1-k)}}\een
is a density on $T=ABC$, also proportional to $y^k |T(x,y)|^{n-1+m-k}$ then by uniqueness we get 
\ben \Qt_{n,m}( V_n\in \d x \d y~|~K_n =k) = \Psi_{n,m,k}(x,y) \d x \d y.\een
Now, with such a simple distribution, we may by simple routine compute the quantities of interest: 
 Set $T_n:=X_n+Y_n$.  We have
\ben (X_n,Y_n)\eqd (T_n-Y_n,Y_n) \eqd \l(2b_{k+2,n+m-k}(1-b_{n+m,1}/2) ,b_{n+m,1}b_{k+2,n+m-k}\r).\een
For any test function $f:\R^2\to\R$ continuous and bounded, we have
\be
\E_{n,m}(f(T_n,Y_n)~|~ K_n=k)= \int f(x+y,y) \Psi_{n,m,k}(x,y) \d x \d y=\int f(t,y) \Psi_{n,m,k}(t-y,y) \1_{0\leq t \leq 2 \atop{0\leq y \leq t/2}}\d t \d y \ee
and then, the density of $(T_n,Y_n)$ under $\Qt_{n,m}( \cdot|~K_n =k)$ is 
\[\Theta_{n,m,k}(t,y):= \frac{y^{n+m-1}}{\rho(n+m-1,n+m-1-k,n+m-1-k)}\l(\frac{2-t}t\r)^{n+m-1-k}\1_{0\leq t \leq 2 \atop{0\leq y \leq t/2}}.\]
By integrating $y$ from 0 to $t/2$, one finds that under $\Qt_{n,m}(.~|~K_n=k)$, $T_n$ is distributed as $2b_{k+2,n+m-k}$ where $b$ is a $\beta(k+2,n+m-k)$, and conditional on $T_n$, $Y_n=b_{n+m,1} T_n/2$, where $b_{n+m,1}$ is a $\beta(n+m,1)$ random variable.

\paragraph{Proof of \Cref{theo:qdzada2}(c):} This property is a consequence of $(b)$ and Thales theorem. We will explain how the formula works for $2-(X_{n-1}+Y_{n-1})$ which is sufficient, because the complete recursive argument is needed even in the $n-1$ case.
Notice that the affine map $\Phi^{V_n}$ that sends $ABC$ onto $V_{n}BP(V_n)$ also sends $BC$ onto $PC$, that is, it preserves the direction $BC$.
Observe now Fig. \ref{fig:calc2}  where $V_n$ and $V_{n-1}$ are represented in general position (recall that these points are sorted according to decreasing abscissa order).
Conditional on $(K_n,K_{n-1})=(k_n,k_{n-1})$, according to Lemma \ref{eq:qez}, it suffices:\\
\indent -- to place a point $V_n$ according to $\Qt_{n,m,k_{n}}$ in ABC,\\
\indent -- then, in a separate copy of $ABC$, place a point $V_{n-1}^\star$ according to $\Qt_{n-1,m-k_n,k_{n-1}}$. Then, apply the affine map $\Phi^{V_n}$ that sends $ABC$ onto $V_nBP(V_n)$, and take $V_{n-1}$ as the image of the point $V_{n-1}^\star$ by  $\Phi^{V_n}$, as shown in Fig. \ref{fig:calc2}.
It now suffices to check that for $q_{n-1}^\star=2-(X_{n-1}^\star+Y_{n-1}^\star)$, $q_n=2-(X_n+Y_n)$, $q_{n-1}=2-(X_{n-1}+Y_{n-1})$, we have
\[\frac{q_{n-1}^\star}{2}=\frac{q_{n-1}}{q_n}\]
which is immediate by Thales. Indeed, since proportions are preserved by affine maps, we get from Fig. \ref{fig:calc2} that
\[\frac{\|B-V'_{n-1}\|}{\|B-V_{n}\|}=\frac{q_{n-1}^\star}{2},\]
but on the right picture of Fig. \ref{fig:calc2},we deduce that
\[\frac{\|B-V'_{n-1}\|}{\|B-V_{n}\|}=\frac{\|B-p_{n-1}\|}{\|B-p_{n}\|}=\frac{q_{n-1}}{q_n}.\]
Therefore
\[2-(X_{n-1}+Y_{n-1})=2 \frac{q_n^\star}{2} \frac{q_{n-1}^\star}2\]
with $q_n=q_n^\star$ and by iterating this construction,
\[2-(X_{n-i}+Y_{n-i})=2 \frac{q_n^\star}{2} \frac{q_{n-1}^\star}2\cdots \frac{q_{n-i}^\star}2.\]

\begin{figure}[h!]
  \centerline{\includegraphics {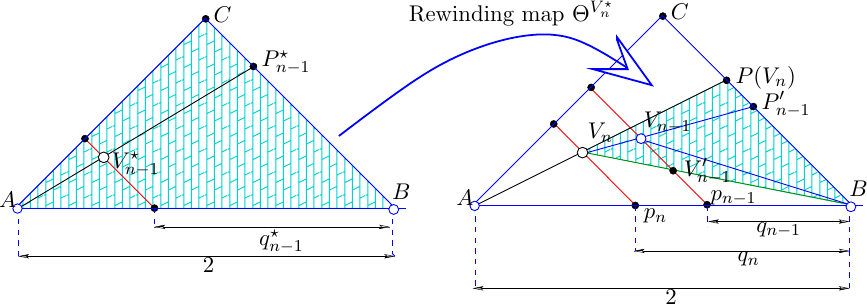}}
\caption{\label{fig:calc2} Observe that for $V_n=(x,y)$, the parallel to $CB$ passing at $V_n$ intersect the $x$-axis at distance $2-(x+y)$ of $B$. }
\end{figure}
   
\subsection{Asymptotics of $\bQ_{n,m}^{\triangle}$: proof of \Cref{theo:betala}}

To prove that $n^{-1}\l(\log\l(\bQt_{n,\floor{n\lambda} }\r) +  2n\log(n)\r)  \cvg\beta_\la$,
we will first show in \Cref{lem:LB} that the liminf of the left hand side is $\geq \beta_\lambda,$ 
and then in \Cref{lem:UB}
that   the limsup of the left hand side is $\leq \beta_\lambda$.

By the leftmost point decomposition of $V[n]$ under $\Qt_{n,m}$, it appears that a good knowledge of $K[n]$ is needed to determine the behavior of $S[n]$, where $S_i=K_1+\cdots+K_i$, this later being involved in the description of the distribution of $V[n]$. 

We will choose $m$ to be equal to $\fnl$, and we are interested in the asymptotic behavior of $S[n]$ as $n\to+\infty$.

\subsubsection{Proof of the lower bound in \Cref{theo:betala}}

The proof of the lower bound is a bit involved, even if it is much simpler than the upper bound.

The main leading idea in this section is that, in the regime $m=\fnl$, $K[n]$ under $\Qt_{n,\fnl}$ behaves ``at the first order'' as another sequence $\obK[n]$, where the $\obK_i$ are independent, but are globally conditioned to sum to $\fnl$.
The reader could be surprised by some choices taken along the proof: they may seem arbitrary, but they are not: ``lower and upper'' bounds coincide, so that these choices are arguably, optimal. 

 \paragraph{Presentation of the family of distributions  $\mu_{(v)}$.}

Since  $\sum_{k\geq 0} (1+k)v^k = 1/(1-v)^2$ for all  $v \in (0,1)$, it appears that 
\ben\label{eq:muv}\mu_{v}:=\sum_{k\geq 0} (1+k)v^k (1-v)^2\delta_k,\een
is a probability distribution on $\{0,1,2,\ldots\}$.  
If $K(v)$ is a random variable with distribution $\mu_{v}$, then
\ben \E\big(K(v)\big)= \frac{2v}{1-v}, ~~\Var\big(K(v)\big)= \frac{2v}{(1-v)^2}.\een 
\begin{lem}\label{lem:LB}[Lower bound on $\bQt_{n,\fnl}$] For any $\lambda>0$,
  \[\liminf_n n^{-1}\l(\log\l(\bQt_{n,\fnl}\r)+2n\log(n)\r) \geq \beta_\lambda.\]
\end{lem}
\begin{proof}

Take some parameters $(d_i,i\geq 1)\in (0,1)$. We rewrite \eref{eq:qfe}:
\ben\label{eq:gre}
\bQt_{n,m}&=&2^n\sum_{k[n]{\sf Comp}(n,m)} \prod_{i=1}^n \frac{(1+k_i)d_i^{k_i}(1-d_i)^2}{(i+S_i)(i+1+S_i)d_i^{k_i}(1-d_i)^2}\een
where we have written $S_i$ instead of $k_1+\cdots+k_i$ for short.
In the denominator the term $\prod_{i=1}^n d_i^{k_i}$ can be rewritten as $d_n^{S_n}\prod_{j=1}^{n-1} (d_j/d_{j+1})^{S_j}$. The numerator in the generic term of the sum \eref{eq:gre} (the factor $2^n$ apart) is the probability for an independent sequence of random variables $(\obK_i,i\geq 1)$ with respective laws $(\mu_{(d_i)},i\geq 0)$ to take the value $k[n]$. Hence, since each $k[n]$ sums to $n$, 
\ben
\label{eq:Fdn}
\bQt_{n,m} =\frac{ 2^n\,\P_{(d)}(\obS_n=m)}{ (n+m)(n+m+1)d_n^m\prod_{i=1}^n (1-d_i)^2} F_{(d),n,m}
\een
where
\ben\label{eq:Fdn2} F_{(d),n,m}= \E_{(d)}\l(\prod_{j=1}^{n-1}\frac{1}{(i+\obS_i)(i+1+\obS_i )  (d_j/d_{j+1})^{\obS_j}} \Big|\obS_n=m\r) 
\een
where here, $\E_{(d)}$ is the expectation when $\obK_i$ is $\mu_{d_i}$ distributed, and $\obS_i=\obK_1+\cdots+\obK_i$ (we will also write $\P_{(d)}$, that is with an index $(d)$, the probability of event relative to the variables $(\obK_i)$ under this distribution). Since we are completely free to choose the $(d_i)$ as we wish, we will pick the $(d_i)$ so that the denominator in the expectation of \eqref{eq:Fdn2} remains under control.\par
We focus on $F_{(d),n,m}$. Take some positive non-negative free parameters $(D_i,1\leq i\leq n-1),$ and introduce the two maps:
\be
L'(d,D)&:=& \sum_{i=1}^{n-1}\log\l(i+D_i\r)+\log(i+1+D_i)-\frac{D_i}{i+D_i}-\frac{D_i}{1+i+D_i} \\
L(d,D)&:=&\log\l(\prod_{j=1}^{n-1} (i+\obS_i)(i+1+\obS_i )  \l(\frac{d_j}{d_{j+1}}\r)^{\obS_j}\r) -L'(d,D).
\ee
Notice that $L(d,D)$ is random while $L'(d,D)$ is deterministic. Since 
\ben\l(\prod_{j=1}^{n-1}\frac{1}{(i+\obS_i)(i+1+\obS_i )  (d_j/d_{j+1})^{\obS_j}}\r)^{-1} = \exp(-L(d,D)-L'(d,D)),\een
we will proceed to a choice of $(d,D)$ in order to lower bound this random quantity by the constant $\exp(-L'(d,D))$.
Since 
\[i+\obS_i= (i+D_i) \l(1+ \frac{\obS_i-D_i}{i+D_i}\r)~~\textrm{and}~~ i+\obS_i+1= (i+D_i+1) \l(1+ \frac{\obS_i-D_i}{i+D_i+1}\r),\]
this gives 
\ben\label{eq:Ld} L(d,D)= \sum_{i=1}^{n-1} \log\l(1+ \frac{\obS_i-D_i}{i+D_i}\r)+\log\l(1+ \frac{\obS_i-D_i}{i+D_i+1}\r)+\obS_i\log\l(\frac{d_i}{d_{i+1}}\r)+\ \frac{D_i}{i+D_i}+\frac{D_i}{1+i+D_i}.\een
Assume for a moment that the sequences $(d_i,1\leq i \leq n)$ and $(D_i,1\leq i \leq n-1)$ are related by the following relation:
\ben\label{eq:equil}
\ell_i:=\log\l(\frac{d_i}{d_{i+1}}\r)=-\frac{1}{i+D_i}-\frac{1}{i+D_i+1},~~ 1\leq i \leq n-1.
\een
Then in this case, replacing $\log\l(\frac{d_i}{d_{i+1}}\r)$ by the rhs of \eref{eq:equil} in \eref{eq:Ld}, and then, using the inequality $\log(1+x)-x\leq 0$ (that holds for all $x$) to bound the two first logarithms in \eref{eq:Ld}, we obtain the crucial bound
\[L(d,D)\leq 0.\] 
Since $L'(d,D)$ is deterministic, as long as $d[n]$ and $D[n-1]$ satisfy \eref{eq:equil}, we have for all sequences $d[n]$: 
\ben\l(\prod_{j=1}^{n-1}\frac{1}{(i+\obS_i)(i+1+\obS_i )  (d_j/d_{j+1})^{\obS_j}}\r)^{-1} \geq  \exp(-L'(d,D))\een
and then
\[\exp(-L'(d,D))\leq F_{(d),n,m}=\E_{(d)}(\exp(-(L(d,D)+L'(d,D)))~|~\obS_n=m)\]
and then we get the following lower-bound $LB$:
\ben\label{eq:LB}
LB:=\exp(-L'(d,D))\frac{2^n\P_{(d)}(\obS_n=m)}{(n+m)^2 d_n^{m}\prod_{i=1}^n(1-d_i)^2}\leq \bQt_{n,m}.
\een
It remains to fix $(d[n])$ and $(D[n-1])$ such that \eref{eq:equil} holds and such that $LB$ is computable and as large as possible. We fix the sequence $(D[n-1])$ to be  
\[D_i = D_i^{(n)}= n \frac{\sinh(2\rl i / n )}{2\rl} -i, ~~ 1\leq i \leq n;\]
the sequence $(D_i)$ is positive, depends on $n$, as well as $(d_i,1\leq i \leq n)=(d_i^{(n)},1\leq i \leq n)$
and we fix
\[d_n^{(n)} = \tanh(\rl)^2.\] 
We rewrite \eref{eq:equil} so that $\log\l({d_i^{(n)}}\r)$ can be expressed with the $d_{j}^{(n)}$ with higher indices $j$:
\ben\label{eq:remonte}
\log\l({d_i}^{(n)}\r)=\log(d_{i+1}^{(n)})-\frac{1}{i+D_i^{(n)}}-\frac{1}{i+D_i^{(n)}+1},~~ 1\leq i \leq n-1,
\een
and then
\ben\label{eq:remonte2}
\log\l({d_i^{(n)}}\r)&=&\log(d_{n}^{(n)})-\sum_{j=i}^{n-1}\frac{1}{j+D_j^{(n)}}+\frac{1}{j+D_j^{(n)}+1},~~ 1\leq i \leq n-1,\\
d_i^{(n)}&=&\tanh(\rl)^2 \exp\l(- \int_{i}^n  \frac{1}{\floor{j}+D_{\floor{j}}^{(n)}}+\frac{1}{\floor{j}+D_{\floor{j}}^{(n)}+1}\,dj\r).
\een
For $t$ fixed in $[0,1]$ we have by a simple application of Lebesgue's dominated convergence theorem (see \Cref{lem:CVD} for the proof of an even stronger convergence):
\ben\label{eq:dnt}
d_{nt}^{(n)}&\cvg&\tanh(\rl)^2 \exp\l(-\int_{t}^1 \frac{4\rl}{\sinh(2\rl u)} du \r)
=\tanh(\rl t)^2.
\een
Let us come back to \eref{eq:LB} to handle every term of the expression.
 By a central local limit Lemma (see \Cref{lem:clt}, for a proof),  $\P_{(d^{(n)}}(\obS_n=\fnl)=C_\lambda/\sqrt{n}$. 
From this point on, 
\ben
-L'(d,D)&:=&\nonumber -\sum_{i=1}^{n-1}\log\l(i+D_i\r)+\log(i+1+D_i)-2+\frac{i}{i+D_i}+\frac{i+1}{i+1+D_i}\\
&=&\nonumber2(n-1)-\sum_{i=1}^{n-1}\log\l(i\r)+\log(i+1) + \log\l(1+\frac{D_i}i\r)+\log\l(1+\frac{D_i}{i+1}\r)+\frac{1}{1+\frac{D_i}i}+\frac{1}{1+\frac{D_i}{i+1}}\\
\nonumber&=&o(n)+2n -(2n\log(n)-2n)- \int_1^n \log\l(1+\frac{D_\floor{i}}{\floor{i}}\r)+\log\l(1+\frac{D_\floor{i}}{\floor{i}+1}\r)+\frac{1}{1+\frac{D_\floor{i}}{\floor{i}}}+\frac{1}{1+\frac{D_\floor{i}}{\floor{i}+1}}di\\
\label{eq:sgdqh}&=&o(n)+2n -(2n\log(n)-2n)-2n\int_{0}^{1} \log\l(\frac{\sinh(2 \rl t)}{2\rl t}\r)+\frac{2\rl t}{\sinh(2\rl t)}dt
\een
where we applied once more the Lebesgue dominated convergence theorem.
Now, still for $m=\fnl$, and still by Lebesgue, 
\be
\frac{2^n}{(n+m)^2 d_n^{m}\prod_{i=1}^n(1-d_i)^2}
&=&\exp\l(n\l(o(1)+\log(2)-2\lambda \log(\tanh(\rl))-2\int_{1}^n\log(1-d_{\floor{i}}) di\r)\r)\\
&=&\exp\l(n\l(o(1)+\log(2)-2\lambda \log(\tanh(\rl))-2n\int_0^1 \log(1-\tanh(\rl u)^2) du\r)\r).\ee

Gathering these three bounds all together, we can finally give the limiting behavior of our lower bound LB:
\[LB = \exp\l(-2n\log(n)+n (R_\lambda+o(1))\r)\]
where 
\beq R_\lambda := -2\int_0^1 \l(\frac{2\rl u}{\sinh(2\rl u)}+\log\l(\frac{\sinh(2\rl u)}{2\rl u}\r) + \log\l(1-\tanh(\rl u)^2\r)\r) du+\log(2)-\la\log(\tanh(\rl )^2)+4.\eq
To conclude, we need to prove that $R_\lambda=\beta_\lambda$, where $\beta_\lambda$ was introduced in \eref{eq:betalambda}.
Since for all $x>0$, $\log(\sinh(2x)/2)+\log(1-\tanh(x)^2)=\log(\tanh(x))$,
\[R_\lambda =  -2\int_0^1 \l(\frac{2\rl u}{\sinh(2\rl u)}+\log(\tanh(\rl u))-\log(\rl u)\r) \d u+\log(2)-\la\log(\tanh(\rl )^2)+4.\]
Now,  $\dis -2\int_0^1 -\log(\rl u) du= 2\log(\rl)-2$, and a primitive of $\dis u\mapsto \frac{2\rl u}{\sinh(2\rl u)}+\log(\tanh(\rl u))$ is $u\mapsto u\log(\tanh(\rl u))$.  Hence we get
\[R_\lambda= 2\log(\rl)-2 -2\log(\tanh(\rl ) +\log(2)-\la\log(\tanh(\rl )^2)+4\] 
which coincides with $\beta_\lambda$.
This concludes the proof of \Cref{lem:LB}.
\end{proof}

\begin{lem}\label{lem:UB}[Upper  bound on $\bQt_{n,\fnl}$] For any $\lambda>0$,
  \[ \limsup_n n^{-1}\l(\log\l(\bQt_{n,\fnl}\r)+2n\log(n)\r)\leq \beta_\la.\]
  \end{lem}
  
\begin{rem}\label{rem:qhtjsf} For two positive sequences $(a_n)$ and $(b_n)$  if there exists a sub-exponential sequence $(c_n)$ (that is, such that $\log(c_n)/n\to 0$), such that  $a_n= b_n c_n$ then $\limsup_n n^{-1}\log(a_n)= \limsup_n n^{-1}\log(b_n)$. In particular,  $\limsup_n n^{-1}\log(a_n)= \limsup_n n^{-1}\log(a_n n^K)$. Hence, if in an inequality  we ``lose'' a polynomial factor in $n$ while handling $\bQt_{n,\fnl}$, it does not jeopardize our search for the best possible bound for $\limsup_n n^{-1}\l(\log\l(\bQt_{n,\fnl}\r)+2n\log(n)\r)$.
\end{rem}

\begin{proof}The proof of this lemma is quite complex. It is divided in 9 sections.
\paragraph{1. Slicing $\bQt_{n,m}$ into sub-sums. }  For any $n\geq 1,m\geq 0, s\geq 0$, define
\[\bQ_{n,m,s}:=2^n\sum_{k[n]\in {\sf Comp}(n,m)}\prod_{i=1}^n \frac{(1+k_i)}{(s+i+S_i)(s+1+i+S_i)},\]
so that, in particular $\bQ_{n,m,0}=\bQt_{n,m}$.

The new array $(\bQ_{n,m,s})$ appears when one decomposes $\bQt_{n,m}$ as in the following example:
\[\bQt_{n,m} = \sum_{m_1=0}^{m}\bQ_{n_1,m_1,0}\;\bQ_{n-n_1,m-m_1,n_1+m_1}\]
where $n_1$ is any integer taken in $[1,n-1]$, and $m_1$ is the partial sum $S_{n_1}=k_1+\cdots+k_{n_1}$ taken by the composition $(k_1,\cdots,k_n)$. More generally, for any $K$-tuple of positive integers $(\Delta n_1,\cdots,\Delta n_K)\in \{1,2,\cdots,n\}^K$ summing to $n$, letting 
$n_j = \Delta n_1+\cdots, + \Delta n_j$ we get
\[ \bQt_{n,m} = \sum \prod_{j=1}^K \bQ_{\Delta n_j,\Delta m_j,n_{j-1}+m_{j-1}}\]
where the sum is taken on all tuples $(\Delta m_1,\cdots,\Delta m_K) \in {\sf Comp}(K,m)$, the set of $K$ tuples with non-negative entries, summing to $m$.  We used again the convention $m_j = \Delta m_1+\cdots+\Delta m_j, 1\leq j \leq K$.

%

The product $\prod_{j=1}^K \bQ_{\Delta n_j,\Delta m_j,n_{j-1}+m_{j-1}}$ is the "\textbf{sub-sum}" collecting all the 
contributions to $\bQt_{n,m}$ of the compositions $k[n] \in {\sf  Comp}(n,m)$ such that $S_{n_j}=n_j+m_j$, for all $j\in\{1,\cdots,K\}$.

\paragraph{2. A sufficient condition in terms of sub-sums.}

If for some fixed $K$, some fixed $0<n_1<\cdots <n_K=n$, we are able to compute
\[Q^\star(n_1,\cdots,n_K;m):=\max\Big\{\prod_{j=1}^K \bQ_{\Delta n_j,\Delta m_j,n_{j-1}+m_{j-1}}~:~ (\Delta m_j,1\leq j \leq K)\in {\sf Comp}(K,m)\Big\} \]
then using that $|{\sf Comp}(K,m)|\leq m^{K}$, we would have an upper bound
\[\bQt_{n,m} \leq m^{K} Q^\star(n_1,\cdots,n_K;m).\]
Since we are interested in the case $m=\fnl$, the factor $m^K$ is sub-exponential, so that \textbf{ it suffices }to prove that for any fixed $\varepsilon>0$, there exists $K$ and $n_1,\cdots,n_K$ such that 
\ben\label{eq:qsfqf} \limsup_n n^{-1}\l(\log(Q^\star\l(n_1,\cdots,n_K;\fnl)\r)+2n\log(n)\r)\leq \varepsilon+\beta_\lambda\een
to conclude the proof of the lemma. Our strategy is to prove that \eref{eq:qsfqf} holds  when one takes the limit over $K$, for $n_i=\floor{ in/K}$ for $1\leq i \leq K$. We will prove that
\ben\label{eq:qsfqf2} \limsup_K\limsup_n n^{-1}\l(\log(Q^\star(n_1,\cdots,n_K;\fnl))+2n\log(n)\r)=\beta_\lambda,\een
that is, a limit on $n$ is taken first for a finite fixed $K$, and then, we take the limit on $K$.

\paragraph{3. Elements to bound  a sub-sum.}
For each $n\geq 1,m,s\geq 0$, since the increments of the list $((s+i+S_i),i\geq 1)$ are greater or equal to 1, we have for $a\geq 1$, $b\geq 0$,
\ben\label{eq:bornebQt} \bQ_{a,b,s} \leq \frac{2^a}{(s+a+b)(s+a+b+1)} V[a,b].W[a,b,s]
\een
where $\dis W[a,b,s]=  \prod_{i=1}^{a-1}\frac{1}{(s+i)(s+i+1)}=\frac{s!(s+1)!}{(s+a-1)!(s+a)!}$ and  
for any $x\in(0,1)$,
\ben\label{eq:Vab} V[a,b]&=&\sum_{k\in {\sf Comp}(a,b)} (1+k_i)=\sum_{k\in {\sf Comp}(a,b)} \frac{\prod_{i=1}^a(1+k_i)x^{k_i}(1-x)^2}{(1-x)^{2a}x^b }=\frac{\P_{(x)}(\bS_a=b)}{(1-x)^{2a}x^b}\\
\label{eq:Vab2}&\leq& \l(1-\frac{b}{b+2a}\r)^{-2a}\l(\frac{b}{b+2a}\r)^{-b}.\een
This last bound is obtained by bounding $\P_{(x)}(\bS_a=b)$ by 1, and we know that for $a$ and $b$ large, this bound is not too bad if $x$ is taken such that $\frac{2x}{1-x}=b/a \iff x=\frac{b}{b+2a}$ which is the value for which $\E_{(x)}(\bS_a)=b$. So we took this value of $x$ to pass from \eref{eq:Vab} to \eref{eq:Vab2}. By a central local limit theorem,  $\P_{(x)}(\bS_a=b)$ can be shown to have order $1/\sqrt{a}$ under minimal hypothesis, but we won't need this thinner analysis (recall \Cref{rem:qhtjsf}).



\paragraph{4. Taking regular $(\Delta n_i)$ and reparametrization in terms of the slopes $(f_i)$.}
Fix an integer $K>1$ and set for $j\in\{0,\cdots,K\}$,
\[n_j=\floor{jn/K},~~ \Delta n_j=n_j-n_{j-1}.\]
In fact, the $n_j$ are ``functions of $n$ and $K$'', and we could have written instead $n_j(n,K)$ to exhibit this dependence. Further in the proof, we will let $K$ grow, and at this time, we will add this parameter $K$ to $n_j$, but for the moment, let us drop this extra parameter. In this part, $K$ is fixed. It is useful to notice the regime at stake when $n$ becomes large:
\[\frac{ \Delta n_j}{n/K}\xrightarrow[n\to\infty]{} 1.~~~ \]

We will optimize the $(m_i)$ progressively to evaluate the maximum sub-sum  $Q^\star(n_1,\cdots,n_k;\fnl)$ associated with the $(n_i)$. However, we won't be working with the current $(m_i)$ but with some new parameters $f[K]:=(f_i,0\leq i \leq K)$ designed to control the slopes $(\Delta m_i/ \Delta n_i)$: the $(m_i)$ and $f[K]$ are related by  
\[m_i= \l\lfloor\frac{n}{K}f_i\r\rfloor,~~~~~\Delta m_i=m_i-m_{i-1}.\]
Hence $f_K =\lambda K $ so that $m_K=m=\floor{\lambda n}$ as expected.
Viewing the $m_i$ as functions of $n$ and of $f[K]$ leads then to
\[\frac{\Delta m_i}{n/K} \xrightarrow[n\to\infty]{} f_i-f_{i-1}=\Delta f_i. \] 

Instead of analyzing $Q^\star$ in terms of $(m_i,1\leq i \leq K)$ we will optimize it in term of  $f[K]$. 

We will use \eref{eq:Vab} applied to $V(\Delta n_j, \Delta m_j)$. When $a$ and $ag$ are both positive integers,
\[V(a,ag)\leq \frac{(2+g)^{a(2+g)}}{2^{2a}g^{ag}}=:\bar{V}(a,ag).\]
In our case,
\[a= \Delta n_j, ag = \Delta m_j\] so that $\Delta m_j/\Delta n_j= \Delta f_{j} + O(1/n)$
and we easily see that for this choice
\beq\label{eq:dqfjo} {V}(\Delta n_j ,\Delta m_j)\leq \bar{V}(\Delta n_j ,\Delta m_j)\leq C \bar{V}(n/K, (n/K) \Delta f_{j})\eq
with a constant $C$ depending  possibly on $K$ and $\lambda$, but valid for all $n$ large enough\footnote{there is an additional constant because $n/K$ is may not be an integer: we have dropped the integral parts}.


\paragraph{5. Finding the regime of a given sub-sum with fixed slopes $f[K]$.}

Let us fix some non-decreasing $f[K]$ with $f_0=0$, $f_K=\la K$, and let us bound the sub-sum $\prod_{j=1}^K \bQ_{\Delta n_j,\Delta m_j,n_{j-1}+m_{j-1}}$ (for $(n_i,m_i)$ given in the section 4 of this proof):
\ben
\prod_{j=1}^K \bQ_{\Delta n_j,\Delta m_j,n_{j-1}+m_{j-1}}&\leq &\nonumber 2^n\prod_{j=1}^K V[\Delta n_j, \Delta m_j].W[\Delta n_j, \Delta m_j,n_{j-1}+m_{j-1}]\\
&\leq& \nonumber 2^n
\l(\prod_{j=1}^K  \frac{(2+\D f_j)^{(\D f_j+2)n/K}}{(\D f_j)^{\D f_j n/K}2^{2n/K}}\r)
\l(\prod_{j=1}^K \frac{(n_{j-1}+m_{j-1})!(1+n_{j-1}+m_{j-1})!}{(-1+n_{j}+m_{j-1})!(n_{j}+m_{j-1})!}\r) \\
\label{eq:b2}&= & C^K 2^{-n} \l(\prod_{j=1}^K  \frac{(2+\D f_j)^{(\D f_j+2)n/K}}{(\D f_j)^{\D f_j n/K} }\r) \l(\prod_{j=1}^K \frac{\l((j-1+f_{j-1})\frac{n}K\r)!^2 }{\l((j+f_{j-1})\frac{n}K\r)!^2}\r)\\
\label{eq:b1}&=&  O(C^Kn^{P(K)})\exp\l(n (-\log (2)+A(f[K]) +B(f[K],n) ) \r)\een
where we went from  \eqref{eq:b2} to \eqref{eq:b1} using \eref{eq:dqfjo}, Stirling's approximation $\log(n!)\simeq n\log(n)-n$, and where we defined
\ben
\label{eq:AfK}A(f[K]) &=& \sum_{j=1}^K \frac{(\D f_j+2)\log(2+\D f_j)-\D f_j\log(\D f_j)}K,\\
\nonumber B(f[K],n) &=& \frac{2}{K}\sum_{j=1}^K(j-1+f_{j-1})\l(\log\l(\frac{n}K(j-1+f_{j-1})\r)-1\r)-(j+f_{j-1})\l(\log\l(\frac{n}{K}( j+f_{j-1})\r)-1\r).
\een
Here, the term $n^{P(K)}$ widely gathers the subexponential terms of \eqref{eq:b2} forgotten when we used Stirling's approximation.
Hence, $B(f(K),n)$ is the result of the approximation of the factorial terms.\\
Since $\sum_{j=1}^K ({j-1}+f_{j-1})- ({j}+f_{j-1})=-K$, by expanding the last two logarithms $\log(n\alpha)=\log(n)+\log(\alpha)$ we get
\beq\label{eq:b3} B(f[K],n) =  -2\log(n)+ C{}(f[K])\eq  (in which we notice the appearance of the $-2\log n$ term) with 
\beq\label{eq:BB}  C{}(f[K])=2+  2\sum_{j=1}^K \frac{j-1+f_{j-1}}{K}\log\l(\frac{j-1+f_{j-1}}K\r)- \frac{j+f_{j-1}}K\log\l( \frac{j+f_{j-1}}{K}\r).\eq

Notice that both $A(f[K])$ and $C(f[K])$ do not depend on $n$ anymore, but they still depend on $K$.

Again, there is a polynomial (in $n$) number of such $(m_1,\cdots,m_K)$ so that it suffices to prove that for any $`e>0$, there exists $K$ such that
\[n^{-1}\l(\log(Q_{n,\fnl})+2n\log n\r)\leq `e+n^{-1}\l(\log\l( \max_{m[K]} \prod_{j=1}^K \bQ_{\Delta n_j,\Delta m_j,n_{j-1}+m_{j-1}}\r)+2n\log n\r)\]
and then by plugging in there \eref{eq:b1}, \eref{eq:AfK},\eref{eq:b3}, \eref{eq:BB},  it suffices to prove that for any $`e>0$, for some $K$,
\beq \max_{f[K]}\left( -\log(2)+2+A(f[K])+C(f[K])\right) \leq \beta_\la+\epsilon\eq
where the maximum is taken on all non-decreasing sequence $f[K]$ such that $f(0)=0$, $f_K=\lambda K$.

\paragraph{6. Reduction to convex sequences $f[K]$.}

Notice that, though $A(f[K])$ depends on the multi-set  $\big\{\big\{\Delta f_j,j\leq K\big\}\big\}$, it does not depend on the order of the increments, while $C(f[K])$ actually does. We will need to prove the following claim: the optimizing sequence $f[K]$ is \textbf{convex}, that is, the optimizing sequence of increments $(\Delta f_j)$ is non-decreasing.
To prove the claim it suffices to observe the effect on $C(f[K])$ when we swap two consecutive increments.
If one is given two coinciding lists of positive increments $\delta[K]=(\delta_i,1\leq i \leq K)$ and  $\delta'[K]$, if we swap two elements for a given $j\leq K-1$ as $(\delta_j,\delta_{j+1})=(a,b)=(\delta_{j+1}',\delta_{j}')$, then in this case the corresponding lists $\ell[K]$ and $\ell'[K]$ with increments respectively the $(\delta_i)$ and $(\delta_{i}')$ coincide except for the single index $i=j$, so that letting $s=j+\ell_{j-1}\geq 1$ we have
\be
C(\ell[K])-C(\ell'[K])&=&\frac{2}{K}\Bigl( (s+a-1)\log(s+a-1)-(s+a)\log(s+a)\\
&&  ~~~~~~  - (s+b-1)\log(s+b-1)-(s+b)\log(s+b)\Bigl)
\ee
which is positive if $a<b$. This argument proves that if one is given a sequence $(\Delta f_i)$ in ascending order, then $C(f[K])$ is larger than any $C(f'[K])$ if the multi-set of increments of $f'[K]$ coincides with that of $f[K]$: it means that the optimizing sequence $f[K]$ is convex.

\paragraph{7. Letting $K$ go to $+\infty$ and working in functional spaces.}

We have no idea how to compute \linebreak $\max_{f[K]} A(f[K])+C(f[K])$, but any reader who is able to do this computation can skip the end of the proof. For those who stay here, our method relies on taking the limit on $K$.

We need to enrich a bit the notation: we write $f_i^{(K)}$ for the $i^{th}$ entry of what we called $f[K]$ so far, in order to make visible the dependence on $K$. We only work with convex sequences $f^{(K)}[K]$.

We add a bar, that is we write $\bar{f}^{(K)}[K]$ for a sequence that maximizes $A(f[K])+C(f[K])$. We do not assume uniqueness, but only choose one of the optimizing sequences. 

We now embed $f^{(K)}[K]$ in the set of function from $[0,K]$ to $\R^+$: we denote by $f^{(K)}$ the function on the interval $[0,K]$ obtained by linear interpolation of the sequence $f^{(K)}_0,\cdots,f^{(K)}_K$. This function ends at $K\lambda$. We introduce its normalized version
\[ F^{(K)}(t) = K^{-1}\, f_{K t}^{(K)},~~~~ t\in[0,1]\]
which ends at $\lambda$; we should have written $F(f^{(K)},t)$ instead to make the dependence on $f^{(K)}$ apparent. Instead, we write simply $F^{(K)}$ for a generic convex function made out of a generic convex function $f^{(K)}$ (still conditioned to end at $K\lambda$ and to be nonnegative), and write  $\bar{F}^{(K)}$ for the convex function associated with the chosen optimizing sequence $\bar f^{(K)}$. 
Define also the function  $t\mapsto \bar\Delta f^{(K)}(t)$ defined on $[0,K]$ that interpolates $\Delta f^{(K)}_0,\cdots,\Delta f^{(K)}_K$, and its normalized version
\[DF^{(K)}(t) = \Delta f_{K t}^{(K)},~~~~ t\in[0,1].\]
Notice that the factor $K^{-1}$ is lacking, but this is the right normalization to have $F^{(K)}(t+1/K)-F^{(K)}(t)$ well approximated by $(1/K)DF^{(K)}(t)$.

We do not have exactly  $F^{(K)}(t)=\int_0^t DF^{(K)}(u)\d u$, even if $t$ is a multiple of $1/K$, because of the interpolation of the increments, but it is easy to see that if $F^{(K)}$ converges uniformly to some convex function $F$, and the $DF^{(K)}$ are computed as the discrete increments of $F$ (that is, for $i\in\{1,\cdots,K\}$,  $DF^{(K)}(i/K)=(F^{(K)}(i/K)-F^{(K)}(i-1/K))/(1/K)$), then $\sup_t |F^{(K)}(t)-\int_0^t DF^{(K)}(u)\d u|\xrightarrow[K\to+\infty]{} 0$.

Now, consider the sequence of optimizing pairs $(\bar F^{(K)},  \Delta\bar F^{(K)})$ indexed by $K$. For each $K$, $\bar F^{(K)}$  is convex and bounded by $\lambda$, and $\Delta\bar F^{(K)}$ is non-negative, non-decreasing.

By a standard compactness argument  and the Cantor diagonalization procedure it is possible to find a sub-sequence $(j_k)$, such that the extracted sub-sequence  $(\bar F^{(j_k)}, \Delta\bar F^{(j_k)})$ has the following properties:\\
\indent -- the sequence $\bar F^{(j_k)}$ converges point-wise on $[0,1]$, and even uniformly on any compact included in $[0,1)$. The limit $F$ is possibly discontinuous at 1.\\
\indent -- the sequence $\Delta\bar F^{(j_k)}$ converges point-wise almost everywhere toward an increasing function $DF$ on $[0,1)$ (as a non-decreasing function, $DF$ may be discontinuous only on a countable set, possibly infinite at 1).

Since $F$ is convex and non-decreasing, it is also continuous in $[0,1)$, differentiable almost everywhere, and $F(u)=\int_{0}^u F'(v)dv$. By the comment above, $F'=DF$ outside a Lebesgue null set.

\paragraph{8. Passage to the limit on $A(\bar f[K])$ and $C(\bar f[K])$}
 
Consider again the formula defining  $A(\bar f[K])$ and $C(\bar f[K])$ (given in \eref{eq:AfK} and \eref{eq:BB}).
We still write under the extracted sequence $(j_k)$ that has the properties discussed in the previous point. 
The map
\[ x \mapsto (x+2)\log(x+2)-x\log(x)\] is continuous and bounded on any compact of $\R$.
Hence
\be A(\bar{f}^{(j_k)}[j_k])&=&  \int_{0}^\lambda \big(2+\Delta \bar f^{(j_k)}_{\ceil{j_kt}}\big)\log\big(2+\Delta \bar f^{(j_k)}_{\ceil{j_kt}}\big)- \Delta \bar f^{(j_k)}_{\ceil{j_k t}}\log\big(\Delta \bar\lambda^{(j_k)}_{\ceil{j_k t}}\big)dt\\
&\underset{k\to+\infty}{\longrightarrow}&\int_{0}^\lambda \big(2+\bar F'(t)\big)\log\big(2+\bar F'(t)\big)- \bar F'(t)\log\big(\bar F'(t)\big)dt.
\ee

The sum in \eref{eq:BB} involves quantities as $\psi(x)=x\log(x)$, so that the sum  is 
\ben
C(\bar f^{(j_k)}[j_k])
&=& 2 +  2\int_{0}^{j_k} \psi\l( \frac{\ceil{j}+\bar f^{(j_k)}_{\ceil{j}}}{j_k}\r)-\psi\l( \frac{1}{j_k}+\frac{\ceil{j}+\bar f^{(j_k)}_{\ceil{j}}}{j_k}\r)dj\\
\label{eq:qsdtsd}&=&  2 +  2 j_k\int_{0}^1 \psi\l( \frac{\ceil{j_kt}}{j_k}+\frac{\bar f_{\ceil{j_kt}}}{j_k}\r)-\psi\l( \frac{1}{j_k}+\frac{\ceil{j_kt}}{j_k}+\frac{\bar f^{(j_k)}_{\ceil{j_kt}}}{j_k} \r)dt\\
\label{eq:Lt}&\underset{k\to+\infty}{\longrightarrow}& 2 +  2\int_{0}^{1} \l(-1-\log(t +\bar F(t))\r) dt =  -2\int_{0}^{1} \log\l(t +\bar F(t)\r) dt.
\een
Proving this last convergence is an exercise whose main lines are given in footnote\footnote{In the footnote, we use $K$ instead of $j_k$ for readability sake. Use that  $\psi(x)-\psi(x+1/K)\sim \frac{-1}K(1+\log(x))+0(K^{-2}/x)$.
The rhs of \eref{eq:qsdtsd} rewrites
\[2+2K \l( \int_{0}^{1/K^{\alpha}} + \int_{1/K^{\alpha}}^1\r) \psi\l( \frac{\ceil{Kt}+ \bar f^{(K)}_{\ceil{Kt}}}{K}\r)-\psi\l( \frac{1}K+\frac{\ceil{Kt}+ \bar f^{(K)} _{\ceil{Kt}}}{K} \r)dt.\]
So, either $ \bar f^{(K)}_{\ceil{Kt}}$ is far from 0 when $t$ is near zero, and \eqref{eq:Lt} holds easily, or it is locally Lipschitz and equals 0 at 0. In this case, since $|x\log(x)|\leq x^{9/10}$ in a neighborhood of zero so that $|2K\int_{0}^{1/K^{\alpha}}   \psi( \frac{\ceil{Kt}+ \bar f^{(K)}_{\ceil{Kt}}}{K})-\psi( \frac{1}K+\frac{\ceil{Kt}+ \bar f^{(K)} _{\ceil{Kt}}}K) dt|\leq C (K/K^{\alpha}) (1/K^{\alpha})^{9/10}\to 0$ if $\alpha =3/5$, for example. The second integral
\[2K\int_{1/K^{\alpha}}^1  \psi\l( \frac{\ceil{Kt}+ \bar f^{(K)} _{\ceil{Kt}}}{K}\r)-\psi\l( \frac{1}K+\frac{\ceil{Kt}+ \bar f^{(K)} _{\ceil{Kt}}}{K} \r)\d t =2K\int_{1/K^{\alpha}}^1 \l(-1-\log(\frac{\ceil{Kt}+ \bar f^{(K)} _{\ceil{Kt}}}{K})\r)\l(\frac{1}K+ O(\frac{\ceil{Kt}+ \bar f^{(K)} _{\ceil{Kt}}}K)\frac1{K^2}\r) \d t\]
and this converges to the rhs of \eqref{eq:Lt} by Lebesgue dominated convergence.}.

\paragraph{9. Final optimization argument.}

Gathering now the terms of $A+C$ in a functional $\theta$, it suffices to prove that any convex function $M^\star$ such that $  M^\star\geq 0$ and $M^\star(1)=\lambda$ satisfies  $\theta(M^\star)\leq \beta_\lambda$, where $\theta$ is defined as
\ben\label{eq:dqg}\theta(M):=  -\log (2)+  \int_{0}^1  (M'(t)+2)\log(2+M'(t))-M'(t)\log(M'(t))  -2 \log(t+M(t))\d t. \een

We will use an alternative to the Euler-Lagrange optimization method, and prove a little bit more, namely, in addition to   $\theta(M^\star)\leq \beta_\lambda$, the uniqueness of $\argmax \theta$, in the set of functions we study.
We proceed to the change of variable  $\ell(t)=M(t)+t$, so that $\ell(1)=1+\lambda$, $\ell(0)=1$. We need to maximize
\[\ell \mapsto \Op(\ell):= \int_{0}^1 \bar{L}(t,\ell(t),\ell'(t))\d t\]
with 
\[\bar{L}(t,\ell(t),\ell'(t))=(\ell'(t)+1)\log(1+\ell'(t))-(\ell'(t)-1)\log(\ell'(t)-1)  -2 \log(\ell(t))\]
(or more formally, $\bar{L}(\alpha,\beta,\gamma)=(\gamma+1)\log(\gamma+1)-(\gamma-1)\log(\gamma-1)-2\log(\beta)$).
By Beltrami identity (the Euler-Lagrange special case $\partial\bar{L}/(\partial \alpha)=0$), the functions $\ell$ that maximize  $\ell \mapsto  \Op(\ell)$ satisfy, for a constant $C$,
\ben\label{eq:dqsdq} \bar{L}-\ell' \frac{\partial \bar{L}}{\partial \ell'}=C\een
which is equivalent to 
\[C+2\log(\ell)=\log((\ell'+1)(\ell'-1)).\]
Solving this standard type of ODE leads to
\[\ell(x)= s \sinh(e^{C/2}(x-c))\exp(-C/2)\]
for a second constant $c$, and a ``sign'' $s \in\{-1,1\}$. 
Since $\ell(0)=0$, $\ell(x)\geq 0$, $\ell(1)=1+\lambda$, we must have $c=0$, $s=1$. The only solution is $\ell^\star(x)=\sinh(\alpha x)/\alpha$ with $\alpha$ such that $\sinh(\alpha)/\alpha=1+\lambda$ (so that $\exp(C/2)=2\rl$, i.e. $C=2\log(2\rl)$), and we recover $\alpha=2\rl$, so that 
\beq\label{eq:Ms23} M^\star(t)=\ell^\star(t)-t=-t+\frac{\sinh(2\rl t)}{2\rl}\eq
is the maximizing function, and \eref{eq:dqsdq} holds with this function.

There remains to give an argument to prove that the function $\ell^\star$ we found is indeed a maximum for $\Op$ (it could be a singular point in which we have cancellation of the derivative \eref{eq:dqsdq} without being a maximum). Here it is: the set of functions we are working in is $\bigl\{f : f(0)=0, f(1)=1+\lambda, \textrm{ $f$ differentiable, $f$  non-decreasing }\bigl\}$. 
This set of functions is included in $E:=\bigl\{ \ell^\star+ f ~:~ f(0)=f(1)=0, \textrm{ $f$ differentiable }\bigl\}$. To prove that $\Op$ possesses indeed a local maximum at $\ell^\star$, it suffices to show that 
\ben\label{eq:dkpqd} \Op(\ell^\star)- \frac{ \Op(\ell^\star + zf) +\Op(\ell^\star - zf)}2> 0\een
for $z$ sufficiently small, and $f\neq 0$ in $E$. Since \eref{eq:dqsdq} is satisfied for $\ell = \ell^\star$, it suffices to extract the second order term in $z^2$  in \eref{eq:dkpqd} (the first order coefficient in $z$ cancels, which is a consequence of \eref{eq:dqsdq}). A few computations out of \eqref{eq:dkpqd} later, it appears that we need to prove that, for $f\neq 0$ in $E$,  we have
\ben\label{eq:qd} \int_0^1 \frac{f'(t)^2-4\rl^2f(t)^2}{\sinh(2\rl t)^2}dt>0.\een
To prove this point, we adapted a well-known proof of Wirtinger's inequality (which, in general, states that $\int_0^a y(x)^2dx\leq \frac{a^2}{4\pi}\int_0^a y'(x)^2dx$).
 
Since the denominator in \eref{eq:qd} cancels at $t=0$ only, and since $f(0)=0$, there are two cases:\\
{\it (a)} the case $f'(0)\neq 0$. In this case  the lhs of \eref{eq:qd} is $+\infty$.\\
{\it (b)} the case $f'(0)=0$. In this case, observe that
\ben\label{eq:dqdqsdrf2}
\frac{f'(t)^2}{\sinh(2\rl t)^2}- \frac{4\rl^2f(t)^2}{\sinh(2\rl t)^2}=\frac{1}{\sinh(2\rl t)^2}( f'(t)-2\rl \tanh(2\rl t) f(t))^2-\frac{\partial}{\partial t} \l(f(t)^2\frac{4\rl}{\sinh(4\rl t)} \r)
.\een
Observe that  on the r.h.s of \eqref{eq:dqdqsdrf2}, the second term integrates to 0 on $[0,1]$ since $f(0)=f(1)=0$.
The integral of the first term is positive (technically, if $f'(t)=2\rl \tanh(2\rl t) f(t)$ for all $t\in[0,1]$, we have  $f(t)=c\cosh(2\rl t)$ so that this second terms integrates to zero. But since we have already treated the case $f'(0)\neq 0$ in $(a)$, we can rule out this subcase here). 
Then, by integrating \eref{eq:dqdqsdrf2} on $[0,1]$, \eref{eq:dkpqd} holds indeed, which concludes the proof that $\ell^\star$ maximizes $\Op$.\\
Now let us rewrite $\theta(M^\star)$ as 
\begin{align*}
   \theta(M^\star)= -\log(2)+\int_0^1 M^\star{}'(t)\log\l(\frac{2+M^\star{}'(t)}{M^\star{}'(t)}\r)dt + 2\int_0^1 \log\l(\frac{2+M^\star{}'(t)}{t+M^\star(t)}\r)\d t.
\end{align*}
Using \eref{eq:Ms}, 
we obtain 
${(2+M^\star{}'(t))}/{M^\star{}'(t)} = 1/\tanh(\rl t)^2$ as well as $\frac{(2+M^\star{}'(t))}{(t+M^\star(t))}=2\rl/\tanh(\rl t)$, so that
\beq\label{eq:Ms}
    \theta(M^\star) 
    =2\log(\rl)+\log(2)-2\int_0^1 \log\l(\tanh(\rl t)\r)\cosh(2\rl t)\d t.
\eq
A primitive of $x\mapsto \log\l(\tanh(x)\r)\cosh(2x) $ is $x\mapsto \dis \frac1{2}\sinh(2x)\log(\tanh(x))-x$ so that 
\[-2\int_0^1 \log\l(\tanh(\rl t)\r)\cosh(2\rl t)dt=-2\frac{\sinh(2\rl)}{2\rl}\log(\tanh(\rl))+2\]
and then, $\theta(M^\star)= \beta_\lambda$. 
\end{proof}

\subsection{Functional convergence of $(X^{(n)},Y^{(n)})$: proof of \Cref{theo:lim}}
In order to prove \Cref{theo:lim}, we will use \Cref{theo:qdzada2} that gives us the distribution of $2-(X_i+Y_i)$ when $K[n]$ is known. To use this theorem, we first need a concentration result for the sequence $S[n]$, where as usual
$S_i=K_1+\cdots+K_i$.

\subsubsection{A limit shape theorem for ${S}[n]$  under $\bQt_{n,\fnl}$}

Recall that $M^\star$ is the function $M^\star(t)=\ell^\star(t)-t=-t+\frac{\sinh(2\rl t)}{2\rl}$ appearing in \eref{eq:Ms23}. 
\begin{theo}\label{theo:podqgh}For all $\lambda>0$, all $\eta>0$,
  \ben \Qt_{n,\floor{n\lambda}}\Big( \sup_{t\in[0,1]} \l|n^{-1}\bS_{nt}-M^\star_t\r|\geq \eta\Big)\xrightarrow[n\to+\infty]{} 0.
  \een
  that is, $n^{-1}\bS_{n.}$ converges uniformly to $M^\star$,  in probability.
\end{theo}
In a nutshell, this theorem is a consequence of the proof of \Cref{lem:UB}: the optimization of the sequence $(f[K])$ leads us to a unique optimization function $M^\star$. The set of paths $S$ ``staying far away'' from this optimization sequence have a weight exponentially smaller than $\bQt_{n,\fnl}$.\par
However, this informal argument is not sufficient:  even if $M^\star$ maximizes alone the map $\theta$ defined in \eref{eq:dqg}, the set of functions $\{f:\|f-M^\star\|_\infty\geq \eta\}$ contains an infinite number of functions, so that the fact that  \Cref{lem:UB} implies \Cref{theo:podqgh} needs a proof.
\begin{proof} We will extract a finite number of events $E_0,\cdots,E_{N_\eta}$ all of which with exponentially small probabilities, such that
  \beq \l\{\sup_{t\in[0,1]} \l|n^{-1}\bS_{nt}-M_t^\star\r|\geq \eta\r\} \subset \bigcup_{i=0}^{N_\eta} E_{i}.\eq
\begin{figure}[htbp]
  \centering
    \includegraphics{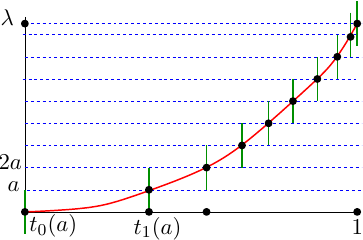}
    \captionn{\label{fig:courbe}}
    
\end{figure}
The sequence $(S_k)$ is a.s. non-decreasing, and $M^\star$ is continuous, increasing (from $[0,1]$ to $[0,\lambda]$), and then invertible.

For $a>0$ fixed, denote by $t_i(a)$ the abscissa $t$ at which $M_t^\star=ia$. Hence $t_0(a)=0<t_1(a) < \cdots <t_{\floor{\lambda/a}}(a)\leq 1$.
If $\lambda/a$ is an integer take $N(a)= \lambda/a$, otherwise, set $N(a)= \floor{\lambda/a}+1$, so that, in all cases $t_{N(a)}(a)=1$ and also, $M_{t_{i+1}}^\star(a)-M_{t_i}^\star(a)\leq a$ for all $i$.

Consider for $0\leq i \leq N(a)-1$, the vertical segments
\[I_i(a)=\{t_i(a)\} \times [M_{t_i(a)}^\star-a,M_{t_i(a)}^\star+a]=\{t_i(a)\} \times [ia-a,ia+a]\] as drawn on Fig. \ref{fig:courbe}. It is easy to see that any non-decreasing function $f$ whose graph intersect all the $I_i(a)$ must satisfy $\|f-M\|_{\infty}\leq 2a$.
The set $\{f:\|f-M^\star\|\geq \eta\}$ is then included in $\cup_{0\leq i\leq N(\eta/2)} E_i$ where
\[E_i=\l\{ f: |f(t_i(\eta/2))-M_{t_i(\eta/2)}^\star|\geq \eta/2\r\}.\]
Now, assume that we start again the optimization problem $\theta$, with the additional constraint that $f$ is in $E_i$: the proof of \Cref{lem:UB} can be done again.
Now, we maximize $\theta$ (for $\theta$ defined in \eref{eq:dqg}) on $E_i$. Since $\theta$ has a unique maximum, we then have
\beq\label{eq:tehd} m_{i,\eta}:=\max_{f\in E_i} \theta(f) < \theta(M^\star).\eq
Now, for a subdivision in $K$ sections (as done in  \Cref{lem:UB}), for $K$ large enough, the weight of the paths $n^{-1}S_{n}$ in $E_i$ (and weighted by $\prod (2(1+k_i))/((i+S_i)(i+1+S_i))$ as usual) is smaller than $\exp\l(-2n\log(n) +nm_{i,\eta}+o_K(1)n\r)$ while the number of paths considered is at most $O( (n\lambda)^K)$. Hence, we see that
\beq \label{eq:setgef}\Qt_{n,\fnl}(n^{-1}\bS_{n} \in E_i) \leq \exp\l(n \bigl(\theta(m_{i,\eta})-\theta(M^\star)\bigl)+o_K(1)n\r) O( n^K),\eq
where $o_K(1)$ is a sequence indexed by $K$ that converges to zero as $K\to+\infty$. By \eref{eq:tehd}, the rhs in \eref{eq:setgef}, for a fixed $i$, goes to zero  exponentially fast in $n$. This allows to take the union bound on all $i\in\cro{0,N(\eta/2)}$ to complete the proof.
\end{proof}

\subsubsection{Proof of \Cref{theo:lim}}

There are two main ingredients in the proof: 
\Cref{theo:qdzada2}$(c)$ that tells us the law of ``the diagonal'' $2-(X_i+Y_i)$ given the $(\bS_i)$ is simple, and \Cref{theo:podqgh} that tells us that asymptotically, $n^{-1}(\bS_{nt})$ is simple, since it is concentrated around $M^{\star}_t$.
By the Skhorokhod representation theorem, there exists a probability space $(\Omega,{\cal A},\P)$, on which there exists a sequence of random variables $\bS^{(n)}[n]=(\bS_i^{(n)},0\leq i \leq n)$, for all $n$, such that for a fixed $n$, $\bS^{(n)}[n]$ is distributed  as $\bS[n]$ under $\Qt_{n,\floor{n\lambda}}$, and  such that ${\bf s}^{(n)}$, defined by
\[ {\bf s}^{n}(t)=n^{-1}\bS_{nt}^{(n)}, 0\leq t\leq 1\]
converges to $M^\star$ almost surely, for the uniform topology in $C[0,1]$.
Recall that $V_{n+1}=A,V_n,\cdots,V_1,V_0=B$ is the convex chain we are interested in, so that it has length $n+1$.  
\begin{lem}\label{lem:jolicv}Take $\lambda>0$. On $(\Omega,{\cal A},\P)$, for all $t$ fixed in $[0,1]$,
\be
\E\bigl(2-(X_{\floor{(n+1)t}}+Y_{\floor{(n+1)t}})~|~\bS^{(n)}[n]\bigl) &\xrightarrow[n\to+\infty]{a.s.}& Q(t)= 2\;\frac{\sinh^2(\rl t)}{\sinh^2(\rl)},  \\
\Var\bigl(2-(X_{\floor{(n+1)t}}+Y_{\floor{(n+1)t}})~|~\bS^{(n)}[n]\bigl) &\xrightarrow[n\to+\infty]{a.s.}0.
\ee
As a consequence,   $2-\l(X_{\floor{(n+1)t}}+Y_{\floor{(n+1)t}}\r)\proba Q(t)$  and by symmetry,   $X_{\floor{(n+1)t}}-Y_{\floor{(n+1)t}} \proba Q(1-t)$.
\end{lem}

\begin{rem} As represented on the Fig. below, the symmetry ${\cal S}:\bma x \\y\ema\to\bma 2-x \\y\ema$  (with respect to the vertical line passing at $C$) preserves the distribution of $U[n+m]$ under $\Qt_{n,m}$. The symmetry of a given realization of the convex chain $(V_n,\cdots,V_1)$ gives a new convex chain $(V'_n,\cdots,V'_1)$ which is defined, to respect the decreasing abscissa convention by
  $V'_{n-i+1}= {\cal S}(V_i)$. 
We then observe that \[\bigl(X_{n-i+1}-Y_{n-i+1},0\leq i \leq n+1\bigl)~\eqd~ \bigl(2-(X_i+Y_i),0\leq i \leq n+1\bigl).\]
\centerline{\includegraphics{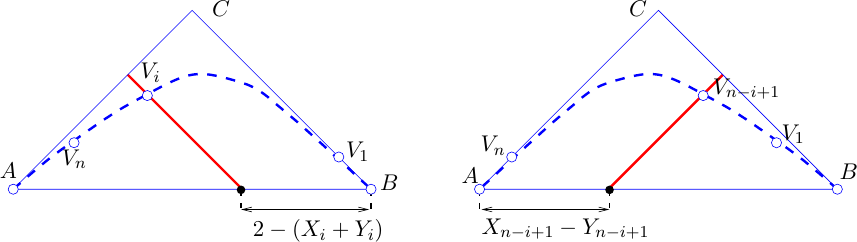}}
\end{rem}
 
\begin{proof}For $t=1$, the result is immediate, we can then suppose that $t<1$.
By \Cref{theo:qdzada2}$(c)$, and using the fact that for a $b_{s,t}\sim\beta(s,t)$ random variable, its expectation is $\E(b_{s,t})=s/(s+t)$, we get 
\ben\label{eq:superformule1}
\E\bigl(2-(X_j+Y_j)~|~\bS[n]\bigl)&=&\nonumber 2\frac{n+\bS_{n-1}}{n+2+\bS_n}\cdots \frac{j+\bS_{j-1}}{j+2+\bS_j}\\
&=&\nonumber2\frac{j+\bS_{j-1}}{n+2+\bS_n} \prod_{i=j}^{n-1} \frac{i+1+\bS_i}{i+2+\bS_i}\\
&=&2\frac{j+\bS_{j-1}}{n+2+\bS_n} \exp\l(\sum_{i=j}^{n-1}\log\l(1-\frac{1}{i+2+\bS_i}\r)\r).
\een
Then, for $t$ fixed in $[0,1]$, on  $(\Omega,{\cal A},\P),$
\[\E(2-(X_{\floor{(n+1)t}}+Y_{\floor{(n+1)t}})~|~\bS[n]) \hspace{10 cm} ~~\]
\ben\label{eq:superformule2}&=&2\frac{\floor{(n+1)t}+\bS_{\floor{(n+1)t}-1}}{n+2+\bS_n} \exp\l(\int_{\floor{(n+1)t}}^{n}\log\l(1-\frac{1}{\floor{i}+2+\bS_{\floor{i}}}\r)di\r)\\
\nonumber&=&2\frac{\frac{\floor{(n+1)t}+\bS_{\floor{(n+1)t}-1}}n}{\frac{n+2+\bS_n}n} \exp\l(\int_{\frac{\floor{(n+1)t}}{n+1}}^{1}\log\l(1-\frac{1}{\floor{(n+1)s}+2+\bS_{\floor{(n+1)s}}}\r)(n+1)ds\r)\\
\nonumber&\xrightarrow[n\to+\infty]{a.s.}& \frac{\sinh(2\rl t)}{\rl(1+\lambda)}\exp\l(-\int_{t}^1\frac{2\rl }{\sinh(2\rl s)}ds\r)
\een
Indeed, this is a consequence of \Cref{theo:podqgh} (that states $\|{\bf s}^{(n)}- M^\star\|_{\infty}\proba 0$), and of the approximation $(n+1)\log\l(1-\frac{1}{\floor{(n+1)s}+2+\bS_{\floor{(n+1)s}}}\r)=-\frac{1}{s+M^\star_{s}}+r_n(s)$ (where the rest $r_n(s)$ converging to zero) uniformly for all $s$ in any compact interval included in $(0,1]$. It remains to prove that this formula coincides with $Q(t)$. 
Using that $M_t^\star=-t+\sinh(2\rl t)/(2\rl)$, observe that
\[Q(t)=  \frac{\sinh(2\rl t)}{\rl(1+\lambda)}\exp\l(-\int_{t}^1\frac{2\rl }{\sinh(2\rl s)}ds\r)=\frac{\sinh(2\rl t)}{\rl(1+\lambda)}\frac{\tanh(\rl t)}{\tanh(\rl)}= 2\frac{\sinh^2(r_\lambda t)}{\sinh^2(r_\lambda)},\]
this last equality coming from $1+\lambda= \frac{\sinh(2\rl)}{2\rl}$.\par
Now, in order to compute the second moment of $2-(X_j+Y_j)$, we  use that  $\E(b^2_{s,t})= \frac{s(s+1)}{(s+t)(s+t+1)}=\E(b_{s,t})\frac{s+1}{s+t+1}$, so that
\ben
\E\l( (2-(X_j+Y_j))^2~|~\bS[n]\r)&=&  \E((2-(X_j+Y_j)~|~\bS[n]))\times 2\frac{j+1+\bS_{j-1}}{n+3+\bS_n} \prod_{i=j}^{n-1} \frac{i+2+\bS_i}{i+3+\bS_i}.
\een
Using the same method as below \eref{eq:superformule2} for this second product, we get, 
\[\E\l(\l(2-(X_{\floor{(n+1)t}}+Y_{\floor{(n+1)t}}\r)^2~|~\bS[n]\r)\xrightarrow[n\to+\infty]{a.s.} Q^2(t).\]
\end{proof}
Denote by $Q_n(t)=2-(X_{(n+1)t}+Y_{(n+1)t})$ and $\bar{Q}_n(1-t)=X_{(n+1)t}-Y_{(n+1)t}$, where both $X$ and $Y$ are defined by linear interpolation between $t$ of the form $k/(n+1)$ (just like we did for $X^{(n)}$ and $Y^{(n)}$). 
\begin{lem}\label{lem:qfgrte} On $(\Omega,{\cal A},\P)$, 
  $ \l\|Q_n - Q\r\|_\infty\proba 0$,   and  $ \l\|\bar{Q}_n - Q\r\|_{\infty}\proba 0$.
\end{lem} 
\begin{proof}By symmetry first, it suffices to prove that $ \l\|Q_n - Q\r\|_{\infty}\proba  0$.
 By Lemma \ref{lem:jolicv}, we already know that for a fixed $t$, $\l|Q_n(t) - Q(t)\r|\proba 0$. Turning this point wise onto an uniform convergence is routine, as a well-known consequence of Dini's theorem valid under regularity and monotonicity hypothesis. Here $2-(X_{(n+1)t}+Y_{(n+1)t})$ is the distance from the projection of $(X_{(n+1)t},Y_{(n+1)t})$ on $AB$ with respect to the direction $BC$, so that $t\mapsto 2-(X_{(n+1)t}+Y_{(n+1)t})$ is non-decreasing (from 0 to 2) and continuous. Here, the point-wise limit is non-decreasing and continuous. Since it is deterministic, the argument is even simpler.   
 
The general argument (and proof of the argument) runs as follows:  assume that $(x_n)$ is a sequence of processes, where $x_n=(x_n(t),0\leq t\leq 1)$ takes its values in $C[0,1]$, and $t\mapsto x_n(t)$ is non-decreasing (a.s.). If $x_n(t)\proba x(t)$ for all $t$ where $x=(x(t),0\leq t\leq 1)$ is continuous (and then necessarily) non-decreasing, then $(x_n)$ converges to $x$ in $C[0,1]$ equipped with $\|.\|_\infty$. When $x$ is deterministic, we then have $\|x_n-x\|_{\infty}\proba 0$. To prove this, it suffices to take a subdivision $t_0=0<t_1<\cdots <t_{k-1}<t_{k}=1$, and write
  \[\sup_{t\in[0,1]} |x_n(t)-x(t)|\leq \max_{j} \{|x_n(t_{j+1})-x(t_j)|,|x_n(t_j)-x(t_{j+1})|\}\proba  \max_j | x(t_{j+1})-x(t_j)|\]
  so that it is clear that by choosing a subdivision $(t_i)$ thin enough, this can be made arbitrary small (since $x$ is continuous).
\end{proof}

  \paragraph{End of Proof of \Cref{theo:lim}.}
We have 
\[X_{(n+1)t} = -\frac{Q_n(t)}{2}+1+\frac{Q_n(1-t)}{2},~Y_{(n+1)t}= -\frac{Q_n(t)}{2}+1-\frac{Q_n(1-t)}{2}.\]
From \eref{lem:qfgrte}, we  get
$\sup_{t}\l|X_{(n+1)t}- \l( 1-\frac{Q(t)-Q(1-t)}{2} \r)\r|\cvg 0$ in probability,
as well as \linebreak$\sup_{t}\l|Y_{(n+1)t}-\l(1-\frac{Q(1-t)+Q(t)}{2} \r)\r|\cvg 0$ in probability, which is \Cref{theo:lim}.


\color{black}

\subsection{Random generation of $V[n]$ under  $\bQt_{n,m}$}

In theory, by rejection sampling method, it is possible to draw $n+m$ uniform points $U_1,\cdots,U_{n+m}$ uniformly and independently in ABC and repeat this operation until exactly $n$ of them are on the boundary of the convex hull of $\{A,B, U_1,\cdots,U_{n+m}\}$. In practice, since $\bQt_{n,m}$ is tiny even for small values of $(n,m)$, for example,
\[\bQt_{12,14}=\frac{4862354860353208414849066648349}{57539992882666702350263300179200000000000}= (8.45...) 10^{-11}\]
this method can be used only for very small values of $n$ and $m$.

In the rest of this section, we explain how, on a standard personal computer, one can simulate the convex chain $V_1,\cdots,V_n$ under $\Qt_{n,\fnl}$ for $n$ and $\fnl$ going to up to some hundreds  (if one has in hands a random generator of iid uniform random variables on $[0,1]$) within minutes/hours of computations, depending on the programming language one uses.

We present here the main ideas of the method used to produce the simulations given in    Fig. \eref{fig:my2}.

\begin{figure}[h!]
      \begin{center}\includegraphics[width=5.5cm, height=4.5cm]{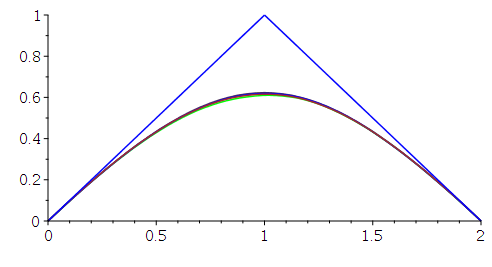}~~~~~~\includegraphics[width=5.5cm,height=4.5cm]{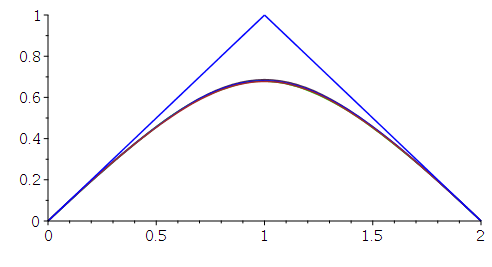}~\\
\includegraphics[width=5.5cm,height=4.5cm]{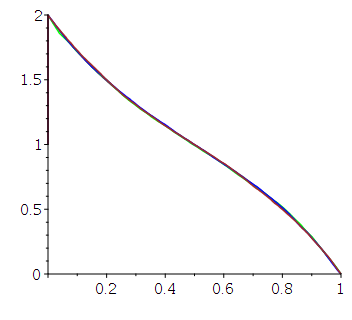}~~~~~~\includegraphics[width=5.5cm,height=4.5cm]{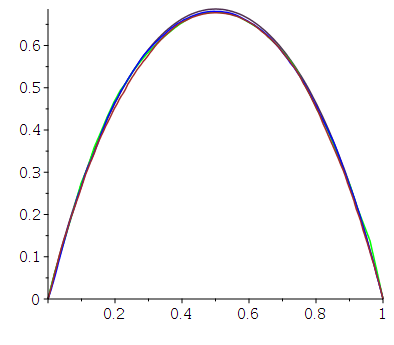}
\end{center}
    \captionn{   \label{fig:my2}Simulations: on the first picture, $\lambda=1$, $(n,m) = (50,50)$ (green), (100,100) (blue), (150,150) (brown). Mean taken over 50 simulations of trajectories, for each size. The three curves, violet, blue, brown, are hardly distinguishable. On the second picture, $\lambda=2$, $(n,m) = (50,100)$ (green), (100,200) (blue), (150,300) (brown), violet, the limiting curve.
     Still in the second case, with the same color code, drawings of $t\mapsto X(t)$, and $t\mapsto Y(t)$, and for the simulation, drawing of $((k/n, X_{k}),k=1..n)$ and $((k/n,Y_{k})),k=1..n)$ }
 
\end{figure}

\paragraph{A random generation algorithm.}
First, start by sampling exactly the random vector $K[n]$ under $\Qt_{n,\floor{n\lambda}}$.
To do so, the idea is to pre-compute all the numbers $\bQt_{n',m'}$ for $0\leq n'\leq n$ and $0\leq m'\leq m$ by doing the following computation:

\paragraph{1. Precomputation (and storage) of the $\bQt_{n,m}$}~\\
For $n'$ from 0 to $n$:  \\
\hspace*{5mm} $c_{n',0}^\triangle=2^{n'}/({n'}!({n'}+1)!)$  
\medskip

\noindent for $m'$ from 1 to $m$:\\
\hspace*{5mm} for $n'$ from 0 to $n$:\\
\hspace*{10mm}compute $c_{n',m'}^\triangle$ (using \eref{eq:Cnm})\\
\hspace*{10mm}compute $\bQt_{n',m'}$  (using \eref{eq:bQt})\\

To perform these computations for $n$ larger than some dozens, say, it is important to use a programming language/library able to work with fractions, whose numerators and denominators are very large integers (otherwise, exact simulations have to be obtained using another strategy).  Since $c_{n',m'}^\triangle$ is a priori just a sum of $m'$ terms, computing all the $c_{n',m'}^\triangle$ demand a cubic number of sums (that is  $O(\max\{n,m\}^3)$. However, exact computations with large numbers, computations of the involved binomial coefficients increase greatly the cost, and depends on the programming language one uses. A pre-computation (and storage) of the binomial coefficients involved in the computation of the $c_{n,m}^\triangle$ improves greatly the performance of the algorithm.

\paragraph{2. Simulation of $K[n]$ under $\bQt_{n,m}$.}

The distribution of  $K_n$ under $\bQt_{n,m}$ is 
\[\bQt_{n,m}(K_n=k) =   2 \bQt_{n-1,m-k}\frac{(k+1)}{(n+m+1)(n+m) \bQt_{n,m}},~~~~k\in \{0,\cdots,m\}.\]
Since the coefficients $\bQt_{n',m'}$ have been precomputed, all these probabilities can be quickly computed, and standard simulations of a discrete distribution can be used.
When $K_n$ has been generated and say $K_n=k_n$ is obtained, then simulate $K_{n-1}$ according to $\Qt_{n-1,m-k_n}$ by the same procedure, and iterate: for $K_{n-j}$, simulate a random variable with distribution $\Qt_{n-j, m-k_n-\cdots-k_{n-j+1}}$.

The cost of this simulation is negligible compared to that of point ${\bf 1.}$  

\paragraph{3. The rewinding construction.}

It remains to observe that when $K_n=k$ is known, then $V_n$ can be simulated. It suffices to construct a random generation for its coordinates $(X_n,Y_n)$ distributed as $(b(2-g),bg)$ for $b\sim \beta(k+2,n+m-k)$ and $g\sim \beta(n+m,1)$ (for two independent beta $\beta$ random variables).

In fact, to construct the points $(V_n,\cdots,V_1)$ we will need to build first the $(V_n^\star,\cdots,V_1^\star)$, where $V_i^\star$ is taken under the law described in \Cref{theo:rewind}.

When the $(V_i^\star)$ are known, in order to build the $(V_i)$ and complete the construction, we need the ``rewinding'' considerations explained in  \Cref{theo:rewind} and slightly before. There are basically two different ways of proceeding: either place $V_n$ before, compute the coordinates of $P$ as drawn in Fig. \ref{fig:KV} (which is $P=2(X_n,Y_n)/(X_n+H_n)$) and then work in the new triangle $(A',B',C)=(V_n,B,P)$ to place $V_{n-1}$ and proceed inductively, or use directly the formulas of \Cref{theo:rewind} to construct the whole picture.

\section{Deterministic geometric considerations}

\subsection{About  the optimization of $\Phi_\lambda$ in the bi-pointed case: Proof of \Cref{theo:optri}}
\label{sec:qddqs}

We divided our proof of  \Cref{theo:optri} in  6 sections, corresponding to the 6 main ideas of this proof.

The keystone here is probably the introduction of a new notion, that we call ``cupola symmetry'', a sort of super-symmetry, which will play a role also in the optimization of $\PLK$ in a general compact convex $\K$.
We will see that the elements of  $\argmax\Phi_\lambda$ are ``cupola-symmetric'', from what we deduce that they must be either a parabola or an hyperbola. Then, a more classical optimization argument within this family allows to conclude.

\subsubsection*{(1). The map $\Phi_\lambda:\CCST\to {\R^+}$ reaches its maximum.}
The proof is the same as that of \Cref{pro:hss}$(a)$ in $\CCST$ instead of ${\sf CCS}_\K$. 

\subsubsection*{(2). An optimizing concave map ${\cal C}$ must be smooth.}  
Let $F$ be a function in ${\sf Conc}(ABC)$ such that ${\cal C}_F$ is an element of $\argmax  \Phi_\lambda$.   Assume that $F$ is not smooth: since it is concave, and since its graph belong to $ABC$, $F$ is differentiable a.e., and there exists an abscissa $x$ such that  $F'(x^+)<F'(x^-)$ (somehow, a jump in the derivative).
Consider two points $z_1=(x_1,y_1)$ and $z_2=(x_2,y_2)$ on ${\cal C}_F$, with $x_1<x_2$, such that there is such a jump for the derivative either at $x_1$ or $x_2$.
Take two supporting lines $L_1$ and $L_2$ at $z_1$ and $z_2$ (we demand the slopes to be in $[-1,1]$, which is a restriction only for $x_1=0$ and $x_2=2$). Denote by $z_3=(x_3,y_3)$ the intersection of these lines. A small picture is sufficient to see that (because of their slopes) $z_3$ is inside ABC. The triangle $z_1z_2z_3$ is enveloping if $\mathcal{C}_F$ is smooth at $z_1$ and at $z_2$, but if it is not the case, then for some choices of $L_1$ and $L_2$, the triangle $z_1z_2z_3$ will contain strictly the enveloping triangle, say $z_1z_2 z'_3$. Take such a pair of (non-minimal) supporting lines $(L_1,L_2)$.

Now, the affine map $\phi_{z_1z_2z_3'\to z_1z_2z_3}$   has  determinant $>1$ .
Let us now perform some  curve surgery. Define  the curve ${\cal C}'$ as follows:\\
\indent -- outside the triangle $z_1z_2z_3$, ${\cal C}_F$ and ${\cal C'}$ coincide.\\
\indent -- inside the triangle $z_1z_2z_3$: the restriction of ${\cal C'}_{~|~{z_1z_2z_3}}$ is taken to be $\phi_{z_1z_2z_3'\to z_1z_2z_3}({\cal C}_{~|~{z_1z_2z'_3}})$. 

It is easy to see that ${\cal C}'$ is still in $\CCST$, and   moreover $\L({\cal C}')>\L({\cal C}_F)$ and $\A({\cal C}')>\A({\cal C}_F)$, so that $\Phi_\lambda({\cal C}')>\Phi_\lambda({\cal C}_F)$, which is a contradiction. Hence, an element in $\argmax \Phi_\lambda$ must be smooth.


\subsubsection*{(3).  An element in $\argmax \Phi_\lambda$ must be cupola-symmetric}

We start by proving that any element of $\argmax \Phi_\lambda$ must be symmetric with respect to $x=1$, and we will progressively deduce that it must be ``cupola-symmetric'', a notion we introduce below, which is a sort of super-symmetry, that, as we will see later on, are only satisfied by conics.

\begin{lem}\label{lem:dqsd}
  If ${\cal C}_F$ is an element of $\argmax \Phi_\la$ for some $F\in{\sf Conc}(ABC)$, then $F$ is symmetric with respect to $x=1$ (or equivalently, ${\cal C}_F={\cal S}({\cal C}_F)$ where ${\cal S}$ is the symmetry with respect to the line $x=1$).
  \end{lem}
\begin{proof}
  Let  $\bar{F}$ be the symmetric of $F$, (that is $\bar F(x)=F(2-x)$ for $x\in[0,2]$). Define $G$ as the concave function such that
  \[{\cal C}_G= ({\cal C}_{F}+{\cal C}_{\bar F})/2,\] where we equipped $\CCST$ with the Minkowski addition (or equivalently, the Blaschke sum).
Since  $\A( ({\cal C}_{F}+{\cal C}_{\bar F})/2 )^{1/2}\geq \A( {\cal C}_{F}/2)^{1/2}+\A({\cal C}_{{\bar F}}/2)^{1/2}$ by the Brünn-Minkowski (or Kneser-Süss) inequality, with $\A( {\cal C}_{F}/2)^{1/2}=\A({\cal C}_{{\bar F}}/2)^{1/2}$,
we have
\[\A( ({\cal C}_{F}+{\cal C}_{{\bar F}})/2 )\geq ( \A( {\cal C}_{F}/2)^{1/2}+\A({\cal C}_{{\bar F}}/2)^{1/2})^2= \A( {\cal C}_{F}).\]
By Lemma 1 and Eq. (3.3) in \Bara \cite{barany1}, $\L$ is concave in $\CCST$ for the Minkowski sum, that is   $\L((S_1+S_2)/2)\geq \frac{1}{2}(\L(S_1)+\L(S_2))$.
Hence the map $G$ satisfies $\Phi_\lambda({\cal C}_G)\geq \max\{ \Phi_\lambda({\cal C}_F),\Phi_\lambda({\cal C}_{\bar F})\}$, and this inequality is strict if ${\cal C}_F$ is not equal to ${\cal C}_{\bar F}$ up to a translation\footnote{The case of equality in  Brunn Minkowski inequality is the case where ${\cal C}_F$ and ${\cal C}_{\bar F}$ are equal (since it is the only case for which the two curves ${\cal C}_F$ and ${\cal C}_{\bar F}$ are translated of each other)}. Hence if $F$ is not symmetric with respect to $x=1$, then ${\cal C}_F \notin \argmax  \Phi_\la.$
\end{proof}
 
This argument ``if ${\cal C}_F$ is not symmetric, then $\Phi_\lambda(({\cal S}({\cal C}_F)+{\cal C}_F))/2)> \Phi_\lambda({\cal C}_F)$'', or in other words, a symmetrization increases the value of $\Phi_\lambda$, can be applied in any enveloping triangle of a portion of ${\cal C}_F$ (such a triangle is always entirely included in $ABC$).

Let us state a trivial lemma that will allow us to further define the notion of ``cupola-symmetric curve''
\begin{lem} Let $abc$ be a non-flat triangle. The affine map $\phi_{abc\to bac}$ that sends $(a,b,c)$ onto $(b,a,c)$ has determinant 1, so that it preserves area and affine lengths. This map is the symmetry  with respect to the line $cm$, where $m$ is the middle of $ab$, according to the direction of $ab$. 
\end{lem}
\begin{defi}A subset $D$ of the plane is said to be $abc$-symmetric if $\phi_{abc\to bac}(D)=D$.
\end{defi}
\begin{figure}[hbtp]
    \centering
    \includegraphics[width=7cm]{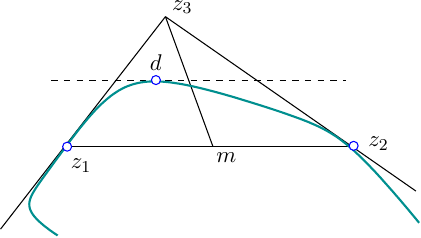}
    \captionn{This green curve ${\cal C}$ is not cupola symmetric, since for this enveloping triangle $z_1z_2z_3$, the affine map $\psi_{z_1z_2z_3\to z_2z_1z_3}$ sends the point $d$ on the dotted line (which is the line parallel to $z_1z_2$ incident to $d$), in the triangle $mz_2z_3$. Hence the restriction ${\cal C}~|_{z_1z_2z_3}$ is not preserved by $\psi_{z_1z_2z_3\to z_2z_1z_3}$.}
    \label{fig:cup}
\end{figure}
\begin{defi} A curve ${\cal C}\in \CCST$ is said to be \underbar{cupola symmetric} if for all pairs of distinct points $z_1=(x_1,y_1)$, $z_2=(x_2,y_2)$ on its upper boundary  with $0\leq x_1<x_2\leq 1$, $y_1,y_2>0$, any  supporting line $L_1$ and $L_2$ at $z_1$ and $z_2$ with slopes in $[-1,1]$, the restriction of ${\cal C}_{~|~{z_1z_2z_3}}$ to the triangle $z_1z_2z_3$ with $z_3=L_1\cap L_2$, is $z_1z_2z_3$-symmetric (that is invariant by $\phi_{z_1z_2z_3\to z_2z_1z_3}$). 
\end{defi}


It is easy to check that any hyperbola ${\cal H}_v$ for $v\geq 0$ is cupola symmetric: indeed, since these curves are smooth, supporting lines are tangents, and any triangle $z_1z_2z_3$ mentioned in the definition of ``cupola symmetric'' is an enveloping triangle. Now, images of (pieces of) hyperbolas through affine maps are (pieces of) hyperbolas, and the conclusion follows.
\begin{lem}\label{lem:sct} If  ${\cal C}$ is an element of $\argmax \Phi_\la$ then:\\
  ${\it (i)}$ ${\cal C}$ is smooth, ${\it (ii)}$ cupola-symmetric, and ${\it (iii)}$ its global enveloping triangle is $ABC$ (the smallest triangle that contains ${\cal C}$ is $ABC$).
\end{lem}
\begin{proof} ${\it (i)}$ is ${\bf (2)}$. Let us pass to ${\it (iii)}$. Take the enveloping triangle of ${\cal C}$: this triangle must contain $A$ and $B$, and since the slope of the tangent line at $A$ is smaller than 1, and that the one at $B$ is larger than $-1$, its third vertex $C'$ is either $C$ or inside $\Int(ABC)$. Let us rule out this last possibility. The affine map $\phi_{ABC'\to ABC}$ has determinant strictly greater than 1, and we have $\phi_{ABC'\to ABC}({\cal C})$ is a convex set in $\CCST$, and moreover, $\Phi_\la(\phi_{ABC'\to ABC}({\cal C}))> \Phi_\la( {\cal C})$, so that ${\cal C}\notin \argmax \Phi_\la$, which is a contradiction.

  ${\it (ii).}$ Pick ${\cal C}$ in $\argmax \Phi_\la$; we know that ${\cal C}$ is smooth by {\bf (2)}. Assume that ${\cal C}$ is not cupola symmetric: there exists an enveloping triangle $z_1z_2z_3$, such that the restriction of $C={\cal C}_{~|~{z_1z_2z_3}}$ is not $z_1z_2z_3$ symmetric.
  In this case, let us perform a small curve surgery to construct $\mathcal{C}'$ out of $\mathcal{C}$ such that $\Phi_\la(\mathcal{C}')>\Phi_\la(\mathcal{C})$:
  take $C'= \phi_{z_1z_2z_3\to z_2z_1z_3}(C)$ the image of $C$ by the $z_1z_2z_3$ symmetry. This symmetry preserves the area and affine perimeter of the intersection of $C$ with $z_1z_2z_3$~:
  \beq\label{eq:rhyrgk} \l(\A( \CH(C)),\L(C )\r)=\l(\A( \CH(C')) ,\L( C' )\r).\eq
  Now, replace the piece $C$ in $z_1z_2z_3$ by 
  \[C''=(C+ \phi_{z_1z_2z_3\to z_2z_1z_3}(C))/2\] which fits
  in the triangle $z_1z_2z_3$.

  The surgery.  Build ${\cal C}'$ as follows:\\
  \indent-- outside $z_1z_2z_3$, ${\cal C}$ and ${\cal C}'$ coincide\\
  \indent -- inside $z_1z_2z_3$: remove $C$ from ${\cal C}$ and replace it by $C''$.

  It remains to compare $\Phi_\lambda({\cal C})$ and  $\Phi_\lambda({\cal C}')$.
  The common part of ${\cal C}$ and ${\cal C'}$ is ${\sf Com}=\overline{{\cal C} \setminus (z_1z_2z_3)}$. The segment $z_1z_2$ is a boundary of ${\sf Com}$, and this part has affine length 0.
  We have
  \ben
  \Phi_\lambda({\cal C}) &=&\l(\A({\sf Com})+\A(\CH(C))\r)^{\lambda}.\l(\L({\sf Com})+ \L(C)\r)^3\\
  \Phi_\lambda({\cal C}') &=&\l(\A({\sf Com})+\A(\CH(C''))\r)^{\lambda}.\l(\L({\sf Com})+ \L(C'')\r)^3
  \een
 The same argument as for the proof of \Cref{lem:dqsd} allows to conclude: by hypothesis, $C'$ is different from $C$. Since both of them are drawn in $z_1z_2z_3$ and are incident to $z_1$ and $z_2$, they are not equal up to a translation. Using \eref{eq:rhyrgk} plus $\A(\CH(C''))>\A(\CH(C))$ as well as the concavity of the affine length with respect to the Minkowski addition, we deduce that $\L(C'')\geq \L(C)$. Hence $\Phi_\lambda({\cal C}')> \Phi_\lambda({\cal C})$ and this contradicts the fact that ${\cal C}$ is in $\argmax \Phi_\la$. We have established that a curve which is not cupola symmetric is not in $\argmax \Phi_\la$.
\end{proof}

\begin{rem}\label{rem:gfq} At this point, a geometer may have concluded that cupola symmetric curves ${\cal C}$ are indeed conics. Basically, these curves are conics because when one takes an enveloping triangle $abc$ of a part $p$ of the curve, the $abc$ symmetry implies that one can cut this triangle in half (by taking the ``axis $cm$ of symmetry'', with $m$ the middle of $ab$). Each part of $p$ obtained, say $p_1$ and $p_2$, lies in one  ``half of the triangle $abc$'', on either side of the line $cm$. They have same affine perimeter and same area, because the affine map $\phi_{abc\to bac}$ has determinant $-1$ and sends $p_1$ onto $p_2$. One can then iterate, and set an enveloping triangle $a_1b_1c_1$ and $a_2b_2c_2$ on each of these two parts, and proceed as before to produce thinner parts... All these 4 parts must have the same affine length and area. We produce recursively $2^n$ parts of $p$ that all have same affine perimeters and enveloping triangles having same area. This implies that the affine curvature must be constant, which immediately allows to conclude (see Chap. 3 in Su \cite{MR724783}).  \end{rem}

However, in the interests of rigor, we provide a more elementary argument in the sequel.

\subsubsection*{(4). A binary decomposition of cupola symmetric function in $\argmax \Phi_\la$}
By \Cref{lem:sct} we already know some properties of an element ${\cal C}$ of $\argmax \Phi_\la$. We add here a binary decomposition property. Take $F \in {\sf Conc}(ABC)$ such that ${\cal C}={\cal C}_F$. The curve ${\cal C}$ is cupola symmetric, and in particular symmetric with respect to the line $x=1$.
Denote by $H:=F(1)$ the central height. Take the central point $p=(1,H)$ of this curve, and $M=(H,H)$ the point on the segment $AC$, at ordinate $H$.
\begin{figure}[h!]
\centerline{\includegraphics{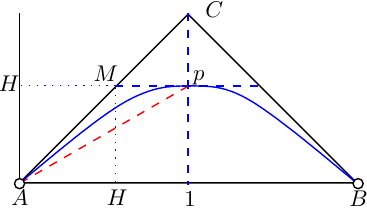}}
\caption{\label{fig:Hp} Toward the central height decomposition}
\end{figure}

\begin{lem} If ${\cal C}={\cal C}_F$ belongs to  $\argmax \Phi_\la$, then the triangle $ApM$ is the enveloping triangle of the first half of ${\cal C}$ (more precisely of $\{( x,F(x))~:~x\leq 1\}$).
\end{lem}
\begin{proof} By \Cref{lem:sct}, $ABC$ is enveloping for ${\cal C}$, so that, we already know that $F'(0)=1$. Hence, $ApM$ is either enveloping, or contains strictly the enveloping triangle  $ApM'$ (this occurs when $F'(1-)>0$). It remains to rule out this possibility. But again, by an argument already used several times, the affine map $\phi_{ApM'\to ApM}$  has determinant $>1$. Then merging $C=\{( x,F(x))~:~x\leq 1\}$ on the first half of the upper boundary and $C'=\phi_{ApM'\to ApM}(C)$ provides us with the half of a second curve ${\cal C}'$ which has largest area and largest affine perimeter than ${\cal C}$.
\end{proof}

This property holds recursively: consider $ApM$ the enveloping triangle of this first half.
Since ${\cal C}$ is cupola symmetric, we have
\beq\label{eq:dqfsfsf} \A({\cal C})= H+ 2\; \A( \CH(C)), ~~~\L({\cal C}) = 2\; \L(C).\eq

Let us rename $({\cal C},H, C,p,M)$ by $({\cal C}_0,H_0,C_0,p_0,M_0)$ and let us start a recursion. Examine this ``half curve $C_0$''. The triangle $Ap_0M_0$ can be sent on $ABC$ by $\phi_{Ap_0M_0\to ABC}$, and then the image of $C_0$ is a curve ${\cal C}_1$ in $\CCST$ (a continuous convex chain in the unit triangle $ABC$). Since the area of $A_0p_0M_0$ is $(1-H_0)H_0/2$, 
we have by updating \eref{eq:dqfsfsf} 
\beq\label{eq:qsdqd}
\A({\cal C}_0)= H_0+ 2\; \frac{(1-H_0)H_0}2\A( {\cal C}_1), ~~~\L({\cal C}) = 2\;\l(\frac{(1-H_0)H_0}2\r)^{1/3} \L({\cal C}_1).\eq
\begin{defi}\label{defi:dp} We call $(H_0,{\cal C}_1)$ the decomposition pair of ${\cal C}_0$.\end{defi}
\begin{lem}If ${\cal C}_0\in \argmax \Phi_\la$,  $H_0\geq 1/2$.
\end{lem}
\begin{proof} If $H_0<1/2$, then the curve ${\cal C}'_0$ obtained by replacing $H_0$ by $1-H_0$ (that is, technically, by applying an affine map $\phi:(x,y)\to(x, y (1-H_0)/H_0)$ to $ABC$, and take ${\cal C}'_0=\phi({\cal C}_0)$  would provide the decomposition \eref{eq:qsdqd} with $(1-H_0,{\cal C}_1)$ instead of $(H_0,{\cal C}_1)$, and we would have
  \[\A({\cal C}_0')>\A({\cal C}_0) ,~~~ \L({\cal C}_0')=\L({\cal C}_0)\]
  so that $\Phi_\la({\cal C}_0')> \Phi_\la({\cal C}_0)$, a contradiction.
\end{proof}

\begin{lem}\label{lem:qdef}If ${\cal C}_0$ is cupola symmetric, has $ABC$ as enveloping triangle, is smooth and has a central height $H_0\geq 1/2$, then  so does ${\cal C}_1$.
  Moreover, if $H_0>1/2$ then the central height of ${\cal C}_1$ satisfies
  \[H_1=\frac{ 1-\sqrt{2(1-H_0)}}{2H_0-1}>1/2\]
  and if $H_0=1/2$, then $H_1=1/2$ too. 
\end{lem}
\begin{rem}Notice that we do not suppose here that ${\cal C}_0$ is in $\argmax \Phi_\lambda$ in  \Cref{lem:qdef}.\end{rem}
\begin{proof}  
  The central height $H_0$ of ${\cal C}_0$ is $\geq 1/2$ by the previous lemma. 
Consider Fig. \ref{fig:zdnze}. In green a concave function $F_0$ is represented, and we suppose that ${\cal C}_0={\cal C}_{F_0}$, and $p_0=(1,F_0(1))=(1,H_0)$, $M_0=(H_0,H_0)$, $M_0'=(2-H_0,H_0)$,. 
We have placed the two points $m_1$ and $m_2$, respectively middle points of $[A,p_0]$ and $[p_0,B]$.
\begin{figure}[h!]
    \centering
    \includegraphics{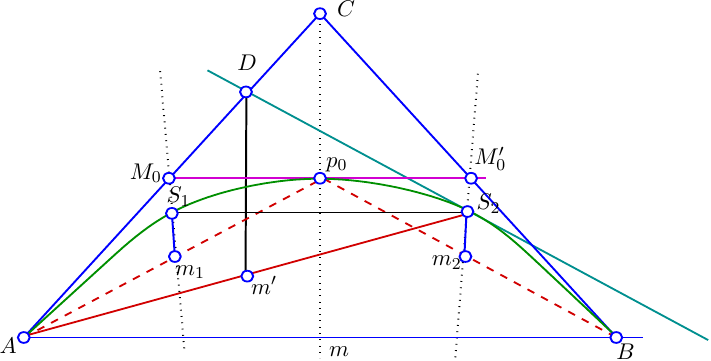}
    \captionn{ \label{fig:zdnze} Illustration of the elements of the proof of \Cref{lem:qdef}}
    
\end{figure} 
Now, 
\[S_1=\bma 1/2+H_1(H_0-1/2) \\ H_0/2+H_0H_1 /2\ema,~~~~~S_2=\bma 3/2-H_1(H_0-1/2)\\ H_0/2+H_0H_1/2\ema\] are the two (key) points that appear by taking into account that the central height of ${\cal C}_1$ is $H_1$ (hence, for example $S_1$ is the image of $(1,H_1)$ by $\phi_{ABC\to Ap_0M_0}$, and $S_2$ obtained by symmetry with respect to $Cm$). Let $m'= S_2/2$ be the middle point of $[A,S_2]$. Then the tangent at $S_2$ intersects $AC$ at $D$ (always in the segment $[A,C]$).

Since  ${\cal C}_0$ is cupola symmetric and smooth (by \Cref{lem:sct}), the three triangles $Ap_0M_0$, $Bp_0M_0'$ and $AS_2D$ are enveloping, and therefore the part of ${\cal C}_0$ lying in such a triangle $abc$ is $abc$-symmetric. 
Hence, there exists : \\
\indent--${\it (i)}$ an affine map $g_1$ with determinant -1 that sends $(A,p_0,M_0)$ onto $(p_0,A,M_0)$,\\
\indent--${\it (ii)}$ an affine map $g_2$ with determinant -1 that sends $(B,p_0,M_0')$ onto $(p_0,B,M_0')$,\\
\indent--${\it (iii)}$ an affine map $g_3$ with determinant -1 that sends $(A,S_2,D)$ onto $(S_2,A,D)$. 

Each of these affine maps has the standard form discussed before, for example $g_1=\phi_{Ap_0M_0\to p_0AM_0}$. 
We need to compute the matrix and translation vector of the affine map
$g_3: \bma x \\y \ema \mapsto \bma a_1 & a_2 \\ a_3 & a_4\ema \bma x \\y \ema  +\bma v_1 \\ v_2\ema$ (so that ${\it (iii)}$ holds).

One finds, \be {v_1}&=&{
\frac{3}{2}}-{H_1}\,{H_0}+{\frac {{H_1}}{2}}, ~~~~~
{v_2}={\frac 
  {{H_0}\, \left( 1+{H_1} \right) }{2}}  \ee
and for $\Gamma =    \left( 6\,{H_1}\,{H_0}+2\,{H_0}-2\,{H_1}-6
 \right)  \left( {H_1}\,{H_0}-{H_1}-2 \right) {H_0}$, we have 
\be {a_1}&=&\Gamma^{-1}\l( \left( -4\,{{H_0}}^{2}+3\,{H_0}-1
 \right) {{H_1}}^{2}+ \left( 10\,{H_0}-6 \right) {H_1}-{H_0}-9\r),\\
 {a_2}&=&\Gamma^{-1} \l( { \left( 2\,{H_1}\,{{H_0}}^{2}-5\,{H_1}\,{H_0}-5\,{
H_0}+{H_1}+3 \right)  \left( 2\,{H_1}\,{H_0}-{H_1}-3
 \right) }\r),\\
 a_3&=&\Gamma^{-1}\l({   \left( 1+{H_1} \right)  \left( 5\,{H_1}\,{H_0}+{H_0}-3\,{H_1}-7 \right) {H_0}}\r),\\
{a_4}&=&-a_1
\ee
 We claim that $g_3$ has the following property
 \ben g_3(p_0)=S_1 \textrm{  and }g_3(S_1)=p_0. \een
Indeed, consider the following parts $b_1,b_2$ and $b_3$ of the boundary $\partial{\cal C}_0$,~:
 \[b_1:=A\to S_1,~~~ b_2:=S_1\to p_0,~~~ b_3:=S_2\to p_0.\] 
 these three parts are image of each other by affine maps with determinant $\pm 1$:
 first $g_1(b_1)=b_2$ and $b_3$ is the image of $b_2$ by the symmetry with respect to the axis $Cm$. Therefore, we must have affine length equality:
 \[\L(b_1)=\L(b_2)=\L(b_3).\]

 Now,  since $g_3$ has determinant $-1$, it preserves the affine length (of any curve).
Then we must have $g_3(S_1)=p_0$. Indeed, first
$\L(A\to S_1)= \L(g_3(A)\to g_3(S_1))=\L(S_2\to  g_3(S_1))$ and by symmetry with respect to the line $mC$, we have also  $\L(A\to S_1)= \L(S_2\to B)$ and $\L(S_1\to p_0)=\L(S_2\to p_0)$.
We deduce from $\L(S_2\to  g_3(S_1))= \L(S_2\to p_0)$ that $g_3(S_1)=p_0$.

If $H_0\in(1/2,1]$, a short analysis shows that this is possible only if $H_1=(1 - \sqrt{-2H_0 + 2})/(2H_0 - 1)$. If $H_0=1/2$, then $H_1=1/2$ too. 
\end{proof}

\subsubsection*{(5). Iteration of the binary decomposition, and appearances of conics}

In {\bf (4)}, \Cref{lem:qdef} has been designed to be iterated! Recall \Cref{defi:dp}.
If ${\cal C}_0$ is cupola symmetric, smooth and has a central height $H_0\geq 1/2$, define successively, for $i\geq 0$, the decomposition pairs $(H_i,{\cal C}_{i+1})$ of ${\cal C}_i$. 
\begin{pro} Assume that ${\cal C}_0$ has $ABC$ as enveloping triangle, is smooth and cupola symmetric, and has central height $H_0\geq 1/2$. Then for all $i\geq 0$, ${\cal C}_{i}$ has also these properties. Moreover, if $H_0>1/2$, then for all $i$,
 \[H_{i+1}=(1 - \sqrt{2-2H_i })/(2H_i - 1)>1/2.\]
and if $H_0=1/2$, then $H_i=1/2$ for all $i\geq 0$.
\end{pro}
Beyond appearance, this proposition states that the set of functions that satisfies the hypothesis is a one parameter family of functions (since the $H_i$ are all functions of $H_{0}$).

\begin{pro}Assume that ${\cal C}_0$ has $ABC$ as enveloping triangle, is smooth and cupola symmetric, and has central height $H_0$. Then 
  if $H_0=1/2$ then ${\cal C}_0$ is the parabola ${\cal P}={\cal H}_0$, and if $H_0>1/2$, then ${\cal C}_0$ is the only hyperbola ${\cal H}_v$ having central point at height $H_0$. It is characterized through the parameter $r_v$ by 
\beq\label{eq:HO} H_0= \frac{2\cosh(r_v)\sinh(r_v/2)^2}{\sinh(r_v)^2}=\frac{\cosh(r_v)}{2\cosh(r_v/2)^2}.\eq
\end{pro}

\begin{proof} Since we know that all elements of the set  $\{{\cal H}_v,v\geq 0\}$ (with ${\cal P}={\cal H}_0$) satisfy the hypothesis of the proposition, and ${\cal P}$ is the only element with central height $1/2$, and also, by the explicit formula of $Y$ (recall \eref{eq:Y}), we know the central height in the hyperbola ${\cal H}_v$ in terms of $r_v$: there is a single parameter $r_v$ for which \eref{eq:HO} holds.
	
It remains then to explain why there is a single element in $\CCST$ satisfying the hypothesis of the proposition for a fixed $H_0$. The reason is that $H_0$ characterizes all the $(H_i,i\geq 0)$. In turn, all convex sets ${\cal C}_0$ (satisfying the hypothesis of the proposition) which have the same first  $(H_i, 0\leq i \leq k)$ (for some $k$), coincides at $2^k$ points. Moreover, these $2^k$ points are well scattered on each of these curves, since the affine length between consecutive points are equal. The conclusion follows by taking a limit over $k$, using that all the curves are smooth (and then continuous).
  \end{proof}

\subsubsection*{(6). Characterization of the hyperbola that maximizes $\Phi_\la$.}

It remains to prove that among all hyperbolas, a unique one maximizes $\Phi_\lambda$, being ${\cal H}_\lambda$.  To do so, for $\nu\geq0$, write
\be  \Phi_\lambda({\cal H}_\nu)= \L^{3/\lambda}({\cal H}_\nu)\A({\cal H}_\nu)= \l( \frac{2r_\nu \cosh(r_\nu)^{1/3}}{\sinh(r_\nu)}\r)^{3/\lambda} r_\nu{\frac {\cosh \left(r_\nu \right)  }{ 
\sinh \left( r_\nu\right)  ^{3}} \left( {\frac{\sinh \left( 2\,r_\nu \right)}{2r_\nu}}-1\right) }.
\ee
The derivative $\frac{\partial}{\partial r_\nu} f_\lambda[r_\nu]$ cancels only when 
\[\frac{\sinh(2r_\nu )}{2r_\nu }-1=\lambda.\]
Observe that the limit when $\nu\to+\infty$ of  $\Phi_\lambda({\cal H}_\nu)$ is zero (because the affine length goes to zero), and $\lim_{\nu\to 0} \Phi_\lambda({\cal H}_\nu)=2^{3/\lambda} (2/3)$.
Now, the Taylor expansion of $\Phi_\lambda$ for $r$ near zero gives
\[\Phi_\lambda[r]=2^{3/\lambda}\frac{2}{3}\l(1+r^2/5 + o(r^2)\r)\]
so that $\Phi_\lambda$ is increasing for $r$ near zero, and now, we can deduce that $\Phi_\lambda$ takes its maximum for ${\cal H}_\lambda$.
 
This ends the proof of \Cref{theo:optri}.

\subsection{A comment on the binary decomposition of hyperbolas }

A consequence of the proof of \Cref{theo:optri} (but it is also clear without the proof!)  is that each hyperbola  ${\cal C}_0={\cal H}_\lambda$ has a decomposition pair $(H, {\cal C}_1)$ where $H_0$ is the central height of ${\cal H}_\lambda$, and where ${\cal C}_1$ is another hyperbola ${\cal H}_{\nu}$. It is then interesting, as much as it was important in the construction of the present paper, to express clearly the link between $\lambda$ and $\nu$. As usual it is easier to rather formulate the link between $r_\lambda$ and $r_\nu$.

We have   $H_0= \frac{\cosh(\rl)}{2\cosh(\rl/2)^2}$. Therefore since 	$H_1=(1 - \sqrt{-2H_0 + 2})/(2H_0 - 1)$,
\begin{align*}
	H_1&=\frac{\cosh(\rl/2)^2-\cosh(\rl/2)\sqrt{2\cosh(\rl/2)^2-\cosh(\rl)}}{\cosh(\rl)-\cosh(\rl/2)^2} =  \frac{\cosh(\rl/2)}{2\cosh(\rl/4)^2}
\end{align*}
Hence
\begin{lem}\label{lem:qhtdfd}
   If ${\cal C}_0$ is ${\cal H}_\lambda$, with ``$r$-parameter'' being $r_\lambda$, then ${\cal C}_1$ has $r$ parameter $\rl/2$, and if we keep decomposing,  ${\cal C}_i$ has $r$-parameter $\rl/2^i$. As a consequence, ${\cal C}_n \to {\cal P}$ (in the sense that $d_H({\cal C}_n,{\cal P})\to 0$).
 \end{lem}
 \begin{proof}The first statement follows the discussion preceding the Lemma, and the second one can be proved using the parametrization $(X,Y)$ of the hyperbolas (whose $\lambda$ or $r$-parameter goes to zero). 
 \end{proof}
This result allows us to prove the product form of $\L({\cal H}_\lambda)$ given in \Cref{eq:Lr}. 
Eq. \eref{eq:qsdqd} yields $\L({\cal C}_i)=2 ((1-H_{i})H_{i}/2)^{1/3} \L({\cal C}_{i+1})$. 
 Denoting by $\ell(R)$ the affine length of the hyperbola with $r$-parameter $R$, and $h(R)= \frac{\cosh(R)}{2\cosh(R/2)^2}$ its central height,
we then have $\ell(R)\to 2$ when $R\to 0$, and 
\[\ell(R)=(2 (h(R)(1-h(R))/2)^{1/3}) \ell(R/2)=\prod_{k\geq 0} 2\big(h(R)(1-h(R))/2\big)^{1/3}. \]  

Since $(\cosh(x)+1)^2=4\cosh(x/2)^4$, we get 
\begin{align*}
    \ell(R)&= 2\prod_{k\geq 0} 2 \l(  \frac{\cosh(R/2^k)}{2(\cosh(R/2^k)+1)^2}\r)^{1/3}
    =\frac{2\cosh(R)^{1/3}}{\prod_{k\geq 1}\cosh(R/2^k) }.
\end{align*}
It remains to prove that $\prod_{k\geq 1}\cosh(R/2^k)={\sinh(R)}/{R}$
but this follows from the simple following observation:
\[\sinh(R)=2 \sinh(R/2) \cosh(R/2)=4\cosh(R/2) \cosh(R/4) \sinh(R/4)=2^K\sinh\l(R/2^{K}\r)\prod_{j=1}^K \cosh\l(R/2^j\r).\]
and $\sinh(x)/x\to 1$ as $x\to 0$.



\subsection{Proof of \Cref{pro:sdgr}}
\label{sec:ProofSignCons}

$\bullet$ Proof of $(i)$.   Take an hyperbola ${\cal H}_\alpha$, and $C_1$ and $C_2$ two connected subsets of this curve.
Denote by $\ell_i$ (resp. $r_i)$ the leftmost (resp. rightmost)  point of $C_i$. For a point $z=(x,y)$ in $\R^2$ denote by $\pi_1(z)=x$ the $x$-axis projection.
Up to renaming $C_1$ and $C_2$, we may assume that
$\pi_1(\ell_1)<\pi_1(\ell_2)$ and since $C_1$ and $C_2$ have enveloping triangle $T_1=\ell_1r_1c_1$ and $T_2=\ell_2r_2c_2$ (for some points $c_1$ and $c_2$) with same area, we have also
$\pi_1(r_1)<\pi_1(r_2)$.

Now, take the enveloping triangle $\ell_1r_2E$ of the part of  ${\cal H}_\alpha$ between $\ell_1$ and $r_2$ (which is also the smallest enveloping triangle containing both $C_1$ and $C_2$). Observe the action of the $\ell_1r_2E$ symmetry $\phi_{\ell_1r_2E\to r_2\ell_1E}$  over the pair $(C_1,T_1)$. Clearly, it sends $C_1$ onto a part $C'_1$ of ${\cal H}_\alpha$ with extremities $(\ell',r_2)$ that is, with same right extremity as $C_2$, and it sends the enveloping triangle $T_1$ of $C_1$ onto $T'_1$ the enveloping triangle of $C'_1$. This implies that $\A(T_1)=\A(T'_1)=\A(T_2)$ (this second equality being the hypothesis). Therefore both $C_2$ and $C_1'$ have same left right extremities, and enveloping triangle with same area: we deduce $C_1'=C_2$. This concludes the proof, since the determinant of (the matrix of) $\phi_{\ell_1r_1c_1\to r_1\ell_1c_1}$ is $-1$, so that $\L(C'_1)=\L(C_1)$, $\A(\CH(C_1))=\A(\CH(C'_1))$.

The same argument (relying on the properties of $\phi_{\ell_1r_1E\to r_1\ell_1E}$ to transport at the same times a sub-curves and its enveloping triangle) allows to prove $(i)$, $(ii)$ and $(iii)$.

$\bullet$  Proof of $(iv)$. By $(i)$, $(ii)$ and $(iii)$ if a pair of values of the type of those $(a)$, $(b)$ or $(c)$, the knowledge of the exact position of $(C_1,T_1)$ is irrelevant: we may assume that $\ell_1=A$ (the vertex of $ABC$).

Denote by $P_{s,v}$ the part of ${\cal H}_v$ in between  $A=(0,0)$ and  $(2- X_s,Y_s)$ (we take $2-X_s$ to ``have'' a parametrization starting at zero).
Denote by $T_{s,v}$ the enveloping triangle, $L(P_s)$ the affine length of the part, and 
$\A(P_{s,v})=\A({\cal H}_v \cap T_s)$.
We have, for the function $f:h\mapsto {\cosh(h)}/{\sinh(h)^3}$ defined in \eref{eq:fh}
\be
\A(T_{s,v})&=& f(r_v)/f(r_v s)\\
\A(P_{s,v})&=&\int_0^s \bar{X}'_uY_udu - \bar{X}_sY_s/2 = \frac{f(r_v)}2(\sinh(2r_vs)-2rs)\\
\L(P_{s,v}) &=& s \L_v= 2s r_v f(r_v)^{1/3}
 \ee
To compute these latter formulas, write $V= \bma \bar{X}'_s\\Y'_s\ema$ the tangent vector at $z:=\bma \bar{X}_s\\Y_s\ema$. Denote by $M$ the point on the line $y=x$ and on the line $\bma \bar X_s \\Y_s\ema - \alpha V$ (indexed $\alpha$). We then compute $\alpha$, then $M$, then the area of the triangle $zMA$ using the standard determinant formula.
 
Now, it suffices to observe that any pair taken in $(\A(T_{s,v}),\A(P_{s,v}),\L(P_{s,v}))$ allows to recover uniquely $(s,v)$, and this is a simple exercise.

\subsection{A short discussion about hyperbolas, and proof of \Cref{eq:hyper}}
\label{sec:SDH}
 
The proof of \Cref{eq:hyper} is mainly and exercise. For example, starting from $(X,Y)$ given in \Cref{theo:lim}, we can compute the area using the classical formula $\A({\cal H}_\lambda)=\int_0^1 X'(t)Y(t)dt$, and the affine perimeter using the standard formula for the curvature:
\[\kappa(t)=\frac{Y''(t)X'(t)-X''(t)Y'(t)}{(X'(t)^2+Y'(t)^2)^{3/2}}.
\]
Letting $\sigma(t)$ the length of the curve between time $0$ and $t$, we get
\[\sigma(t)=\int_0^t \sqrt{ X'(u)^2+Y'(u)^2} du\imp d\sigma(t)=\sqrt{ X'(t)^2+Y'(t)^2}dt,\]
so that 
\beq
\L_\lambda=\int_0^1 \kappa(t)^{1/3} d\sigma(t)= \int \l({Y''(t)X'(t)-X''(t)Y'(t)}\r)^{1/3} dt.\eq
Since ${Y''(t)X'(t)-X''(t)Y'(t)}$ is constant and equal to $\frac{32 \rl^3 \cosh(\rl)}{(\sinh(3\rl) - 3\sinh(\rl))}$, and since $\sinh(3\rl)-3\sinh(\rl)=4\sinh(\rl)^3$, we get the announced formula.

On the way, observe that we are working with the very special case where $Y''(t)X'(t)-X''(t)Y'(t)=\L_\la^3$ is constant under its ``natural parametrization'' coming from the limit of $(X^{(n)},Y^{(n)})$.
For a parametric function $(x,y)$, the constancy of $y''(t)x'(t)-x''(t)y'(t)$  implies that $x'y'''-x'''y'=0$ so that $(x''',y''')=-\mu(x',y')$. With the sign of $\mu$ one can then identify the curve parameterized by $(x,y)$: either $\mu=0$ and this is a parabola, or $\mu>0$, and this is an ellipse, or $\mu<0$ and this is an hyperbole, which is the case here since we find $\mu=-4\rl^2$. 

If one applies an invertible affine map to an hyperbole, seen as parametric function,  
$ T:\bma {x}\\y  \ema \to M \bma {x}\\ y \ema +V$
then for 
$\bma \bar{x}(t) \\ \bar{y}(t)\ema :=V+ M \bma  {x}(t) \\  {y}(t)\ema$,
we get $  { \bar y''(t) \bar x'(t)-  \bar x''(t) \bar y'(t)} =\det(M) \big({y''(t)x'(t)-x''(t)y'(t)}\big)$, so that under the inherited parametrization we retrieve the special property $ { \bar y''(t) \bar x'(t)-  \bar x''(t) \bar y'(t)}={y''(t)x'(t)-x''(t)y'(t)}$ constant.  

We send any reader interested by this kind of considerations to Bu Chin \cite{MR724783},  Sapiro and Tannenbaum \cite{ST}.

\begin{figure}
    \centering
    \includegraphics[height=5cm,width=8cm]{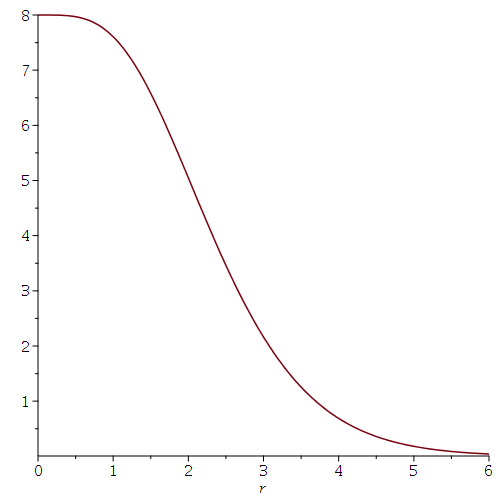}
    \captionn{The function $\L(R)^3$}
    \label{fig:my_label3}
\end{figure}

In order to get an intrinsic definition of a parametric curve ${\cal C}$ (drawn in the standard coordinate system)  one can use the affine length parametrization, which is defined intrinsically up to the choice of an orientation and an origin: this choice allows to define $(x_t,y_t)$ such that the affine length of the part $P_t$  between $(x_0,y_0)$ and $(x_t,y_t)$ is $t$. Giving the area of the enveloping triangle of the set $\{(x_s,y_s),s\in[0,t]\}$, allows to find the signature $v$ (and we must have $L(P_t)=t = \L_v (t/\L_v)$ which allows to reconstitute a parametrization $(X_s,Y_s)$ such that  $Y''_tX'_t-Y'_tX''_t=\L_v$ if needed, using the considerations of the previous section).  
 \color{black}

\section{Optimization of $\PLK$ in a general compact convex set $\K$}

\subsection{Proof of \Cref{pro:hss}$(d)$} \label{sec:qsdyjsd}
 \label{sec:qsdyjsd2}


 In the bi-pointed case, the notion of cupola symmetry is the main geometrical argument in the proof that ${\cal H}_\lambda$ maximizes $\Phi_\lambda$. This argument may actually be reformulated as follows: for a triangle $abc$ enveloping a portion $p$ of the curve ${\cal C}$, if $\phi_{abc\to bca}(p)\neq p$, then $p$ can be replaced by another curve
 $p':=(p+\phi_{abc\to bca}(p))/2$, and then $\L(p')\geq \L(p)$ as well as $\A(p')>\A(p)$.

 This argument is valid in the bi-pointed case, that is, in $\CCST$, and remains true in ${\sf CCS}_\K$ for any $\K$, with a caution: we need $p$ and $\phi_{abc\to bca}(p)$ to be both included in $\K$, and for this, a sufficient condition is that the enveloping triangle $abc$ is included in $\K$; indeed, there exists a curve $p$ with enveloping triangle $abc$ intersecting $\partial K$, and for which  $\phi_{abc\to bca}(p)$ intersects the complement of $\K$ in the plane:  the local symmetrization does not conserve the property of being a subset of $\K$ in general.

 Hence, if $C$ is in $\argmax \PLK$, this argument can be applied to any portion $c$ of a connected component of $(\partial C)\setminus \K$ to prove that this connected component is an hyperbola: indeed the restriction of $c$ inside any enveloping triangle small enough to be included in $\K$ must be an hyperbola, otherwise it can be cupola-symmetrized while increasing $\PLK$.

We present now the final argument for the proof of \Cref{pro:hss}$(d)$, which allows to symmetrize simultaneously inside two enveloping triangles -within the same component or not- of $(\partial C)\setminus \K$. These two triangles are like a cat's two ears:
\begin{theo}[The cat's ears theorem]\label{theo:TCET}
  Assume that ${\cal C}$ is an element of $\argmax \PLK$, and that $P_1$ and $P_2$ are two (possibly intersecting) connected subsets of ${\cal C}\setminus \K$ such that
   the enveloping triangles $T_1$ and $T_2$ of $P_1$ and $P_2$ (respectively) exist, have same area, and are both totally included in the interior of $\K$. \par
  In this case, we have $(\A(\CH(P_1)),\L(P_1)) = (\A(\CH(P_2)),L(P_2))$.
\end{theo}
\begin{figure}[h!]
\centerline{\includegraphics{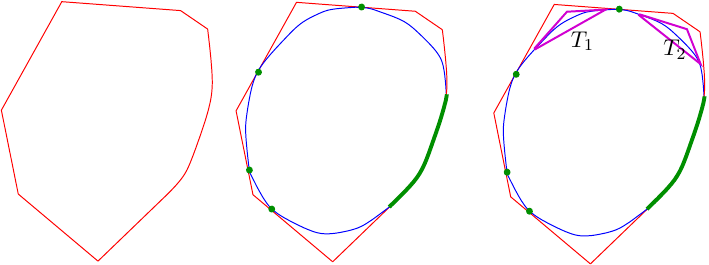}}
\caption{\label{fig:2ears} On the first picture a compact convex set of $\R^2$, on the second one in blue and green the boundary of a compact convex subset ${\cal C}$ of $\K$. The green part is $\partial{\cal C}\cap \partial\K$, while the blue part is $\partial {\cal C}\setminus \partial K$. On the third figure, two enveloping triangles of the curve that do not contain contact points. Both triangles have same area. These triangles must contain two portions $p_1$ and $p_2$ of ${\cal C}$ that have same affine perimeter, and each of these triangles must contain the same area  of ${\cal C}$. }
\end{figure}
\begin{proof} Choose ${\cal C}$ in $\argmax \PLK$. By the discussion above, we know that the curve corresponding to the connected components of $(\partial C)\setminus \K$ are hyperbolas.
  Decompose $\A ({\cal C})$ as $\A (\CH(P_1))+A (\CH(P_2))+AO$ where $AO$ is the area outside these two regions, and $\L({\cal C})=\L O+ \L(P_1)+\L(P_2)$ where again, $\L O$ is the affine length of ${\cal C}$ carried by the complementary of $P_1$ and $P_2$. One then sees that $(\A({\cal C}),\L({\cal C}))=(\A({\cal C}'),\L({\cal C}'))$ if one obtains ${\cal C}'$ by again, some simple curve surgery:\\
  -- removing the curve $P_1$ and replace it by $\phi_{T_2\to T_1}(P_2)$,\\
  --  remove $P_2$ and replace it by $\phi_{T_1\to T_2}(P_1)$.
  \par
  What has to be noticed here, is that since both affine map have determinant 1, they preserve affine perimeter and area, in the sense that $\A(\CH(\phi_{T_1\to T_2}(P_1)))=\A(P_1)$ (the same goes for $P_2$).
  
  It is easy to check that ${\cal C}'$ is a compact convex set. But now, we see that if ${\cal C}'\neq {\cal C}$, that is, if the curve surgery modified the global curve, then the connected component of $\partial C \backslash \partial K$ that contains $P_1$ is no more an hyperbola. This implies that ${\cal C}'$ is not optimal, and since $\PLK({\cal C}')= \PLK({\cal C})$ this prevents ${\cal C}$ from being optimal too. A contradiction. 
  \end{proof}
  The cat's ears theorem is somehow an infinitesimal property, in the sense that if $C$ is not signature-homogeneous (and is different from $\K$), for all $\epsilon>0$, it is possible to find two enveloping triangles $T_1$ and $T_2$, with same areas $\leq \epsilon$, included in $\Int \K$,  of some portions $C_1$ and $C_2$ of connected components of $(\partial C)\setminus \K$, that are not image of each other by an affine map with determinant $1$; as a consequence, $C$ is not in $\argmax \PLK$.

{\bf This ends the proof of \Cref{pro:hss}$(d)$}: the connected components of $(\partial C)\setminus \K$, when they exist, are hyperbolas. The proof of the cat's ears theorem implies that if one takes two enveloping triangles $t_1$ and $t_2$ of any portion of these connected components small enough to be included in $K$, and with same area, then $(\partial C)\cap t_1$ and $(\partial C)\cap t_2$ are image of each other by an affine map with determinant $1$. Therefore, these hyperbolas have same signature.

    \subsection{Optimization of $\PLK$ in the regular $\kappa$-gone: proof of  \Cref{pro:Reg}}
{        \begin{lem}\label{lem:max}
        Let $u>0$ and $x\geq 0$ be fixed. The map \[\B:v\mapsto  \L({\cal H}_v)^u (x+\A({\cal H}_v))\] reaches its maximum for a single value $v=v^\star$. Moreover $\frac{\partial}{\partial v}\B(v)$ cancels only at $v=v^\star$.
\end{lem}
\begin{proof}
We let 
\be  g(r)= \l( \frac{2r \cosh(r)^{1/3}}{\sinh(r)}\r)^{u}\l( x+ r\frac {\cosh \left(r \right)  }{ 
\sinh \left( r\right)  ^{3}} \left( {\frac{\sinh \left( 2\,r \right)}{2r}}-1\right)\r) ,
\ee
so that $g(r_v)= B(v)$. 
We have $B'(v)=g'(r_v)r'_v$, and $r'_v$ does not cancel while $r\mapsto g'(r)$ cancels once, at some ${r}^\star$. 
Observe that $\lim_{v\to +\infty}B(v)=0$ (since $\lim_{v\to +\infty}\L({\cal H}_v)\to 0$), and $\lim_{v\to 0} B(v)=2^{u} (x+2/3)$.
Now, the Taylor expansion of $g(r)$ for $r$ near zero gives
\[g(r)=2^u(x+2/3)+{2^{1+u}r^2}/{15}+o(r^3),\]
so that $g$ is increasing for $r$ near zero, and since $v\mapsto r_v$ is increasing too, $B$ is also increasing near $0$. We can now deduce that $g$ reaches its maximum at ${r}^\star$, and further that $B$ reaches its maximum at $v^\star$ only, where $v^\star$ solves $r_{v^\star}=r^\star$. 
\end{proof}
\subsubsection*{Proof of  \Cref{pro:Reg}}
By  \Cref{pro:hss}, we know that any $C\in\argmax \PLK$ (where $\K={\sf Reg}(\kappa)$) is invariant by the symmetry $(x,y)\to (-x,y)$ and invariant by rotation of angle $2\pi/\kappa$. It suffices then to describe $C$ in a triangle $m_{i-1}m_{i}w_i$. Since between contact points, we must have hyperbola, there are two possibilities: either the hyperbola in the triangle $m_{i-1}m_{i}w_i$ intersects $[m_{i-1},w_i]$ at $m_{i-1}$, or at a point $z_i\in (m_{i-1},w_i)$. The cupola symmetrization argument in the triangle $m_{i-1}z_ip_i$ where $p_i$ is the middle point of the hyperbola,  allows one to rule out this second possibility.
The hyperbola in  $m_{i-1}m_{i}w_i$ is the image by the affine map $\phi_{ABC\to m_{i-1}m_{i}w_i}$ of an hyperbola of our family $({\cal H}_v)_v$. This affine map has determinant $b_k$ (see \eref{eq:bk}).

Now take the hyperbola ${\cal H}_v$ drawn in $ABC$, we then search for
\[\argmax\l(v\mapsto \big(a_\kappa+\kappa b_\kappa \A({\cal H}_v)\big) \l(\kappa b_\kappa^{1/3} \L({\cal H}_v\r)^{3/\lambda}\r)\]
which coincides with $\argmax\l(v\mapsto  F_\lambda(r_v)\r)$
with $F_\lambda(v)=\big(\frac{a_\kappa}{\kappa b_\kappa }+\A({\cal H}_v)\big)  \L({\cal H}_v)^{3/\lambda}$. By \Cref{lem:max}, the argmax contains a single element. 
We \underbar{guess} that $v$ is characterized by the following equation:
\beq\label{eq:dqgdqhte1} v=\Psi(r_v) = \kappa \frac{\lambda }{a_\kappa/c_{v,\kappa}+\kappa }\eq  
where \beq\label{eq:dqgdqhte2} c_{v,\kappa}=  \A({\cal H}_v)b_\kappa = r_v f(r_v)\Psi(r_v)b_\kappa  \eq
(where $f$ is given in \eref{eq:fh} and $\Psi$ in \eref{eq:Ps}) in other words, after simplifications, we claim that $v$ solves the equation given in \eref{eq:dqdu} (written in terms of $r_v$). 
Now, in order to prove the guess, by \Cref{lem:max}, it suffices to
prove that
$\frac{\partial}{\partial v} F(v)$ cancels exactly when \eref{eq:dqdu} is satisfied.
Now, taking into account that $\A({\cal H}_v)=r_v f(r_v)\Psi(r_v)$, $L(r_v)=2r_vf(r_v)^{1/3}$, and $F(v)=G(r_v)$ with
\[G(r)= \l(\frac{a_\kappa}{\kappa b_\kappa }+r f(r)\Psi(r)\r) \l( 2rf(r)^{1/3}\r)^{3/\lambda}\]
it is from here just a simple exercise to prove that
$G'(r_v)=0$ when $v$ solves \eref{eq:dqdu}. By \Cref{lem:max}, the maximum of $\Phi_\lambda^K$ is indeed reached if $r=r_\lambda$.
\begin{rem}[Where the guess comes from?]
  The idea is to find an equation which would characterize the limit shape $C$ of $\CH(U[n+\fnl])$ under $\QK_{n,\fnl}$, inside one of this small triangle $\triangle_i:=m_{i-1}m_iw_i$ (with area $b_k$). We know that the optimizing curve between contact points are hyperbola, so we need to find the parameter $v$ of ${\cal H}_v$. By symmetry, we know that each $\triangle_i$ will receive around $n/\kappa$ points of the boundary points of  $\CH(U[n+\fnl])$, and the $n\lambda$ interior points will be shared between $\CH(\{m_0,\cdots,m_{\kappa-1}\})$  (with area $a_\kappa$) and the $\kappa$ domains below the curve, that is $\Delta_i \cap C$. If one denotes by ${\alpha}$ the surface of one of this domain, we expect that around
  \beq\label{eq:eg} {\alpha} n\lambda/(a_\kappa+\kappa{\alpha})=nf\eq interior points be present in  $\Delta_i \cap C$. But our limit shape theorem says that under $\bQt_{n/\kappa, nf}$, the limit hyperbola  is ${\cal H}_{\kappa f}$, with area (in $ABC$) $\A({\cal H}_f)$, and then in $\Delta_i$, $\A({\cal H}_{\kappa f})b_\kappa$. Finally,  at the end we solve \eref{eq:eg} together with
  \[ {\alpha}=\A({\cal H}_{\kappa f})b_\kappa,\]
which is equivalent to our guess.
  \end{rem}
\color{black}

\section{Limit shape in the general case}
 
\subsection{Proof of Theorem \ref{theo:compa} }

Denote by ${\sf Area}_{n,m}$ a random variable distributed as the area $\A(\CH(U[n+m]))$ under $\QK_{n,m}$. 

\begin{lem}\label{lem:sto} For all $n\geq 1$, $m\geq 0$, 
  \beq\label{eq:qs} \E\l({\sf Area}_{n,m+1}\r) \geq \E\l({\sf Area}_{n,m}\r),\eq
  and moreover, ${\sf Area}_{n,m+1}$ is larger than ${\sf Area}_{n,m}$ for the stochastic order, which means that for all $x$, \[\P({\sf Area}_{n,m+1}\geq x)\geq \P({\sf Area}_{n,m}\geq x).\] 
  \end{lem}
\begin{proof}
Denote by ${\sf C}^{\K}(n)$ the subset of $\K^n$ formed by the $z[n]$ that are in a convex position, and define ${\sf C}_{n,m}^\K:=\Bigl\{(z[n],w[m])~:~ z[n]\in {\sf C}^{\K}(n), w_1,\cdots,w_m \in \CH(z[n])\Bigl\}$.
	We have
	\ben
	\P( \Area_{n,m} \in da ) &=& \gamma_{n,m} \int_{z[n]\in{\sf C}^\K(n)} a^m\,\1_{\A(\CH(z[n]))\in \d a}\d z_1\cdots \d {z_n}\\
	&=&  \frac{  a^m \mu_n(\d a)}{\int_{0}^1 b^m \mu_n(\d b) },
	\een
	where $\gamma_{n,m}$ is the only constant such that $\P( \Area_{n,m} \in da )$ is a probability distribution 
	and \[\mu_n(\d a)= \int_{z[n]\in{\sf C}^\K(n)}\,\1_{\A(\CH(z[n]))\in \d a}\d z_1\cdots \d {z_n}.\]
	Hence, there exists a sequence of constants $(\nu_{n,m})$, such that
	\beq\label{eq:geht} \P( \Area_{n,m+1} \in da ) = {\nu}_{n,m} \;a\;\P( \Area_{n,m} \in da ). \eq
	By integrating the function $a\mapsto 1$ on either side of  \eref{eq:geht} we get  $1= {\nu}_{n,m} \E(\Area_{n,m} )$, and next, by integrating $a\mapsto a$, we get
	\be
	\E ( \Area_{n,m+1})&=&{\nu}_{n,m}\;\E(\Area_{n,m}^2)
	>{\nu}_{n,m}\;\E (\Area_{n,m} )^2
	\geq \E (\Area_{n,m} ),
	\ee
	which  ends the proof. The inequality for the stochastic order is implied by the fact that the quotient of densities  $a\mapsto  \P( \Area_{n,m+1} \in da )/\P( \Area_{n,m} \in da )$ is a non-decreasing function of $a$ (see e.g. \cite{MR4568490} for a proof of this fact and more). Of course this inequality for the stochastic order implies also \eref{eq:qs}.
\end{proof}


\paragraph{Proof of \Cref{theo:compa}}

For a sequence of (deterministic) compact convex subsets $({\cal C}_n)$ of $\K$, we have $d_H({\cal C}_n,\K)\to 0
\iff  \A({\cal C}_n)\to \A(\K)$. It is also simple to see that if the $({\cal C}_n)$ are now random, then
\[ d_H({\cal C}_n,\K)\proba 0 \iff \E(\A({\cal C}_n)) \to \A(\K).\]
On the one hand, we have by hypothesis $\CH(U[n+\fnl])\proba \K$ under $\QK_{n,\fnl}$, which is equivalent to $\E(\Area_{n,\fnl})\to A(\K)$. On the other hand, since $\E(\Area_{n,\fnl})\leq \E(\Area_{n,\floor{n\lambda'}})\leq \A(K)$ (by Lemma  \ref{lem:sto}), it means that $\E(\Area_{n,\floor{n\lambda'}})\to \A(\K)$ from what we deduce that $\CH(U[n+\floor{n\lambda'}])\proba \K$ under $\QK_{n,\floor{n\lambda'}}$, and ends the proof.

\begin{cor} If $\K$ is a disk  for all $\lambda\geq 0$, under $\QK_{n,\fnl}$, $\CH(U[n+\fnl]) \proba \K$.
  \end{cor}
  \begin{proof} By \Bara, under $\QK_{n,0}$, $\CH(U[n]) \proba \K$ (see also \Cref{cor:dqfqe}).
    \end{proof}

\subsection{About the evaluation of $\bQ_{n,m}^{\K}$ and a potential limit shape theorem}
\label{sec:AbC}
Take a convex compact set of $\R^2$ with area 1.
Denote by ${\sf C}_{n,0}^{\K}$ the subset of $\K^n$ formed by the $z[n]$ that are in convex position.
For $n\geq 1$, $m\geq 0$ denote by ${\sf C}_{n,m}^{\K}$ the subset of $\K^n\times\K^m$ of pairs $(z[n],w[m])$ such that $z[n]\in {\sf C}_{n,0}^{\K}$, and $w_1,\cdots,w_m \in \CH(\{z_1,\cdots,z_n\})$. Set \[c_{n,m}^{\K}=\Leb({\sf C}_{n,m}^{\K})=\int_{\K^n}  \Leb({\CH \{z_1,\cdots,z_n\}})^m  dz_1\cdots dz_n\]
so that $\bQK_{n,m}$, the probability that the number of vertices of the convex hull of $n+m$ iid uniform points in $\K$ is $n$, satisfies $\bQK_{n,m}= \binom{n+m}n c_{n,m}^{\K}$. Computing ${\sf C}_{n,0}^{\K}$ or $\bQK_{n,0}$ are then equivalent problems.

In this section, we provide some elements about \Cref{theo:conv1} and explain the obstructions that prevent us from completing the proof; we tried to be specific enough, so that an interested reader could engage with this problem from this section, without too much effort. To explain a bit this conjecture, we will need to discuss the concluding approach of B{\'a}r{\'a}ny \cite{barany2} in the case $\QK_{n,0}$, to make apparent the differences.

We will also need to discuss the link between his limit shape theorem under $\QK_{n,0}$ and the computation of  $\bQ_{n,0}^{\K}$ which  are deeply interconnected. Even if most of the ideas of the four next pages are due to B\'ar\'any, we discute them in a new light.

Let us start by a remark: 
\begin{lem}For all $\lambda\geq 0$, the sequence  $(\CH(U[n+\fnl])_{n\geq 0}$ is tight in  $CS(\K)$, equipped with the Hausdorff topology. 
\end{lem}
Indeed, the set ${\sf CCS}_\K$ (of compact subsets of $\K$) equipped with the Hausdorff distance $d_H$ is compact, and seen as a topological space $({\sf CCS}_\K,d_H)$ is complete and  separable (it is a Polish space). Hence, as a consequence of  Prokhorov theorem, see e.g. Billingsley \cite[Theo.5.1]{MR1700749}, any sequence of probability measures $(\mu_n,n\geq 0)$ on $({\sf CCS}_\K,d_H)$ is tight, and, further the sequence $(\mu_n,n\geq 0)$ is relatively compact, meaning that it contains a sub-sequence converging weakly: in our case, it implies the following general statement:
\begin{lem}
  For any sequence of non-negative integers $(a_n)$, and $\CH(U[n+a(n)])$ taken under $\QK_{n,a(n)}$, there exists a sub-sequence $(n_k)$ such that $(U[n_k+a(n_k)])$ converges in distribution for the Hausdorff topology.
\end{lem}
This result is ``trivial'' in the sense that it is a consequence of compactness and it is then valid for all sequences $(a(n))$. It allows us to concentrate on the difficult point: to prove a limit shape theorem, it is necessary and sufficient to prove the uniqueness of the accumulation point. An accumulation point is a measure on ${\sf CCS}_\K$ that can be a Dirac mass on a single convex domain but could also be a more general probability distribution on ${\sf CCS}_\K$.
Under $\QK_{n,0}$, B\'ar\'any states the convergence in distribution toward a deterministic shape, which is the only element $C^\star$ of $\argmax \Phi_0^\K$.  As usual in probability theory, when the limit law is a Dirac mass, there are additional ways to prove convergence, in particular concentration arguments. In this case, it amounts to proving that, 
\[\textrm{ for all }\eps>0,~~\QK_{n,\fnl}\l(d_H(\CH(U[n]),C^\star)\geq \eps\r)\to 0.\]

However, it is hard to precisely compute $\QK_{n,\fnl}\l(d_H(\CH(U[n]),C^\star)\geq \eps\r)$ since even the global volume of points in convex position, that is, the normalizing constant $c_{n,0}^\K$, can not be computed exactly. B\'ar\'any shows that the exact value is not important because, somehow, all volumes and probability that we need to evaluate in order to conclude, have the form either $\exp( n f(g))$ or $\exp(-2n\log(n)+nf(g))$ when $n\to+\infty$, where in this formula $g$ represents a geometric property, and $f$ a function of this geometry. These formulas holds up to a factor $\exp(o(n))$ and in general. Because all quantities have different exponential order (the coefficients of the linear terms in $n$ in the exponentials are different), at the end, we can only care about the geometry $g$ that maximizes $f(g)$, and don't care about sub-linear terms in $n$ in the exponential.


The main ideas in B\'ar\'any's proof are the following: 

First, the bi-pointed triangular model $\bQt_{n,0}$ is the source of everything, going from the limit shape theorem, passing  through  the asymptotic behavior of $\QK_{n,0}$, to  the appearance of the affine perimeter in the formula of  $\QK_{n,0}$.

To see this, take some integer $d\geq 3$,  and consider ${\cal R}_{d}$ the set of convex equiangular polygon $R$ with $d$ sides, where the $\bar \j$ th side direction is  $\exp(2i \pi  \bar \j/d)$ (in this paragraph $\bar \j$ is taken in $\Z/d\Z$ so that  $\bar \j= \bar \j+d= \bar \j-d$).

Let $R(z[n])$ be the smallest equiangular polygon $R$ with $d$ sides containing $z[n]$ (it is well defined as an intersection). Let us call this polygon, the enveloping $d$-gone of $z[n]$.

Hence for $d\geq 3$ fixed and $R\in{\cal R}_{d}$, let ${\sf C}_{n,0}^\K(R)= \l\{z[n] \in {\sf C}_{n,0}^\K~:~R(z[n])=R\r\}$ be the set of $z[n]$ having $R$ as enveloping $d$-gone. Since the ${\sf C}_{n,0}^\K(R)$ are disjoints, and since
\[{\sf C}_{n,0}^\K = \bigcup_{R\in {\cal R}_{n,d}} {\sf C}_{n,0}^\K(R)\]
we may use this decomposition.

Since we want to integrate on the points coordinates, observe that $R[z[n]]= R$, if and only if, 
	\begin{enumerate}
		\item for each $\bar \j$, the $\bar \j^{th}$ side of $R$ contains one point of the set $\{z_1,\cdots,z_n\}$ (the so-called ``contact point'' $c_{\bar\j}$),
		\item the set of triangles, corresponding to the (closure of the) connected components of $R \setminus \CH(\{c_k, 0\leq k\leq d-1\})$ contains all the $z_i$,
		\item the $z_i$ that are contained in each of these given triangles form a convex chain in there.
	\end{enumerate}

To be a bit more specific, denote by $ v_{\bar \j}$ the $j$th vertex of $R_{n,d}$, with such a labeling such that $c_{\bar \j}$ is on the segment $[v_{{\bar \j}-1},v_{\bar \j}]$. Denote by $t_{\bar \j}$ the triangle with vertices $c_{{\bar \j}}v_{\bar \j}c_{{\bar \j}+1}$. Now, let $C_{d,k}$ be the set of possibles values for $(c_0,\cdots,c_{d-1})$ that are filtered according to the cardinality $k\geq 3$ of the set $\{c_0,\cdots,c_{d-1}\}$. This number $k$ can be smaller than $d$, notably when some sides of the enveloping regular $d$-gones have zero length, but also when the contact points are vertices of this enveloping $d$-gone.

The set $C_{d,k}$ has positive Lebesgue measure in $(\R^2)^k$, and we won't need to say more on this set. Further $c_{n,0}^K$ can be written as
\ben\label{eq:cn0} c_{n,0}^\K=\sum_{k=3}^d \int_{C_{d,k}} \frac{n!}{(n-k)!}\sum_{(s_1,\cdots,s_k)~:~\sum s_i = n -k}  \binom{n-k}{s_1,\cdots,s_k}\prod_{i=1}^k   c_{s_i,0}^{\bullet,\bullet, t_i\cap \K}   dc_1,\cdots,dc_k\een
where:\\
-- $\frac{n!}{(n-k)!}$ accounts for the number of possibilities of indices $(i_1,\cdots,i_k)$ of the $z_i$ that are contact points,\\
--  $\binom{n-k}{s_1,\cdots,s_k}$ accounts for the number of ways to partition the remaining $z_i$ in the $k$ triangles,\\
--  and $c_{s_i,0}^{\bullet,\bullet, t_i\cap \K}$ is the Lebesgue measure of the set of convex chains formed by $s_i$ points in $t_i\cap \K$, and this chain must be convex with the two contact points, vertices of $t_i$.

An upper-bound $UB_{\K}^{(d)}$ is obtained if one replaces in this formula  $c_{s_i,0}^{\bullet,\bullet, t_i\cap \K}$ by $c_{s_i,0}^{\bullet,\bullet, t_i}$, and a lower bound $LB_\K^{(d)}$ is obtained by integrating only on the subset of contacts points $C'_{d,k}$, such that $t_i$ is included totally in $\K$ (or equivalently, this amounts to measure only the $z_i$ that are in convex position and having an enveloping $d$-done included in $\K$).

In both cases, the intersection $t_i\cap \K$ disappears from the considerations, and we can use the formula
\[c_{s_i,0}^{\bullet,\bullet, t_i\cap \K}=|t_i|^{s_i} c_{s_i,0}^{\bullet,\bullet, ABC}= \frac{|t_i|^{s_i}2^{s_i}}{s_i!(s_i+1)!},\] where $ABC$ is our favorite unit triangle.

 For these bounds, we will then use that the function to be summed and to be integrated in \eref{eq:cn0} is

\ben f_{c[d]}(s[k])=\frac{n!}{(n-k)!}  \binom{n-k}{s_1,\cdots,s_k}\prod_{i=1}^k   \frac{(2|t_i|^{s_i})}{s_i!(s_i+1)!}1_{\sum_{i=1}^k s_i=n-k} ,\een
in which it is implicit that the $|t_i|$ depends on the contact points $c_i$ that are here fixed (for the upper bound the integration is done on $C_{d,k}$ and for the lower bound on $C'_{d,k}$). 
 
 Now, we reach the final argument, for both the computation of $c_{n,0}^\K$ and the proof of the limit shape theorem. The connection with the limit shape theorem works as follows.

 If one takes a set of points $U[n]$ under $\QK_{n,0}$, then one can still define, for a given $d$, the enveloping regular $d$-gone of $U[n]$, and the contact points $(\bc_i,1\leq i \leq d)$ that are now random variables (and we will use the font $\bc_j,\bt_j,\bs_i$ to take into account the type of these objects).
The joint distribution of  $((\bc_j,0\leq j \leq k-1),(\bs_j,0\leq j \leq k-1))$ giving the contact points has a simple representation (and is non-zero for $3\leq k \leq d$ and $n$ large enough), and the number of elements $(\bz_j)$ in the interior of $\bt_j$ is:
\ben \label{eq:qgrgf}\QK_{n,0}( \bc_j \in dc_j,\bs_j = s_j, 0\leq j \leq k)=\frac{1_{c[d]\in C_{d,k}}}{c_{n,0}^{\K}}\frac{n!}{(n-k)!}\binom{n-k}{s_0,\cdots,s_{k-1}} \prod_{i=0}^{k-1} c^{t_j\cap K}_{s_j,0} dc_j.\een

{It turns out that understanding what the order of $c_{n,0}^\K$ is, or what the contact points $\bc[d]$ under $\bQ_{n,0}^\K$ are more likely to be, or what pairs $(\bc[d],s[d])$ are most probable, are three equivalent problems. As we will see, configurations that are most probable have a much heavier weight compared to the others (if one works up to a factor $\exp(o(n))$).}

For example, when the contact points ${c[k]}$ are fixed
\[f_{c[k]}= \max\l\{ f_{c[k]}(s[k])~:~\sum s_i=n-k\r\} =\exp(o(n)) \sum_{s:\sum s_i=n-k} f(s[k])\]
because there is only a polynomial number $O(n^k)=\exp(o(n))$ of elements $s[k]$ in the sum. Moreover, $\frac{n!}{(n-k)!}=\exp(o(n))$ too, since $k\leq d$ is bounded when we work with a fixed $d$.
Besides, using that for all $\frac{n^n}{e^{n-1}}\leq n!\leq \frac{n^{n+1}}{e^{n-1}}$, we may replace $s_i!$ by $\exp(s_i\log(s_i)-s_i)$ in $f_{c[k]}$, without losing more than an additional $\exp(o(n))$ factor, so that we can maximize, instead, the simpler function (equal up to a uniform $\exp(o(n))$ factor),
\[\tilde f_{c[k]}(s[k])= 2^n e^{-n} \exp\l(n\log n +\sum_{i=1}^k  3s_i -3s_i\log s_i  +s_i\log(|t_i|)\r)\] 
and now, proceed to the change of variable  $s_j =n \alpha_j$ such that $\sum \alpha_j=1$ to completely solve the optimization problem in $\alpha[k]$. 
We get 
\[\tilde f_{c[k]}(s[k])=2^ne^{-n} \exp\l(  3 n - 2n\log(n)+\sum -3n\alpha_j\log(\alpha_j)+n\alpha_j(\log |t_j|)\r).
\]
The value of the $\alpha[k]$ maximizing this function, can be computed using Lagrange multiplier technique: set $F((\alpha)_j,c[d])= \sum \alpha_i \log(|t_i|)-3\alpha_i\log(\alpha_i)  + y (\sum \alpha_i-1)$, cancel $\partial F/\partial x_i$, and $\partial F/\partial y$, and check that this is maximal for $\alpha_k=2|t_k|^{1/3}/S_d$ with $S_d=S_d(c[d])=2\sum |t_j|^{1/3}$ (where we have decorated this formula with a 2, because of the affine perimeter formula of a convex set being $\L(S)=2\lim \sum |T_i|^{1/3}$ for triangles constructed as the $t_i$ before). 
One then gets, 
\be
\max_s f_{c[k]}(s[d]) &=&    2^ne^{2n  - 2n\log(n)+o(n)} \exp\l(\frac{n}{S_d}\sum -6|t_k|^{1/3}\log(2|t_k|^{1/3}/S_d)+2|t_k|^{1/3}(\log |t_k|)\r)\\
&=&    2^ne^{2n  - 2n\log(n)+o(n)} \exp\l(\frac{n}{S_d}\sum 6|t_k|^{1/3}\log(S_d/2) \r)\\
&=&  2^ne^{2n  - 2n\log(n) + n\log(S_d^3)-3n\log(2)+o(n)}.
\ee
It remains to maximize in $c[k]$: the points $(c[k])$ that maximizes $f_{c[k]}$ are those maximizing $S_d(c[d])$. There is a continuity in the sense that if $c[d]$ maximizes $f_{c[k]}$, then for $\tilde{c}[k]$ close to $c[k]$, the quantity $S_d^3(\tilde{c}[k])$ is close to $S_d^3(c[k])$. This ensures that the complete integral and sum defining $c_{n,0}^\K$ has the same order (up to $\exp(o(n))$ factor), as $\max_{c[k]} f_{c[k]}$.

To end the optimization problem, it remains to prove that the upper and lower bound $UB_\K^{(d)}$ and $LB_\K^{(d)}$ coincide, and that
\ben\label{eq:riludqsd} \limsup_d \max_{c[d]} S_d(c[d])=  \max\{\L(C): C\subset \K\}.\een

For a given factor $f>0$, let $\K_f$ be the compact convex set with same center of mass as $\K$, and obtained by a dilatation of factor $1-f$ (so that $\A(\K_f)=(1-f)^2\A(\K)$. Of course
\ben\label{eq:quo} c_{n,m}^{\K_f}=(1-f)^{2(n+m)}c_{n,m}^{\K},\een
that we will use shortly, with $m=0$. It is easily seen that the lower bound $LB_\K$ we have (which amounts to integrating over $C'_{d,k}$, corresponding to the  $t_i$'s that are totally included in $\K$) satisfies
\[ c_{n,0}^{\K_f}\leq LB_\K^{(d)} \leq c_{n,0}^{\K}\]
for $d$ large enough, for $n$ large enough, since all the $z[n]$ is convex position in $\K_f$ will be enveloped by a $d$-gone totally included in $\K$ for $d$ large enough. By taking $f$ close to 1, it is apparent that the lower bound and $c_{n,0}^{\K}$ stay within an $\exp(o(n))$ factor, and this is true also for the upper bound, by the same reasoning.


The justification of \eref{eq:riludqsd} can be proved again by a compactness argument: first, by taking the $c[d]$ close to  $C^\star$, it appears that $ \limsup_d \max_{c[d]} S_d(c[d])\geq  \L(C^\star)=\max\{\L(C): C\subset \K\}$. If one finds a sub-sequence $c[d]$ (indexed by $d$) such that $ \max_{c[d]} S_d(c[d]) \leq a+\L(C^\star)$ with $a>0$, then by taking an accumulation point $\bar{C}$ of $\CH(c[d])$ along this sub-sequence (for the Hausdorff topology), then $\L(\bar C)\geq a+\L(C^\star)$, a contradiction. 

Finally, this gives that the maximum, and then the value of $\bQK_{n,0}$ satisfies
\[\bQK_{n,0}= {4^{-n}}e^{2n  - 2n\log(n) + 3n\log(\L(C^\star))+o(n)}\]
But this gives also a limit shape theorem: now we know $c_{n,0}^\K$ up to a $\exp(o(n))$ factor.  Take a convex domain $C'$ different from $C^\star$, then there exists $\epsilon>0$ such that $d_{H}(C',C^\star)\geq 2\epsilon$.  We can now evaluate the probability $\QK_{n,0}( \d_H(\CH(U[n]),C')\leq \eps)$, by computing again the Lebesgue measure of the corresponding $z[n]$, and by normalizing by $c_{n,0}^\K$. Since $\max\{\L(C):C\in B_{H}(C',\epsilon)\}<\L(C^\star)$, if one maximizes $\max f_{c[d]}$ for the $c[d]$ in $B_H(C',\epsilon)$, using the argument given above, one will find at the end a global weight
\[ 4^{-n}e^{2n  - 2n\log(n) + n \max_{C\in B_H(C',\epsilon)}\L(C)^3+o(n)}\] which is negligible in front of the total mass: it implies that $C'$ is not in the support of any accumulation point under $\QK_{n,0}$: only $C^\star$ is in the support of the limiting measure, which implies that under $\QK_{n,0}$, $\CH(U[n])\proba C^\star$ for the Hausdorff topology on ${\sf CCS}_\K$.

\paragraph{Construction of the conjectures in the case $\QK_{n,\floor{n\lambda}}$}

Let us review quickly the difference and common points between the computation of $c_{n,\floor{n\lambda}}^\K$, or the limit shape theorem under $\QK_{n,\floor{n\lambda}}$ and the case just treated, when $\lambda=0$. 
We have again the relation 
\[\bQK_{n,m}= \binom{n+m}n c_{n,m}^{\K}\]
and again we can write a sum of ``$d$-contact points'' (that are included in the $n$ first one) and write

\[c_{n,m} = \sum_{k=3}^d\int_{C_{d,k}} \binom{n}{n-k} \sum_{s[k], m[k+1]} \binom{n-k}{s_1,\cdots,s_k}\binom{m}{m_1,\cdots,m_{k+1}}|\CH(c[d])|^{m_{k+1}}\prod_{j=1}^k c_{s_i,m_i}^{\bullet,\bullet, t_i\cap \K} dc_1\cdots dc_k\]
where the $t_i=t_i(c[d])$ are defined as before, and what changes is that the $m_i$ have to be shared among the $t_i$ and also with $\CH(c[d])$.   

Again we can use the upper and lower bound strategy: for the upper bound replace  $t_i\cap \K$ by $t_i$ again, and for the lower bound, integrate on the $C_{d,k}$ for which $t_i$ is entirely inside $\K$. In both cases
\[ c_{s_i,m_i}^{\bullet,\bullet, t_i\cap \K}=|t_i|^{s_i+m_i} \bQt_{s_i,m_i}\]
and then we get to the first big difference with B\'ar\'any's case. We don't have no simple formula for $\bQt_{s_i,m_i}$, only an asymptotic formula for $\bQt_{n,\floor{n\lambda}}$. It is not difficult however to prove that if $x(n)/n\to \lambda\in [0,+\infty)$ then
\[\bQt_{n,x(n)}=\bQt_{n,\floor{n\lambda}}\exp( n o(1))=\exp(-2n\log(n)+n\beta_\lambda+no(1))\]
by adapting the proofs we presented for the asymptotics of $\bQt_{n,\floor{n\lambda}}$, however, we have not found a complete argument to prove that the optimization of 
\[F_{c[d]}:=\binom{n}{n-k} \sum_{s[k], m[k+1]} \binom{n-k}{s_1,\cdots,s_k}\binom{m}{m_1,\cdots,m_{k+1}}|\CH(c[d])|^{m_{k+1}}\prod_{j=1}^k |t_i|^{s_i+m_i} \bQt_{s_i,m_i}\]
is equivalent (up to an $\exp(o(n))$ factor) to that of
\[\bar F_{c[d]}= \max_{s[k],m[k+1]} \binom{n-k}{s_1,\cdots,s_k}\binom{m}{m_1,\cdots,m_{k+1}}|\CH(c[d])|^{m_{k+1}}\prod_{j=1}^k |t_i|^{s_i+m_i} \exp(-2s_i\log(s_i)+s_i \beta_{m_i/s_i})\]
even if it is very likely the case (again we have suppressed the factor $\binom{n}{n-k}$ for it is $\exp(o(n))$). A sort of ``uniform approximation theorem'' could suffice to complete this step, but we face another problem, which is that $m_i/s_i$ is not bounded above. It is unlikely that in a given triangle $t_i$ the ratio $m_i/n_i$ becomes very large, but in the end rare events may be those that contribute the most to the weights.

Now, take $(z[n],w[m])$ under $\QK_{n,m}$, so that $z[n]$ is the set of vertices of $\CH(\{z_1,\cdots,z_n,w_1,\cdots,w_m\})$. If $z[n]$ is fixed, then the $w_i$ are uniform and independent in $\CH(\{z_1,\cdots,z_n\})$. A small picture allows to see that when $d$ becomes large, the $t_i$ become very small, and $\CH(c[d])$ become close to $\CH(\{z_1,\cdots,z_n\})$. Moreover, the sum $\max_{c[d]} \sum_{i=1}^d |t_i|$ goes to zero with $d$, because $\max_{c[d]}\sum |t_i|^{1/3}$ is bounded. Therefore, the total number of points $w_i$ that are likely to be in the triangles becomes negligible compared to $n$, as $d\to+\infty$.

We may then conjecture that for any $\epsilon>0$, if $d$ large enough, for $(s[d],m[d])$ maximizing $f_{c[d]}$, $\max m_i/s_i\leq \epsilon$. If all of this is right, and if we can now use a series expansions of $\beta_{m_i/n_i}$ close to zero to pursue the optimization scheme, then the conjecture construction goes on as follows.

We need an expansion of $\beta_\lambda$. First, from $\sinh(2\rl)/(2\rl)=\lambda+1$, we can prove that near zero, $\rl = \sqrt{3\lambda/2} + O(\lambda^{3/2})$, from what we see that
\[\beta_\lambda = 2+\log(2)+(1+\log(2/3)+\log(1/\lambda))\lambda + o(\lambda),\]
so that, near zero, the main term is $2+\log(2)$.  

We may then conjecture that replacing $\beta_{m_i/s_i}$ by $2+\log(2)$, the order of $\bar F_{c[d]}$ would be unchanged. We may then again make a change of variable, replace $s_i$ by $n\alpha_i$ and conjecture, that the optimization is within a factor $\exp(o(n))$ the same as
\[\bar{\bar F}_{c[d]}=  e^{n\log n-n} |\CH(c[d])|^{n\lambda} \max_{s[k],m[k+1]}\binom{m}{m_1,\cdots,m_{k+1}}\prod_{i=1}^k |t_i|^{s_i} \l(\frac{|t_i|}{|\CH(c[d])|}\r)^{m_i} e^{-3s_i\log(s_i)+s_i (3+\log(2))}\]
where we have used the formula $b!\sim_{\exp} b^b/e^b$ for $n!$ and $s_i!$, and $s_i \beta_{m_i/s_i}$ replaced by $\exp(-2s_i\log(s_i)+s_i (2+\log(2)))$. The sum over $(m_i)$ of $\sum_{(m_i)}\binom{m}{m_1,\cdots,m_{k+1}}  \prod_{i=1}^k (\frac{|t_i|}{|\CH(c[d])|})^{m_i}=(1+\sum |t_i|/|\CH(c[d])|)^{m}$  and since we expect that $\sum |t_i|$ goes to 0 with $d$, this term is eventually sub-exponential in $m$. We then go on to conjecture that the optimization problem is the same as
\[\bar{\bar{\bar F}}_{c[d]}=  e^{n\log n-n+(3+\log(2))n} |\CH(c[d])|^{n\lambda} \max_{s[k] } \prod_{i=1}^k |t_i|^{s_i}  \exp(-3s_i\log(s_i))\]
where we have used the formula $b!\sim_{\exp} b^b/e^b$ for $n!$ and the $s_i$.
By Lagrange multiplier method, we find again that $s_i = \frac{2n|t_i|^{1/3}}{S_d}$ (where again $S_d=\sum 2|t_i|^{1/3}$) maximizes this quantity, which gives finally
\be
\bar{\bar{\bar F}}_{c[d]}&=&  e^{n\log n+(2+\log(2))n+o(n)} |\CH(c[d])|^{n\lambda} \exp\l(\sum_{i=1}^k \frac{2n|t_i|^{1/3}}{S_d}\log(|t_i|) - 3\frac{2n|t_i|^{1/3}}{S_d}\log\l(\frac{2n|t_i|^{1/3}}{S_d}\r)   \r)\\
&=& e^{n\log n+(2+\log(2))n+o(n)} |\CH(c[d])|^{n\lambda} \exp\l(- \sum_{i=1}^k3\frac{2n|t_i|^{1/3}}{S_d}\log\l(\frac{2n }{S_d}\r)   \r)\\
&=& e^{n\log n+(2+\log(2))n+o(n)} |\CH(c[d])|^{n\lambda} \exp\l(-  3 {n } \log\l({2n }/{S_d}\r)\r)   \\
&=& e^{-2n\log n+(2+\log(2))n+o(n)}   \exp\l(-  3 n \log(2 )+3n\log(S_d^3 |\CH(c[d])|^\lambda \r)   
\ee
We can then see that a limit shape around the convex set $C^\star$ maximizing $\PLK(C)$ appears\footnote{but to get this conclusion we would need to use also that $\argmax \PLK$ is reduced to a single element, which is clearly true for some $\K$, but we are not totally convinced that this is true for all $\K$. In any case, the following reasoning applies only in the case where $\argmax \PLK$ is reduced to a single element}, and taking the limit over $d$ to get the asymptotic behavior we get the conjecture 
\[c_{n,\fnl}^\K=\exp\l(-2n\log (n)+(2-{2}\log(2))n +n \log( \Phi_\K(C^\star))+    o(n)\r)\]
therefore since $\bQK_{n,\fnl}=\binom{n+\fnl}{n} c_{n,\fnl}^K$, we get using
\[\binom{n+\fnl}{n} = \exp(o(n)) \exp(n(\lambda+1)\log(\lambda+1)-n\lambda\log(\lambda))\]

the second conjecture

\[n^2\left(\bQK_{n,\fnl}\right)^{1/n} \cvg \frac{e^2}{4}\cdot\frac{(\lambda+1)^{\lambda+1}}{\lambda^\lambda}\cdot\Phi_\K(C^\star). \] 
Then, using the rest of the reasoning presented in the $\bQK_{n,0}$ case, we may expect that  under $\QK_{n,\fnl}$, $\CH(U[n+\fnl])\proba C^\star$ for the Hausdorff topology.

\section{Appendix}

\subsection{Proof of \Cref{theo:mon}}

Proof of $(i)$.
We will see that we only need to prove the following trivial lemma:
\begin{lem}
  If $m(n)=o(n)$, for all $\eta>0$,
  \ben \Qt_{n,m(n)}\Bigl( \sup_{t\in[0,1]} n^{-1}\bS_{nt}\geq \eta\Bigl)= 0.
  \een 
\end{lem}
\begin{proof} $ \sup_{t\in[0,1]} n^{-1}\bS_{nt} \leq \bS_n/n=m(n)/n \to 0$. \end{proof}
Now to prove the theorem, it suffices to adapt the proof of \Cref{theo:lim} from 
\Cref{lem:jolicv}, which has to be replaced by
\begin{lem}\label{lem:jolicv2} On $(\Omega,{\cal A},\P)$, for $m(n)=o(n)$,  for all $t\in[0,1]$,
\be
\E(2-(X_{\floor{(n+1)t}}+Y_{\floor{(n+1)t}})~|~\bS^{(n)}[n]) &\xrightarrow[n\to+\infty]{a.s.}& Q(t)= 2t^2  \\
\Var(2-(X_{\floor{(n+1)t}}+Y_{\floor{(n+1)t}})~|~\bS^{(n)}[n]) &\xrightarrow[n\to+\infty]{a.s.}0.
\ee
\end{lem}
\begin{proof} The same proof as that of \Cref{lem:jolicv} applies except that the limits have to be changed.
Solving the system $2-X_t-Y_t=Q(t)$ and $X_t-Y_t=Q(1-t)$ we get $X_t=2-2t$ and $Y_t=2t(1-t)$, in which we recognize the parametrization of ${\cal P}$ in $ABC$.
\end{proof}

$(ii)$ The proof can be adapted simply from \Cref{lem:sto}, which is easily seen to be valid in the bi-pointed case (that is $\E({\sf Area}^{\triangle \bullet\bullet}_{n,m+1})\geq \E({\sf Area}^{\triangle \bullet\bullet}_{n,m})$). Then, we may deduce from this that for $m(n)/n\to+\infty$, for any fixed $\lambda$, for $n$ large enough, $\E({\sf Area}^{\triangle \bullet\bullet}_{n,m(n)})\geq\E({\sf Area}^{\triangle \bullet\bullet}_{n,\fnl})\to \A({\cal H}_\la)$. This implies that  $\lim\inf\E({\sf Area}^{\triangle \bullet\bullet}_{n,m(n)})\geq \sup_{\lambda}  \A({\cal H}_\la)=1$. The argument given in the proof of \Cref{theo:compa} allows to conclude.
\color{black}

\subsection{Proof of a central local limit theorem under $d^{(n)}$}

The local limit theorem we state below is needed in the proof of \Cref{lem:LB}.
We use the notation of the proof of this Lemma, in particular we take independent random variables $(\obK_i,1\leq i \leq n)$ where $d_i^{(n)}$ is defined in \eref{eq:remonte2}. 


First, let us write as an index $d^{(n)}$ the computation related to the independent random variables $\obK_1,\cdots, \obK_n$, and where $\obK_i$ is $\mu_{d_i^{(n)}}$ distributed. We set 
\be  {M}_{j} &=&\E_{d^{(n)}}(\obS_j)= \sum_{i=1}^j \E_{d^{(n)}} \l(\obK_i\r)=2 \sum_{i=1}^j  {d_i^{(n)}}~/~\l({1-d_i^{(n)}}\r),\\
B_j&=&\Var_{d^{(n)}}(\obS_j  )= 2\sum_{i=1}^j  {d_i^{(n)}}~/~{\l(1-d_i^{(n)}\r)^2}.
\ee 
\begin{lem}\label{lem:CVD}Let $\lambda>0$, we have:
    \beq \sup_{t\in[0,1]} \l| d_{\floor{nt}}^{(n)} - \tanh(\rl t)^2  \r|\to 0.\eq
    The two following convergences hold uniformly on $[0,1]$:
    \ben M_{nt}/n &\to& M_t= -t+\frac{\sinh(2\rl t)}{2\rl}  \\
    B_{nt}/n &\to& B_t  =-\frac{t}{4}+\frac{\sinh(4\rl t)}{16\rl}\een
    \end{lem}
\begin{proof}
Take some $A\in(0,1]$. Let us prove the uniform convergence of $d_{nt}^{(n)}$ on $[A,1]$.
To prove the point-wise convergence $d_{nt}^{(n)}\to \tanh(\rl t)^2$, make first the change of variable $j=nu$, so that 
\[d_{nt}^{(n)}=\tanh(\rl)^2 \exp\l(- \int_{t}^1  \frac{1}{\floor{nu}/n+D_{\floor{nu}}^{(n)}/n}+\frac{1}{\floor{nu}/n+D_{\floor{nu}}^{(n)}/n+1/n}\,du\r),\] then observe that $ \frac{1}{\floor{nu}/n+D_{\floor{nu}}^{(n)}/n}+\frac{1}{\floor{nu}/n+D_{\floor{nu}}^{(n)}/n+1/n}\to \frac{4\rl}{\sinh(2\rl u)}$ point-wise, and for the domination, use that $D_{\floor{nu}}^{(n)}/n\geq 0$ to get the bound, valid for all $u\in[t,1]$ (and also for all $t \in[A,1]$:
\[\frac{1}{\floor{nu}/n+D_{\floor{nu}}^{(n)}/n}+\frac{1}{\floor{nu}/n+D_{\floor{nu}}^{(n)}/n+1/n}\leq C/(\floor{nu}/n)\leq C'/u\] for $u\in [t,1]$.\par
Now, to get the uniformity, it suffices to observe that
\[\max_{u\in[A,1]} \l|\frac{1}{\floor{nu}/n+D_{\floor{nu}}^{(n)}/n}+\frac{1}{\floor{nu}/n+D_{\floor{nu}}^{(n)}/n+1/n}-\frac{2}{u+D_u} \r|\to 0. \]
It remains to prove the uniform convergence on $[0,A]$ for a $A$ of our choice: take $\varepsilon>0$ small, and take $A$ such that, $\sup_{t\in[0,A]} \tanh(\rl t)^2\leq \varepsilon/2$. Since ${\floor{nu}/n+D_{\floor{nu}}^{(n)}/n}=\sinh(2\floor{nu}\rl/n)/(2\rl)\leq \sinh(2\rl u)/(2\rl)$, from what we see that $d_{nt}^{(n)}\leq \tanh(\rl t)^2$, and the conclusion follows. The two other statements are consequences:
$M_{nt}/n \to M_t=2\int_0^t \frac{\tanh(\rl u)^2}{1-\tanh(\rl u)^2}du$ and $B_{nt}\to 2\int_0^t \frac{\tanh(\rl u)^2}{(1-\tanh(\rl u)^2)^2}$ uniformly on $[0,1]$. And a simplification of these formulas provide the announced results.
\end{proof}

We are ready to state a local limit theorem for $\obS_n$ under $\P_{d^{(n)}}$:
 \begin{lem}\label{lem:clt} Let $\lambda>0$.
We have
\[\sup_N \l|\sqrt{B_n}\P_{d^{(n)}}\l(\obS_n=N\r)-\frac{1}{\sqrt{2\pi}}\exp\l(-\frac{(N-M_n)^2}{2B_n}\r)\r|\to 0,\]
so that, in particular,
\beq\P_{d^{(n)}}\l(\obS_n=M_n\r) \equ  C_{\lambda}/\sqrt{n},\eq
for a positive finite constant $C_\lambda$.
\end{lem}
\begin{proof} 
We will use the central local limit given by Davis \& McDonald \cite[Theo. 1.2]{Davis_McDonald}.  
Set $q_{i,n}= \sum_k \P_{d^{(n)}}(\obK_i = k) \wedge \P_{d^{(n)}}(\obK_i = k+1)$ and $Q_n=\sum_{i=1}^n q_{i,n}$.
We will prove that $(a)$, $(\obS_n-M_n)/\sqrt{B_n}\dd {\cal N}(0,1)$, and $(b)$, $\limsup B_n/Q_n <+\infty$. 

$(a)$ We will use the central limit theorem under the Lindeberg condition. 
Since $B_n\equ C\cdot n$, it suffices to prove that  
\[G_n:=\frac{1}{n}\sum_{k=1}^n\E_{d^{(n)}}\l(|\obK_i-\E(\obK_i)|^2\1_{|\obK_i-\E(\obK_i)|\geq \eps n}\r)\to 0.\]
Under the condition $m(n)=\floor{n \lambda}$, the sequence $\max_n \max_{1\leq i\leq n} \E_{d^{(n)}}(\obK_i) $ is bounded, so that it suffices to prove that 
\[G_n':=\frac{1}{n}\sum_{k=1}^n\P_{d^{(n)}}\l( \obK_i \geq \eps n/2\r)\to 0, ~~~ G_n'':=\frac{1}{n}\sum_{k=1}^n\E_{d^{(n)}}\l(\obK_i^2 \1_{\vert\obK_i\vert \geq \eps n /2}\r)\to 0\]
We have $\mu_v([\ell,+\infty)=(1-v)^2\sum_{k\geq \ell} (1+k)v^k = v^\ell(1+(1-v)\ell)$.
Hence, since $\sup_i|d_i^{(n)}-\tanh(\rl i/n)^2|\to 0$ and $\max\{d_i^{(n)}, 1\leq i\leq n\}\leq \gamma\to\tanh(\rl)^2<1$, gives
 \[G_n'\leq \frac{ \gamma^{\eps n}}{n}\sum_{i=1}^n (1+(1-d_i^{(n)})\eps n)=O( n\gamma^{\eps n} )\to 0.\]
To control $G_n''$, compute for $X_v\sim \mu_v$,
$\E(X_v^2 \1_{X_v\geq \ell})=\sum_{k\geq \ell} (1-v)^2 (1+k)v^k k^2$ 
so that since for $i\in\cro{1,n}$, $d_i^{(n)}\leq \gamma<1$
\[G_n'' \leq \frac1{n}\sum_{i=1}^n   \sum_{k\geq \eps n} (1+k)^3 \gamma^{k}=  \sum_{k\geq \eps n} (1+k)^3 \gamma^{k}\]
which goes to 0 as $n\to +\infty$, since $\sum_{k } (1+k)^3 \gamma^{k}$ converges.\\
$(b)$ If we have $\limsup B_n/Q_n' <+\infty$ for a sequence $(Q_n')$ such that $Q_n'\leq Q_n$, then $(b)$ holds. 
We can then lower-bound $Q_n$. 
Recall \eref{eq:muv}. The parameter $v$ belongs to $[0,1]$.

We have $\mu_v(k) \wedge \mu_v(k+1)= (1-v)^2v^k \min \{1+k,(2+k)v\}=(1-v)^2v^k\big( (1+k)v+ \min \{(1+k)(1-v),v\}\big)$, so that 
\[\sum_k \mu_v(k) \wedge \mu_v(k+1) \geq \sum_k v \mu_k(v) =v.\]
Hence, $Q_n'= \sum_{i=1}^n d_i^{(n)}$ is then a lower-bound of $Q_n$. 
By uniform convergence of $d_i^{(n)}$ to $\tanh\l(\rl \frac{i}{n}\r)^2$, when $n$ is large enough, $|Q_n'-\sum_{i=1}^n \tanh\l(\rl \frac{i}{n}\r)^2|=o(n)$. Since $\sum_{i=1}^n \tanh\l(\rl \frac{i}{n}\r)^2 \equ n\int_0^1 \tanh({\rl t})^2 dt$, showing $\limsup B_n/Q_n' <+\infty$ is easy and therefore $(b)$ holds.
\end{proof}

\bibliographystyle{abbrv}
\bibliography{biblio.bib}

\tableofcontents

\end{document}